\documentclass{amsart}
\usepackage[foot]{amsaddr}
\usepackage[margin=1.5in]{geometry}

\usepackage{amsthm,amssymb,amsfonts,amsmath}
\usepackage{stmaryrd}
\usepackage{bbm} 
\usepackage[numbers, sort&compress]{natbib}
\usepackage{graphicx}
\usepackage{color}
\usepackage{enumerate}
\usepackage{mathtools}
\usepackage{ulem}
\usepackage{enumitem}
\usepackage{tikz}
\usepackage[symbol]{footmisc}
\usetikzlibrary{arrows.meta, positioning}
\usepackage{hyperref}
\usepackage[noabbrev,capitalise]{cleveref}

\providecommand{\R}{}

\renewcommand{\R}{\mathbb{R}}

\newcommand{\set}[1]{\left\{ #1 \right\}}

\newcommand{\Esub}[2]{{\mathbf E_{#1}}\left[#2\right]}

\newcommand\cB{\mathcal B}
\newcommand\cC{\mathcal C}
\newcommand\cD{\mathcal D}
\newcommand\cE{\mathcal E}
\newcommand\cF{\mathcal F}

\newcommand\cP{\mathcal P}

\newcommand\cS{{\mathcal S}}

\newcommand\cU{{\mathcal U}}

\newcommand{\bE}{\mathbf{E}}

\newcommand{\bP}{\mathbf{P}} 
\newcommand{\bQ}{\mathbf{Q}} 
\newcommand{\bR}{\mathbf{R}}

\newcommand{\f}{\mathfrak{f}} 
\newcommand{\g}{\mathfrak{g}} 
 
\renewcommand{\c}{\mathfrak{c}}   
\renewcommand{\a}{1}
\renewcommand{\d}{2}
\newcommand{\dbleft}{[\![}
\newcommand{\dbright}{]\!]}
\newcommand{\db}[1]{\left\dbleft #1 \right\dbright}

\newcommand{\rS}{\cS} 
\newcommand{\leb}{\mathfrak{m}} 

\newcommand{\pran}[1]{\left(#1\right)}
\newcommand{\brac}[1]{\left[#1\right]}
\newcommand{\inn}[1]{\langle#1\rangle}
\newcommand{\abs}[1]{\left|#1\right|}
\newcommand{\norm}[1]{\left\|#1\right\|}
\newcommand{\KL}[2]{\mathrm{KL}(#1||#2)}

\providecommand{\ora}[1]{}
\renewcommand{\ora}[1]{\overrightarrow{#1}}

\newcommand\one{\mathbbm{1}}
\makeatletter
\newcommand{\customlabel}[2]{%
   \protected@write \@auxout {}{\string \newlabel {#1}{{#2}{\thepage}{#2}{#1}{}} }%
   \hypertarget{#1}{}
}
\makeatother

\newtheorem{thm}{Theorem}
\newtheorem{lem}[thm]{Lemma}
\newtheorem{prop}[thm]{Proposition}
\newtheorem{cor}[thm]{Corollary}
\newtheorem{dfn}[thm]{Definition}

\AddToHook{env/prop/begin}{\crefalias{thm}{prop}}
\AddToHook{env/lem/begin}{\crefalias{thm}{lem}}
\AddToHook{env/cor/begin}{\crefalias{thm}{cor}}
\AddToHook{env/dfn/begin}{\crefalias{thm}{dfn}}

\numberwithin{thm}{subsection}
\numberwithin{equation}{subsection}
\definecolor{clc}{rgb}{0,0.53,0.74}
\definecolor{caz}{rgb}{0,0.5,0}

\begin{document}

\title{The Schr\"odinger Bridge Problem for Jump Diffusions with Regime Switching}
\author{Andrei Zlotchevski}
\address{Department of Mathematics and Statistics, McGill University, 
Montreal, Canada}
\thanks{We acknowledge the support of the Natural Sciences and Engineering Research Council of Canada (NSERC), awards 559387-2021 and 241023.}
\email{andrei.zlotchevski@mail.mcgill.ca}
\author{Linan Chen}
\email{linan.chen@mcgill.ca}

\begin{abstract}
    The Schr\"odinger bridge problem (SBP) aims at finding the measure $\widehat\bP$ on a certain path space which possesses the desired state-space distributions $\rho_0$ at time $0$ and $\rho_T$ at time $T$ while minimizing the KL divergence from a reference path measure $\bR$. This work focuses on the SBP in the case when $\bR$ is the path measure of a jump diffusion with regime switching, which is a Markov process that combines the dynamics of a jump diffusion with interspersed discrete events representing changing environmental states. To the best of our knowledge, the SBP in such a setting has not been previously studied. 
    
    In this paper, we conduct a comprehensive analysis of the dynamics of the SBP solution $\widehat\bP$ in the regime-switching jump-diffusion setting. In particular, we show that $\widehat\bP$ is again a path measure of a regime-switching jump diffusion; under proper assumptions, we establish various properties of $\widehat\bP$ from both a stochastic calculus perspective and an analytic viewpoint. In addition, as a demonstration of the general theory developed in this work, we examine a concrete unbalanced SBP (uSBP) from the angle of a regime-switching SBP, where we also obtain novel results in the realm of uSBP.
\end{abstract}
\keywords{Schr\"odinger bridge, regime-switching jump diffusion, stochastic control, KL divergence, Sinkhorn algorithm, unbalanced Schr\"odinger bridge, jump diffusion with killing}
\subjclass{35Q93, 45K05, 49J20, 60H10, 60H30, 60J25, 93E20, 94A17}

\maketitle
\section{Introduction}
Our topic of interest is the Schr\"odinger bridge problem (SBP) for regime-switching jump diffusions. As an extension of the framework of the jump-diffusion SBP \cite{zlotchevski2024schrodinger}, the formulation adopted in this work incorporates additional layers of complexity in order to model real-world phenomena present under changing environmental states. To begin, jump diffusions themselves generalize classical diffusions by integrating jumps driven by L\'evy measures, thereby capturing sudden events such as market crashes or neuron firings. Regime switching further enhances this by having the drift, diffusion, and jump coefficients vary according to a Markov chain on a finite set of states called \textit{regimes}, which represents sudden environmental changes. We are interested in studying the SBP in such a setting. Even though the general study of the SBP has a long and rich history, the regime-switching SBP has not seen much development in the literature. This work is dedicated to introducing a general theory on the SBP for regime-switching jump diffusions.  

\subsection{Set-up of the problem}\label{Subsection: setup} 

Our model is set on a ``hybrid'' state space $\R^d\times \rS$ over a finite time interval $t\in[0,T]$, where $\rS$ is a finite set consisting of all the regimes. Without loss of generality, we will assume that $\rS=\set{1,2,\dots,\abs{\rS}}$. Let $\Omega=D([0,T];\R^d \times \rS)$ be the Skorokhod space, i.e., the space of c\`adl\`ag functions $\omega : t\in[0,T] \mapsto(\omega^{\R^d}(t),\omega^{\rS}(t))\in \R^d\times\rS$, equipped with the Skorokhod metric. For every $t\in[0,T]$, define $(X_t,\Lambda_t):\Omega\rightarrow\R^d\times\rS$ to be the time projection such that, for every $\omega\in\Omega$, $X_t(\omega)=\omega^{\R^d}(t)$ and $\Lambda_t(\omega)=\omega^{\rS}(t)$. If $\cF:=\cB_\Omega$ is the Borel $\sigma$-algebra on $\Omega$, then $\cF$ is generated by the time projections $\set{(X_t,\Lambda_t)}_{[0,T]}$. Let $\set{\cF_t}_{[0,T]}$ be the natural filtration (or its usual augmentation) associated with $\set{(X_t,\Lambda_t)}_{[0,T]}$. Denote by $\cP(\Omega)$ the set of probability measures on $(\Omega,\cF)$. For $\bP\in\cP(\Omega)$, to which we refer as a \textit{path measure}, we say that $\set{(X_t,\Lambda_t)}_{[0,T]}(\omega)$ is the \textit{canonical process} under $\bP$ if $\omega\in\Omega$ is sampled under $\bP$.  Finally, for $0\leq t \leq T$, we write the distribution of $(X_t,\Lambda_t)$ under $\bP$ as $\bP_t$, and call it the \textit{marginal distribution} at time $t$.\\ 

Now we present the rigorous definition of the \textit{Schr\"odinger Bridge Problem} (SBP). 
\begin{dfn}
    Given a path measure $\bR\in\cP(\Omega)$ and two probability distributions $\rho_0,\rho_T$ on $\R^d \times \rS$, the \textup{Schr\"odinger bridge problem} is to find the path measure $\widehat\bP\in\cP(\Omega)$ such that
\begin{equation}\label{Definition:DynamicSBP}
    \widehat\bP=\arg\min\set{\KL{\bP}{\bR} \text{ such that } \bP_0=\rho_0,\bP_T=\rho_T }.
\end{equation}
\end{dfn}
\noindent Here, $\KL{\cdot}{\cdot}$ denotes the \textit{KL divergence} (relative entropy), defined as \[\KL{\bP}{\bR}:=  \Esub{\bP}{\log\pran{\frac{d\bP}{d\bR}}}\;\textup{ if }\bP\ll\bR,\;+\infty\textup{ otherwise.}\]
In the above SBP, we refer to $\bR$ as the \textit{reference measure}, $\rho_0,\rho_T$ as the \textit{target distributions}, and the solution $\widehat\bP$, if it exists, as the \textit{Schr\"odinger bridge}. Heuristically speaking, the SBP is the quest to find the path ``bridging'' $\rho_0$ and $\rho_T$ that is the ``closest'' to a given reference. 

In this paper, we consider the SBP with the reference measure $\bR$ arising from a \textit{regime-switching jump diffusion}, which is a two-component strong Markov\footnote[2]{Throughout the paper, all the occasions of ``(strong) Markov process'', ``(strong) Markov property'' and ``martingale'' are with respect to the natural filtration associated with the concerned process.} process $\set{(X_t,\Lambda_t)}_{[0,T]}$ (defined on some generic probability space) with $(\R^d\times\rS)$-valued c\`adl\`ag sample paths. Below we will explain each component of this process. 

We start with the $\R^d$-valued spatial component $X_t$. For every value of $\Lambda_t$, we take $\set{X_t}_{[0,T]}$ to be a jump-diffusion process governed by the following stochastic differential equation (SDE):
\begin{equation}\label{SDE for Xt without hybrid jumps}
 \begin{aligned}   dX_t&= b(t,X_{t},\Lambda_t)dt +\sigma(t,X_{t},\Lambda_t)dB_t\\ 
 &\hspace{0.5cm}+\int_{|z|\leq 1}\gamma(t,X_{t-},\Lambda_{t-} ,z)\tilde{N}(dt,dz)
+\int_{|z|>1}\gamma(t,X_{t-},\Lambda_{t-} ,z)N(dt,dz),
\end{aligned}
\end{equation}
where $b,\sigma,\gamma$ are coefficient functions satisfying certain conditions (which will be made explicit in Assumption \ref{Assumptions C}), $\set{B_t}_{t\geq0}$ is a standard $d$-dimensional Brownian motion and $N(dt,dz)$ is an $\ell$-dimensional Poisson random measure with intensity $\nu(dz)$, with $\tilde{N}(dt,dz):=N(dt,dz)-\nu(dz)dt$ denoting its compensated counterpart.

Next, we turn to the $\rS$-valued regime component $\Lambda_t$ and suppose that $\set{\Lambda_t}_{[0,T]}$ is a continuous-time discrete Markov chain such that, given the time $t\in[0,T]$ and the spatial component position $X_t=x\in\R^d$, $\Lambda_t$ has the transition rate matrix $Q(t,x)=(Q_{ij}(t,x))_{i,j\in\rS}$.
Intuitively, we wish to impose a regime-switching mechanism such that for every $i,j\in\rS$, as $h\searrow0$,
\[
\mathbb{P}(\Lambda_{t+h}= j | (X_t,\Lambda_t) = (x,i)) = \delta_{ij}+ Q_{ij}(t,x)h+o(h),
\]
where $\delta_{ij}$ is the Kronecker delta.
To this end, following a standard construction (e.g. \cite{xi2017feller,kunwai2020feller}), for every $(t,x)\in[0,T]\times\R^d$, we choose and fix $\set{\Delta_{ij}(t,x)}_{i,j \in \rS}$, which is a collection of subsets of $\R_+$ such that $\Delta_{ii}(t,x)=\emptyset$ for every $i\in\rS$, and for all $i\neq j$, $\Delta_{ij}(t,x)$'s are disjoint intervals and the length of every $\Delta_{ij}(t,x)$ is equal to $Q_{ij}(t,x)$. Further, define \[\alpha(t,x,i,w):=\sum_{j \in \rS} (j-i)\one_{\Delta_{ij}(t,x)}(w)\textup{ for }(t,x,i,w)\in[0,T]\times\R^d\times\rS\times\R_+.\]Then,  the desired regime-switching process $\set{\Lambda_t}_{[0,T]}$ can be characterized as the solution to 
\begin{equation}\label{Integral with alpha, N1}
d\Lambda_t=\int_{\R_+}\alpha(t,X_{t-},\Lambda_{t-},w)N_1(dt,dw),
\end{equation}
where $N_1(dt,dw)$ is a Poisson random measure on $\R_+$ with intensity measure $\leb(dw)$ being the Lebesgue measure. Note that the stochastic integral in \eqref{Integral with alpha, N1} is well defined because $w\mapsto\alpha(t,x,i,w)$ is compactly supported for every $(t,x,i)\in[0,T]\times\R^d\times\rS$.

Finally, we also incorporate \textit{hybrid jumps} \cite{blom2024feller} in our regime-switching mechanism. Namely, given a set of measurable functions $\psi:=\set{\psi_{ij}:[0,T]\times\R^d\rightarrow\R^d}_{i,j\in\rS}$, we assume that, upon switching from regime $i$ to regime $j$ at time $t$, the spatial component of the process jumps from $X_{t-}$ to a new position given by $\psi_{ij}(t,X_{t-})$. Consequently, the SDE \eqref{SDE for Xt without hybrid jumps} will have an additional term $\int_{\R_+}\beta(t,X_{t-},\Lambda_{t-},w)N_1(dt,dw)$ 
where \[\beta(t,x,i,w):=\sum_{j\in\rS} (\psi_{ij}(t,x)-x)\one_{\Delta_{ij}(t,x)}(w)\textup{ for }(t,x,i,w)\in[0,T]\times\R^d\times\rS\times\R_+.\] 

Putting everything together, we obtain the SDE formulation of the regime-switching jump diffusion $\set{(X_t,\Lambda_t)}_{[0,T]}$ as:
\begin{equation}\label{Reference SDE general theory}
\left\{
    \begin{aligned}
        dX_t&=b(t,X_{t},\Lambda_t)dt+\sigma(t,X_{t},\Lambda_t)dB_t + \int_{|z|\leq 1}\gamma(t,X_{t-},\Lambda_{t-} ,z)\tilde{N}(dt,dz)\\
        &\hspace{1cm} +\int_{|z|>1}\gamma(t,X_{t-},\Lambda_{t-} ,z)N(dt,dz) + \int_{\R_+}\beta(t,X_{t-},\Lambda_{t-},w)N_1(dt,dw), \\
    d\Lambda_t &= \int_{\R_+}\alpha(t,X_{t-},\Lambda_{t-},w)N_1(dt,dw).     
    \end{aligned}
\right.
\end{equation}
To ensure the existence of all the  stochastic integrals, we impose the following conditions:
\paragraph{\textbf{Assumption (C)}}\customlabel{Assumptions C}{\textbf{(C)}}
\begin{enumerate}
    \item 
    $b=(b_m)_{1\leq m\leq d}:[0,T]\times\R^d \times \rS \to \R^d$ is measurable and locally bounded.
    \item 
    $\sigma=(\sigma_{mn})_{1\leq m,n\leq  d} : [0,T]\times\R^d\times \rS \to\R^{d\times d}$ is measurable and locally bounded, and $(\sigma\sigma^\top)(t,x,i):=a(t,x,i)$ is non-negative definite for every $(t,x,i)\in[0,T]\times\R^d\times\rS$.
    \item 
    $\gamma=(\gamma_m)_{1\leq m\leq d} :[0,T]\times\R^d\times \rS\times\R^\ell \to \R^d$ is measurable, and as a function in $(t,x,i)\in[0,T]\times\R^d\times\rS$,  $\gamma(t,x,i,0)\equiv 0$ and $\sup_{z\in \mathbb{R}^\ell}\frac{|\gamma(t,x,i,z)|}{1\wedge|z|}$ is locally bounded.
    \item The intensity $\nu(dz)$ is a L\'evy measure, i.e., $\nu(\set{0})=0$ and  $\int_{\R^\ell} (1 \wedge |z|^2) \nu(dz) <\infty$.
    \item For every $i,j\in\rS$, $\psi_{ij}:[0,T]\times\R^d\rightarrow\R^d$ is measurable and locally bounded, and $\psi_{ii}(t,x)=x$ for every $(t,x)\in[0,T]\times\R^d$.
    \item For every $i,j\in\rS$ and $i\neq j$, $Q_{ij}:[0,T]\times\R^d\rightarrow\R_+$ is measurable and locally bounded. 
\end{enumerate}

\noindent While Assumption \ref{Assumptions C} will be imposed throughout the work, they are mild conditions on the coefficients only to allow the SDE \eqref{Reference SDE general theory} to be well defined (see, e.g., \cite[Section 4.2.2]{applebaum2009levy}). To conduct analytic work on the SDE and obtain any meaningful result, we will need various additional assumptions, which will become clear in later sections. \\

For the SBP studied in this work, we consider a reference measure $\bR\in \cP(\Omega)$ whose canonical process $\set{(X_t,\Lambda_t)}_{[0,T]}$ is a regime-switching jump diffusion in the sense described above, i.e., as a process on $(\Omega,\cF,\set{\cF_t}_{[0,T]},\bR)$, $\set{(X_t,\Lambda_t)}_{[0,T]}$ is a strong Markov process satisfying the SDE \eqref{Reference SDE general theory}. In other words, $\bR$ is a strong Markov path measure. For every $(t,x,i)\in[0,T]\times\R^d\times\rS$, we denote by $\bR^{t,x,i}$ the conditional distribution of $\bR$ conditioning on $(X_t,\Lambda_t)=(x,i)$. We refer to $\bR_0$ and $\bR_T$ as the \textit{initial data} and the \textit{terminal data} respectively, and write $\bR_{0T}$ for the joint distribution of $
\bR$ at the endpoints $t=0$ and $t=T$.  

\subsection{Static Aspect of the SBP}\label{subsection:static view of bP^}
In the previous part we introduced the SBP in the regime-switching jump-diffusion setting. In order to conduct a proper study of this problem, we require some basic conditions to guarantee that both the SBP \eqref{Definition:DynamicSBP} and the SDE \eqref{Reference SDE general theory} are ``solvable'' with solutions in desired forms. To this end, we adopt the following general assumptions for the rest of the work.

\paragraph{\textbf{Assumption (G)}}\customlabel{Assumptions A}{\textbf{(G)}}
\begin{enumerate}[label=(\textbf{G\arabic*})]
\item For every $(t,x,i)\in[0,T]\times\R^d\times\rS$, the SDE \eqref{Reference SDE general theory} admits a solution\footnote[2]{In this work only the notion of weak solution to SDE is concerned.} $\bP^{t,x,i}\in\cP(\Omega)$ where the canonical process $\set{(X_r,\Lambda_r)}_{[0,T]}$ under $\bP^{t,x,i}$ satisfies \eqref{Reference SDE general theory} for $r\in[t,T]$ with $(X_t,\Lambda_t)\equiv(x,i)$. Moreover, the family $\set{\bP^{t,x,i}}_{[0,T]\times\R^d\times\rS}$ is strong Markovian in the sense that for every stopping time $\tau\in[t,T]$, the conditional distribution of $\bP^{t,x,i}$ conditioning on $\cF_\tau$ is given by $\bP^{\tau,X_\tau,\Lambda_\tau}$.\\ \label{Assumption: SDE} 
\item In the SBP \eqref{Definition:DynamicSBP}, the reference measure $\bR$ is the solution to the SDE \eqref{Reference SDE general theory} with initial data $\bR_0$ (the existence of such an $\bR$ is guaranteed by \ref{Assumption: SDE}), and further, $\bR_{0T}\ll\bR_0\otimes\bR_T$. In addition, the target distributions $\rho_0,\rho_T$ are such that $\KL{\rho_0\otimes\rho_T}{\bR_{0T}}<\infty$.
\label{asm: A2 - leads to fg existence}\\
\end{enumerate}

It is known in the literature (e.g., \cite{leonard2014survey}) that under Assumption \textup{\ref{Assumptions A}}, the SBP \eqref{Definition:DynamicSBP} admits a unique solution $\widehat\bP\in\cP(\Omega)$ with $\KL{\widehat\bP}{\bR}<\infty$. Further, this Schr\"odinger bridge takes the form
\begin{equation}\label{eq:bP^ in terms of fg}
    \widehat\bP=\f(X_0,\Lambda_0)\g(X_T,\Lambda_T)\bR 
    \end{equation} 
    where $\f,\g :\R^d\times\rS \to \R_+$ are non-negative measurable functions solving the system below:
    \begin{equation}\label{Equation: fg Schrodinger System with lambda}
        \begin{cases}
        \vspace{0.2cm}\displaystyle \f(x,i)\Esub{\bR}{\g(X_T,\Lambda_T)|(X_0,\Lambda_0)=(x,i)} = \frac{d\rho_0}{d\bR_0}(x,i)& \text{ for $\bR_0$-a.e. $(x,i)$},\\
        \displaystyle \g(y,j)\Esub{\bR}{\f(X_0,\Lambda_0)|(X_T,\Lambda_T)=(y,j)} = \frac{d\rho_T}{d\bR_T}(y,j)& \text{ for $\bR_T$-a.e. $(y,j)$}.
        \end{cases}
    \end{equation} 

We refer to \eqref{Equation: fg Schrodinger System with lambda} as the \textit{static Schr\"odinger system} as $\f,\g$ are entirely determined by the joint endpoint distribution $\bR_{0T}$ and the target distributions $\rho_0,\rho_T$ \cite[Theorem 2.14]{leonard2014reciprocal}, without invoking the dynamics of the whole path $\set{(X_t,\Lambda_t)}_{[0,T]}$. Following a similar viewpoint, although $\widehat\bP$ is a path measure, the formulation \eqref{eq:bP^ in terms of fg} of $\widehat\bP$ only reflects the static aspect of this Schr\"odinger bridge. Therefore, our main focus in this work is to develop an accurate and comprehensive ``dynamic picture'' for $\widehat\bP$.

\subsection{Main Results}\label{subsection:main results} Below we give an overview of the main results obtained in this work. First, we examine various properties of the Schr\"odinger bridge $\widehat\bP$ in the general regime-switching jump-diffusion setting. Then, we put the general theory into practice by applying it to a concrete example in the context of the \textit{unbalanced SBP}.  Our results generally extend the existing literature on the diffusion (and jump-diffusion) SBP to the regime-switch setting. Our unbalanced SBP example is also a novel model in its own realm.

Throughout the work, Assumptions \ref{Assumptions C} and \ref{Assumptions A} will be in effect, as they provide the foundation of our work. Depending on the specific aspect of $\widehat\bP$ we study, we may invoke additional assumptions. In \cref{subsection: general background}, we will discuss the generality and the feasibility of our assumptions.

\subsection*{Schr\"odinger bridge SDE and generator}
\cref{Section: General Theory} is dedicated to an examination of $\widehat\bP$ from a stochastic analysis perspective, where a central piece that underlies all the results is the measurable and non-negative function 
\begin{equation}\label{Definition: general varphi}
    \varphi(t,x,i):=\bE_{\bR}\brac{\g(X_T,\Lambda_T)|(X_t,\Lambda_t)=(x,i)}\textup{ for }(t,x,i)\in[0,T]\times\R^d\times\rS,
\end{equation}
where $\bR$, $\g$, and $\set{(X_t,\Lambda_t)}_{[0,T]}$ are the same as in the previous subsection. By invoking Assumption \ref{Assumption: general harmonic varphi} (the \textit{harmonic} assumption), which endows $\varphi$ with desirable properties, we establish the SDE formulation of $\widehat\bP$. In particular, we discover in \textbf{\cref{Theorem: General P-hat}} that the canonical process under $\widehat\bP$ is again a regime-switching jump diffusion, satisfying an SDE in the same form as \eqref{Reference SDE general theory} but with new coefficients in terms of $\varphi$, from where we also obtain the explicit expression of the infinitesimal generator associated with $\widehat\bP$. Our main methodology is to adapt stochastic calculus tools for jump diffusions to the regime-switching model. In fact, we establish that $\widehat\bP$ can be viewed as a Girsanov transform from (a modification of) the reference measure $\bR$, where the ``change of measure'' can be made explicit in terms of $\varphi$, and it affects not only the drift and the jump component of the original SDE \eqref{Reference SDE general theory}, but also its regime-switching mechanism.

\subsection*{Stochastic control formulation of SBP}
In \cref{Subsection: stochastic control}, we take the aforementioned Girsanov transform approach further and develop a stochastic control formulation of our SBP. Upon strengthening Assumption \ref{Assumption: general harmonic varphi}, we prove in \textbf{\cref{Theorem : optimal controls}} that the specific Girsanov transform which yields $\widehat\bP$ can in fact be attained by solving a properly crafted \textit{Stochastic Control Problem}. This confirms that the rich connection between the SBP and the field of stochastic control extends to the regime-switching setting.

\subsection*{Schr\"odinger bridge PIDE and dynamics}
In \cref{Section: densities}, we adopt an analytic approach and study the dynamics of $\widehat\bP$ under Assumption \ref{Assumptions D}, which requires the reference measure $\bR$ to admit a transition density function. In such a setting, we show in \textbf{\cref{Theorem: phi*phi-hat product density}} the existence of a transition density function, as well as marginal density functions at all time, under the Schr\"odinger bridge $\widehat\bP$; all these density functions can be explicitly expressed in terms of the \textit{Schr\"odinger potentials} $(\varphi,\widehat\varphi)$, where $\varphi$ is as in \eqref{Definition: general varphi}, and $\widehat\varphi$ is the measurable and non-negative function given by \[
\widehat\varphi(s,y,j):=\bE_{\bR}[\f(X_0,\Lambda_0);(X_s,\Lambda_s)=(y,j)]\;\textup{ for }(s,y,j)\in[0,T]\times\R^d\times\rS.
\]
Moreover, upon imposing stronger regularity conditions on the transition density function (Assumption \ref{Assumptions D^}), we present in \textbf{\cref{Lemma: cL*p=0 for rsjd}} $\varphi$ and $\widehat\varphi$ as the respective solution to a pair of partial integro-differential equations (PIDEs), which in turn gives rise to the \textit{dynamic Schr\"odinger system} for our SBP in terms of $(\varphi,\widehat\varphi)$, as stated in \textbf{\cref{Theorem : dynamic schrodinger system - general}}. Finally, in \textbf{\cref{Theorem: forward equation for Phat density - general}} we derive the Kolmogorov forward equation, in the form of a PIDE, associated with $\widehat\bP$, which complements the backward generator of $\widehat\bP$ found in \cref{Subsection: P-hat SDE and generator}.

\subsection*{The Fortet-Sinkhorn algorithm}
Also under the presence of the transition density function, in \cref{Subsection: Algorithm} we formulate the \textit{Fortet-Sinkhorn algorithm} in the regime-switch setting. This celebrated algorithm enables one to solve the Schr\"odinger system iteratively in a computationally efficient manner \cite{chen2016entropic} to obtain $(\varphi,\widehat\varphi )$ and thus to determine the Schr\"odinger bridge $\widehat\bP$. In \textbf{\cref{Theorem: algorithm convergence - General}}, we prove that the multi-regime version of the algorithm converges to a unique solution under conditions analogous to those in the single-regime case.

\subsection*{The unbalanced SBP}
In \cref{Section: uSBP example}, we apply the general theory from the previous sections to study an \textit{unbalanced Schr\"odinger bridge problem} (uSBP) originated from a jump diffusion with killing. It is natural to fit such a process into our regime-switching model with two regimes only - ``active'' and ``dead'', where the regime-switching mechanism is reduced to a simple form governed by the killing rate. We set up the uSBP in a novel way that, at the terminal time, not only the position of ``active'' particles but also the killing position of ``dead'' particles are required to match the target distribution. By adapting the existing regularity theory of jump diffusions to the case with killing, we can configure the uSBP model so that all the assumptions invoked in the general theory become accessible, in the sense that they are either directly fulfilled or verifiable by the target distribution. Moreover, most of the general results from the previous sections, including the stochastic control formulation, the transition (and marginal) density function of $\widehat\bP$, the dynamic Schr\"odinger system of $(\varphi,\widehat\varphi)$, the Schr\"odinger bridge forward PIDE, etc., will take explicit forms in the uSBP setting. Some possible variations of this uSBP will also be discussed.

\subsection{Motivation, Background and Assumptions}\label{subsection: general background}
Regime-switching (jump) diffusions form a rich class of stochastic processes that capture a broad range of applications in operations management \cite{sethi2005average}, biology \cite{luo2007stochastic, tuong2019extinction}, financial modeling \cite{ramponi2012fourier, zhang2001stock}, actuarial science \cite{hu2025risk}, etc. In contrast to single-regime models, regime-switching jump diffusions provide a framework for modeling discrete events that affect the ``environment'' in which the stochastic process is evolving. As a result, multi-regime models enable a more realistic representation of the dynamical system. In the last thirty years, regime-switching diffusions and jump diffusions have seen increasing interest, especially for optimization and control problems. In fact, the SBP belongs to both of these categories: it is an optimization problem in its primal formulation and a control problem in its stochastic control formulation. The SBP is thus highly relevant for both modeling and control purposes in regime-switching applications. 

The SBP has a rich history going back to E. Schr\"odinger's original thought experiment in the 1930s \cite{schrodinger1931uber,schrodinger1932theorie} and has generated a large body of research (e.g. \cite{beurling1960automorphism,fortet1940resolution,jamison1975markov,pavon1991free,daipra1991stochastic,chen2016entropic}), especially in the diffusion case. As regime-switching models become more and more widely adopted, we believe it is necessary to expand existing knowledge on the SBP in this direction. On the other hand, introducing regime switching to a (jump) diffusion adds yet another layer of complexity that imposes new technical challenges. The purpose of this work is twofold: first, we aim to establish fundamental theoretical results for the SBP for regime-switching (jump) diffusions under assumptions that are analogous to the single-regime SBP. Second, we show how an unbalanced SBP can be modeled via a regime-switching approach, and we provide a detailed analysis that establishes new results beyond the existing literature on this topic.\\

Despite the general literature being scarce compared to that of diffusions, the field of jump diffusions, with and without regime switching, is fast developing and the research body is expanding in both the depth and the diversity. With this in mind, instead of restricting ourselves to a rigid set-up, we put the basic requirements only (Assumption \ref{Assumptions C}) on the model under consideration so that our work can potentially reach a broad class of regime-switching jump diffusions. On the other hand, when exploring certain aspects of the SBP, in order to establish the desired results, we will invoke assumptions in terms of the specific properties required of our model, and these properties may be fulfilled in several ways, as we will briefly explain below. 

To start, Assumption \ref{Assumption: SDE} calls for the existence and uniqueness of the solution to the SDE \eqref{Reference SDE general theory}. The solution theory for general jump diffusions is well studied; in particular, it is known that \eqref{Reference SDE general theory} admits a unique (\textit{strong}) solution under Lipschitz-type conditions on the coefficients (see, e.g., \cite[Chapter 6.2]{applebaum2009levy}). Specifically in the regime-switching setting, results on the existence and uniqueness of solutions can be found in \cite{blom2024feller,kunwai2020feller,kunwai2021regimeswitching,zhang2020regimeswitching,xi2017feller,yin2009hybrid}, among others. In fact, for results in \cref{subsection:preliminaries} and discussions in \cref{Section: densities}, \ref{Assumption: SDE} can be relaxed to the \textit{wellposedness of the martingale problem} associated with the SDE \eqref{Reference SDE general theory}\footnote[2]{Unlike for diffusions, it is not yet generally established for jump diffusions that the wellposedness of the martingale problem is equivalent to the exisitence and uniqueness of solution to the corresponding SDE; some results in this direction are known in the case of time-homogeneous coefficients \cite{lepeltier1976probleme, kurtz2011equivalence}.}, which is known to hold when the coefficients are bounded and continuous in both time and space (e.g., \cite{stroock1997multidimensional,komatsu1973markov}).

As for Assumption \ref{asm: A2 - leads to fg existence}, we rely on it, together with the Markov property of the SDE solution, to obtain the Schr\"odinger bridge $\widehat\bP$ in the desired static form \eqref{eq:bP^ in terms of fg}, which serves as the starting point of our investigation. However, this is not the only way to achieve \eqref{eq:bP^ in terms of fg}; $\widehat\bP$ is known to hold such a representation under different combinations of conditions. For a detailed discussion, we refer readers to \cite[Section 2]{leonard2014survey}.

In \cref{Section: General Theory}, in order to develop a stochastic analysis approach to the SBP, we will impose Assumption \ref{Assumption: general harmonic varphi} (and its strengthening \ref{asm : log varphi C12b}), which requires the function $\varphi$ defined in \eqref{Definition: general varphi} to possess sufficient regularity (in the temporal and the spatial variables) and to be \textit{harmonic} in the sense that $\varphi$ satisfies the Kolmogorov backward equation corresponding to the SDE \eqref{Reference SDE general theory}. This is the case for a wide range of diffusions \cite{jamison1975markov,daipra1991stochastic}. As for jump diffusions, while the regularity theory is still developing, there have been many results on the Kolmogorov backward equation in recent years (\cite{zlotchevski2024schrodinger,chen2016heata,chen2016heat,chen2018heat}).

A common scenario in which the regularity and the harmonic property of $\varphi$ can be achieved is when the reference measure $\bR$ admits a transition density function that is sufficiently regular and the function $\g$ is bounded; this is exactly the motivation of Assumption \ref{Assumptions D} and its strengthening \ref{Assumptions D^}), which are the main assumptions in \cref{Section: densities}. In the current literature, there are results on the existence and the regularity of the transition density function for single-regime jump diffusions (e.g. \cite{kunita2019stochastic, chen2018heat, yin2009hybrid}). In the case of \textit{time-homogeneous} jump diffusions, the method in \cite[Section 5]{xi2009asymptotica} offers a way to derive the multi-regime transition density function in the case when every ``fixed regime'' jump diffusion (i.e., the SDE \eqref{Reference SDE general theory} without the regime-switching component and with the remaining coefficients restricted to a fixed regime) admits a transition density function. In addition, when the regime-switching mechanism is simple (e.g., the uSBP model considered in \cref{Section: uSBP example}), it is also possible to extend the results on the transition density function of a single-regime jump diffusion to the multi-regime setting.

Finally, the uSBP treated in \cref{Section: uSBP example} is a demonstration that all the aforementioned assumptions can be fulfilled when the coefficients of the SDE \eqref{Reference SDE general theory} are sufficiently regular and the target distributions $\rho_0,\rho_T$ are properly chosen, for which we adopt the framework in \cite[Section 4-6]{kunita2019stochastic}. More generally, whenever the regularity theory of a diffusion or jump diffusion is available, we can apply the same method to study the uSBP arising from the process with killing.

\subsection{Notation and Terminology}
For a function $(t,x,i)\in[0,T]\times\R^d\times\rS\mapsto f(t,x,i)\in\R$, we denote by $\nabla f,\nabla^2 f$ the gradient and the Hessian respectively of $f$ with respect to the spatial variable $x$ only. 

We say that a function $f$ on $[0,T]\times\R^d\times\rS$ is \textit{of class $C^{1,2}$} if, for every $i\in\rS$, $f(\cdot,i)\in C([0,T]\times\R^d)$, and as a function on $(0,T)\times\R^d$, $f(\cdot,i)$ is once differentiable in the temporal variable and twice differentiable in the spatial variable with $\frac{\partial}{\partial t}f(\cdot,i),\nabla f(\cdot,i),\nabla^2 f(\cdot,i)\in C((0,T)\times\R^d)$. If, in addition, $f(\cdot,i)$ is supported on a compact subset of $(0,T)\times\R^d$, then we say that $f$ is \textit{of class $C^{1,2}_c$}. 

We say that $f$ is \textit{of class $C_b^{1,2}$}, if $f$ is of class $C^{1,2}$, and for every $i\in\rS$, as functions on $(0,T)\times\R^d$, $\frac{\partial}{\partial t}f(\cdot,i), \nabla f(\cdot,i),\nabla^2 f(\cdot,i)$ are bounded in the spatial variable, with the bounds being uniform in the temporal variable over any compact subset of $(0,T)$. 

Other function classes appearing in the paper (e.g. $C^{1,1}$, $C^{1,\infty}_b$) are similarly defined. 

\section{The Schr\"odinger Bridge}\label{Section: General Theory}
As stated in \cref{subsection:main results}, throughout this work, Assumptions \ref{Assumptions C} and \ref{Assumptions A} are in effect for the SDE \eqref{Reference SDE general theory} and the SBP \eqref{Definition:DynamicSBP}. Given a path measure $\bR$ whose canonical process $\set{(X_t,\Lambda_t)}_{[0,T]}$ is a strong Markov process satisfying \eqref{Reference SDE general theory}, and two target distributions $\rho_0,\rho_T$, recall from \cref{subsection:static view of bP^} that the unique solution to \eqref{Definition:DynamicSBP}, also called the Schr\"odinger bridge, is given by $\widehat\bP=\f(X_0,\Lambda_0)\g(X_T,\Lambda_T)\bR$, where $\f,\g$ are a solution to the static Schr\"odinger system \eqref{Equation: fg Schrodinger System with lambda}. We will first carry out some preliminary analysis on $\widehat\bP$ based on this static viewpoint. 

\subsection{Preliminaries of the Schr\"odinger Bridge}\label{subsection:preliminaries}
As before, we denote by $\bR_0$ and $\bR_T$ the initial data and the terminal data respectively of $\bR$. Instead of the original path measure $\bR$, it is often convenient to work with the following modification: 
\begin{equation*}
   \bP^0 := \frac{d\rho_0}{d\bR_0}(X_0,\Lambda_0)\bR. 
\end{equation*}
It is clear that $\bP^0\in\cP(\Omega)$ is again a strong Markov path measure which only differs from $\bR$ at the initial data. In other words, the canonical process $\set{(X_t,\Lambda_t)}_{[0,T]}$ under $\bP^0$ remains as a strong Markov process satisfying the SDE \eqref{Reference SDE general theory}, but with the desired initial distribution $\rho_0$ instead of $\bR_0$. 

Suppose that the transition distributions under $\bP^0$ (same as under $\bR$) are
\[\set{P_{ij}(t,x,s,dy):i,j\in\rS,\,0\leq t<s\leq T,\,(\rho_0\textup{-a.e., if $t=0$) }x\in\R^d}.\] 
Then, this family of distributions possesses the \textit{semi-group property} in the sense that, for every $0\leq t<r<s\leq T$, ($\rho_0$-a.e. if $t=0$) $x\in\R^d$, $i,j\in\rS$ and Borel $B\subseteq\R^d$,
\begin{equation}\label{eq:C-K eq for p_ij}
    P_{ij}(t,x,s,B)=\sum_{\iota\in\rS}\int_{\R^d}P_{\iota j}(r,w,s,B)P_{i\iota}(t,x,r,dw).
\end{equation}

Let $\varphi$ be defined as in \eqref{Definition: general varphi}. By \eqref{Equation: fg Schrodinger System with lambda}, for $\bR_T$-a.e. $(y,j)\in\R^d\times\rS$ and $\bR_0$-a.e.  $(x,i)\in\R^d\times\rS$, $\varphi(T,y,j)=\g(y,j)$ and \begin{equation*}
\f(x,i)\varphi(0,x,i)=\f(x,i)\bE_{\bR}\brac{\g(X_T,\Lambda_T)|(X_0,\Lambda_0)=(x,i)}=\frac{d\rho_0}{d\bR_0}(x,i). 
\end{equation*} 
From here it follows that $\varphi(0,x,i)>0$ for $\rho_0$-a.e. $x\in\R^d$ (or equivalently,  $\varphi(0,X_0,\Lambda_0)>0$ $\bP^0$-a.e.), and hence $\widehat\bP$ can be written as 
\begin{equation}\label{equation:bP^ as h-transf of bP^0}
    \widehat\bP=\frac{\varphi(T,X_T,\Lambda_T)}{\varphi(0,X_0,\Lambda_0)}\bP^0.
\end{equation}
Therefore, we also have $\varphi(T,X_T,\Lambda_T)=\g(X_T,\Lambda_T)<\infty$ $\bP^0$-a.e. 

Based on the above expression of $\widehat\bP$, we make the following observation.
\begin{lem}\label{lem: transition distribution under bP^}
The Schr\"odinger bridge  $\widehat\bP$ is a strong Markov path measure with the transition distributions
\begin{equation}\label{eq: transition distribution under P^}
    \widehat P_{ij}(t,x,s,dy):=\frac{\varphi(s,y,j)}{\varphi(t,x,i)}P_{ij}(t,x,s,dy)\;\textit{if }\varphi(t,x,i)>0,\;\textit{trivial  otherwise},
\end{equation}
for every $i,j\in\rS$, $0\leq t<s\leq T$ and every ($\rho_0$-a.e. if $t=0$) $x\in\R^d$.
\end{lem}
\begin{proof}
 
We first observe that, according to \eqref{Definition: general varphi}, for every $(x,i)\in\R^d\times\rS$ and $0\leq t<s\leq T$,
\[
\sum_{j\in\rS}\int_{\R^d}\varphi(s,y,j)P_{ij}(t,x,s,dy)=\varphi(t,x,i),
\]
so the expression in \eqref{eq: transition distribution under P^} is always well defined. Next, to establish \eqref{eq: transition distribution under P^}, it is sufficient for us to examine the finite-dimensional marginal distributions of $\widehat\bP$ based on the relation \eqref{equation:bP^ as h-transf of bP^0} between $\widehat\bP$ and $\bP^0$. Then, the strong Markov property of $\widehat\bP$ follows directly from that of $\bP^0$. 

Let  $\set{(X_t,\Lambda_t)}_{[0,T]}$ be the canonical process under $\widehat\bP$. We take any positive integer $K\geq1$, $0=t_0< t_1<\cdots<t_{K-1}<t_K=T$, $i_0,i_1,\cdots,i_{K-1},i_K\in\rS$, and Borel sets $B_0,B_1,\cdots,B_{K-1},B_K\subseteq\R^d$, and write
    \begin{equation*}
    \begin{aligned}
        \widehat\bP&\pran{\bigcap_{k=0}^K\set{X_{t_k}\in B_k,\,\Lambda_{t_k}=i_k}}\\
        &=\bE_{\bP^0}\brac{\frac{\varphi(T,X_T,\Lambda_T)}{\varphi(0,X_0,\Lambda_0)}\one_{\bigcap_{k=0}^K\set{X_{t_k}\in B_k,\,\Lambda_{t_k}=i_k}}}\\
        &=\idotsint_{B_0\times\cdots\times B_K}\frac{\varphi(T,x_K,i_K)}{\varphi(0,x_0,i_0)}\\
        &\hspace{4cm}P_{i_{K-1}\,i_K}(t_{K-1},T,x_{K-1},dx_K)\cdots P_{i_0\,i_1}(0,x_0,t_1,dx_1)\rho_0(dx_0,i_0)\\
        &\overset{\triangle}{=}\idotsint_{\widetilde B_0\times\cdots\times \widetilde B_{K-1}\times B_K}\frac{\varphi(t_1,x_1,i_1)}{\varphi(0,x_0,i_0)}\cdot \frac{\varphi(t_2,x_2,i_2)}{\varphi(t_1,x_1,i_1)}\cdots\frac{\varphi(T,x_K,i_K)}{\varphi(t_{K-1},x_{K-1},i_{K-1})}\\
    &\hspace{4cm}P_{i_{K-1}\,i_K}(t_{K-1},T,x_{K-1},dx_K)\cdots P_{i_0\,i_1}(0,x_0,t_1,dx_1)\rho_0(dx_0,i_0)\\
    &=\idotsint_{B_0\times\cdots\times B_K}\widehat P_{i_{K-1}\,i_K}(t_{K-1},T,x_{K-1},dx_K)\cdots \widehat P_{i_0\,i_1}(0,x_0,t_1,dx_1)\rho_0(dx_0,i_0),
    \end{aligned}
\end{equation*}
where $\widetilde B_k:=B_k\cap\set{x:\varphi(t_k,x,i_k)>0}$ and $\widehat P_{i_{k}\,i_{k+1}}(t_{k},x_{k},t_{k+1},dx_{k+1})$ defined as in \eqref{eq: transition distribution under P^} for $k=0,\cdots,K-1$.

The only thing that remains to be done is to justify the equality ``$\triangle$''.
To this end, we note that, for every $k=0,\cdots,K-1$,
\begin{equation*}
\begin{aligned}
    &\idotsint_{B_{k+1}\times\cdots\times B_K}\varphi(T,x_K,i_K)P_{i_{K-1}\,i_K}(t_{K-1},x_{K-1},T,dx_K)\cdots P_{i_{k}\,i_{k+1}}(t_k,x_k,t_{k+1},dx_{k+1})\\
    &\qquad\leq \sum_{(j_{k+1},\cdots,j_K)\in\rS^{K-k}}\idotsint_{\R^d\times\cdots\times \R^d}\varphi(T,x_K,j_K)\\
    &\hspace{5cm}P_{j_{K-1}\,j_K}(t_{K-1},x_{K-1},T,dx_K)\cdots P_{i_{k}\,j_{k+1}}(t_k,x_k,t_{k+1},dx_{k+1})\\
    &\qquad=\sum_{j_K\in\rS}\int_{\R^d}\varphi(T,x_K,j_K)P_{i_k\,j_K}(t_k,x_k,T,dx_K)\;\textup{ by \eqref{eq:C-K eq for p_ij}}\\
    &\qquad=\varphi(t_k,x_k,i_k)\;\textup{ by \eqref{Equation:phi-definition}},
\end{aligned}
\end{equation*}
and hence
\begin{equation*}
  \begin{aligned} &\idotsint_{B_1\times\cdots\times B_{k-1}\times(B_k\setminus \widetilde B_k)\times B_{k+1}\times\cdots\times B_K}\frac{\varphi(T,x_K,i_K)}{\varphi(0,x_0,i_0)}\\
  &\hspace{4cm}P_{i_{K-1}i_K}(t_{K-1},T,x_{K-1},dx_K)\cdots P_{i_0 i_1}(0,x_0,t_1,dx_1)\rho_0(dx_0,i_0)\\
  &\leq \idotsint_{B_1\times\cdots\times B_{k-1}\times(B_k\setminus \widetilde B_k)}\frac{\varphi(t_k,x_k,i_k)}{\varphi(0,x_0,i_0)}\\
  &\hspace{4.5cm}P_{i_{k-1}\,i_k}(t_{k-1},x_{k-1},t_k,dx_k)\cdots P_{i_0\,i_1}(0,x_0,t_1,dx_1)\rho_0(dx_0,i_0)\\
  &=0\;\textup{ by the fact that }\varphi(t_k,x_k,i_k)=0\textup{ when }x_k\in B_k\setminus\widetilde B_k.
  \end{aligned}  
\end{equation*}
\end{proof}

\noindent\textbf{Remark. }Note that the above lemma on the transition distribution indicates that $\widehat\bP$ can be identified as the \textit{Doob $h$-transform} of $\bP^0$ by the function $\varphi$. We will not expand on this point here, since it is not relevant to the rest of the work, but refer readers to \cite{zlotchevski2024schrodinger,nagasawa1989transformations,chung2005markov} for more details. \\

A consequence of the lemma above is the following refinement of the relation \eqref{equation:bP^ as h-transf of bP^0}.
\begin{cor}\label{cor:relation between bP^ and bP from t to s}
   For every $(t,x,i)\in[0,T]\times\R^d\times\rS$, denote by $\bP^{t,x,i}$ the path measure that is the conditional distribution of $\bP^0$ conditioned on $(X_t,\Lambda_t)=(x,i)$, where $\set{(X_t,\Lambda_t)}_{[0,T]}$ is the canonical process under $\bP^0$, and let $\widehat\bP^{t,x,i}$ be the counterpart  obtained from $\widehat\bP$. Then, 
   \begin{equation*}
     \widehat\bP^{t,x,i}=\frac{\varphi(T,X_T,\Lambda_T)}{\varphi(t,x,i)}\bP^{t,x,i}\;\textit{ if }\;\varphi(t,x,i)>0,\;\textit{trivial otherwise.}  
   \end{equation*}
Moreover, for every $s\in[t,T]$, if $\bP^{x,i}_{t,s}:=\bP^{t,x,i}\bigg|_{\cF_s}$ and $\widehat\bP^{x,i}_{t,s}:=\widehat\bP^{t,x,i}\bigg|_{\cF_s}$ are the restrictions of the respective path measures on $\cF_s$, then
   \begin{equation}\label{eq:relation betwel bP^ and bP over (t,s)}
     \widehat\bP^{x,i}_{t,s}=\frac{\varphi(s,X_s,\Lambda_s)}{\varphi(t,x,i)}\bP^{x,i}_{t,s}\;\textit{ if }\;\varphi(t,x,i)>0,\;\textit{trivial otherwise.}  
   \end{equation}
\end{cor}
\noindent We will omit the proof of the corollary, as both statements are direct consequences of the finite-dimensional distributions derived in the proof of \cref{lem: transition distribution under bP^}\\

We need to make one more preparation for the upcoming discussions, and it concerns the infinitesimal generator, denoted by $L$, associated with the path measure $\bP^0$ (and hence $\bR$ as well) where, for a measurable and bounded function $f$ on $[0,T]\times\R^d\times\rS$, 
\begin{equation}\label{eq: L as the infinitismal generator}
    Lf(t,x,i):=\lim_{h\searrow0}\frac{1}{h}\bigg(\bE_{\bP^{t,x,i}}[f(t,X_{t+h},\Lambda_{t+h})]-f(t,x,i)\bigg)
\end{equation}
whenever the limit exists, with $\set{(X_t,\Lambda_t)}_{[0,T]}$ being the canonical process under $\bP^{t,x,i}$. In view of the SDE \eqref{Reference SDE general theory} that $\set{(X_t,\Lambda_t)}_{[0,T]}$ satisfies, $L$ also takes an integro-differential operator form as follows: if $f$ is of class $C^{1,2}$, then for every $(t,x,i)\in(0,T)\times\R^d\times\rS$,
\begin{equation}\label{Equation:Reference RS-generator}
\begin{aligned}
    &Lf(t,x,i)
    \\&\quad=b(t,x,i)\cdot\nabla f(t,x,i)+\frac{1}{2}\sum_{m,n} (\sigma\sigma^\top)_{m,n}\frac{\partial^2 f}{\partial x_m \partial x_n}(t,x,i)\\
    &\hspace{0.5cm}+\int_{\R^\ell} \pran{f(t,x+\gamma(t,x,i,z),i)-f(t,x,i)-\one_{|z|\leq 1}\gamma(t,x,i,z) \cdot \nabla f(t,x,i)}\nu(dz) \\
    &\hspace{5.5cm}+ \sum_{j\in\rS} Q_{ij}(t,x)\pran{f(t,\psi_{ij}(t,x),j)-f(t,x,i)}.
\end{aligned}  
\end{equation}
Note that the last line in the above expression of $L$ is a rewriting of the original regime-switching component via the identity
\begin{equation}\label{Equation: beta and alpha to Qij}
    \begin{aligned}
    \int_{\R_+}(f(t,x+\beta(t,x,i,w),i+&\alpha(t,x,i,w))-f(t,x,i))\leb(dw)\\
    &=
    \sum_{j\in\rS} Q_{ij}(t,x)(f(t,\psi_{ij}(t,x),j)-f(t,x,i)).
    \end{aligned}
\end{equation}

For $f$ sufficiently regular, e.g., $f$ of class $C^{1,2}_b$, the formulations \eqref{eq: L as the infinitismal generator} and \eqref{Equation:Reference RS-generator} of $Lf$ coincide. We simply refer to the integro-differential operator $L$ in \eqref{Equation:Reference RS-generator} as the \textit{generator} of the regime-switching jump diffusion described by the SDE \eqref{Reference SDE general theory}. The generator plays an important role in the study of the dynamics of the process. In particular, it yields the \textit{martingale problem} view of the SDE. Namely, for every $(t,x,i)\in[0,T]\times\R^d\times\rS$, the path measure $\bP^{t,x,i}$ as in \cref{cor:relation between bP^ and bP from t to s} (same as in Assumption \ref{Assumption: SDE}) is a martingale problem solution to the SDE \eqref{Reference SDE general theory} in the sense that, for any $f$ of class $C^{1,2}_c$,
\begin{equation*}
        \set{f(s,X_s,\Lambda_s)-f(t,x,i)-\int_t^s\bigg(\frac{\partial}{\partial r}+L\bigg)f(r,X_r,\Lambda_r)dr:t\leq s\leq T}
\end{equation*}
is a martingale with $\set{(X_t,\Lambda_t)}_{[0,T]}$ being the canonical process under $\bP^{t,x,i}$. 

In addition to the martingale perspective, we can also acquire from the generator information about the drift, the diffusion, the jump, and the regime-switching rate of the concerned process. As we will see in the later sections, the generator approach also offers a revealing view when it comes to the dynamics of the Schr\"odinger bridge $\widehat\bP$.   
\subsection{Schr\"odinger Bridge SDE and Generator} \label{Subsection: P-hat SDE and generator}
To study the Schr\"odinger bridge $\widehat\bP$ as in \eqref{equation:bP^ as h-transf of bP^0}, we first approach it from an SDE viewpoint. More specifically, we will interpret $\widehat\bP$ as a Girsanov transform from $\bP^0$, based on which we will derive the SDE that governs the canonical process under $\widehat\bP$. However, to achieve this goal, we need to impose additional conditions on the function $\varphi$ defined in \eqref{Definition: general varphi} and the generator $L$ as in \eqref{Equation:Reference RS-generator}, so that we can fit the formulation \eqref{equation:bP^ as h-transf of bP^0} of $\widehat\bP$ into the Girsanov transform framework. To this end, the following assumption will be invoked in this subsection.
\begin{enumerate}[label=(\textbf{H})]
\item \label{Assumption: general harmonic varphi}
   As a function on $[0,T]\times\R^d\times\rS$, $\varphi$ is of class $C^{1,2}$. In addition,
    \begin{enumerate}
        \item $\varphi>0$ on $(0,T)\times\R^d \times \rS$, and
        \item $\varphi$ is \textit{harmonic} under $L$ in the sense that $(\frac{\partial}{\partial t}+L) \varphi=0$ on $(0,T)\times\R^d \times \rS$.
    \end{enumerate} 
\end{enumerate}
\noindent \textbf{Remark. }Assumption \ref{Assumption: general harmonic varphi} entails that, as $t\nearrow T$, $\varphi(t,x,i)\rightarrow\varphi(T,x,i)=\g(x,i)$ for every $(x,i)\in\R^d\times\rS$. However, in view of the Schr\"odinger system \eqref{Equation: fg Schrodinger System with lambda}, only the restriction of $\g(\cdot,i)$ on $\textup{supp}(\bR_T(\cdot,i))$ is concerned in solving the SBP, and the value of $\g(\cdot,i)$ elsewhere has no impact. Therefore, the $C^{1,2}$-requirement in \ref{Assumption: general harmonic varphi} at $t=T$ should be understood as $\lim_{t\nearrow T}\varphi(t,\cdot,i)=\g(\cdot,i)$ $\bR_T$-a.e.

We have discussed in the \cref{subsection: general background} possible scenarios when $\varphi$ has the desired order of regularity and satisfies the harmonic condition. On the other hand, the positivity of $\varphi$ required in Assumption \ref{Assumption: general harmonic varphi} can be fulfilled, for example, when the target distributions $\rho_0,\rho_T$ are such that $\g>0$ everywhere, or the reference measure $\bR$ has an everywhere positive transition distribution. 
\subsubsection*{General Girsanov Transform}
To set up the Girsanov transform framework, consider measurable functions $u:(0,T)\times\R^d\times\rS\rightarrow\R^d$, $\theta:(0,T)\times\R^d\times\rS\times\R^\ell\rightarrow(-\infty,1)$ and $\xi:(0,T)\times\R^d\times\rS\times\R_+\rightarrow(0,\infty)$ such that, as functions on $(0,T)\times\R^d\times\rS$,\[
u(\cdot),\;\sup_{z\in\R^\ell}\frac{\abs{\theta(\cdot,z)}}{1\wedge\abs{z}},\;\sup_{w\in\R_+}\abs{\xi(\cdot,w)}
\]
are locally bounded, and for every $(t,x,i)\in(0,T)\times\R^d\times\rS$, $w\mapsto1-\xi(t,x,i,w)$ is compactly supported. Let $\set{(X_t,\Lambda_t)}_{[0,T]}$ be the canonical process under $\bP^0$. For every $0<t<s<T$,
we define
\begin{equation}\label{girsanov zt}
\begin{aligned}
    Z^{(u,\theta,\xi)}_{t,s}:=&\exp\bigg(-\int_t^su(r,X_{r},\Lambda_{r})\,dB_r-\frac{1}{2}\int_t^s|u(r,X_{r},\Lambda_{r})|^2\,dr\\
    &+ \int_t^s\int_{\R^\ell}\log(1-\theta(r,X_{r-},\Lambda_{r-},z))\tilde{N}(dr,dz) \\
    &+\int_t^s\int_{\R^\ell}\brac{\log(1-\theta(r,X_{r},\Lambda_{r},z))+\theta(r,X_{r},\Lambda_{r},z)}\nu(dz)dr\\
    &+ \int_t^s\int_{\R_+}\log\xi(r,X_{r-},\Lambda_{r-},w)\tilde{N}_1(dr,dw) \\
    &+\int_t^s\int_{\R_+}\brac{\log\xi(r,X_{r},\Lambda_{r},w)+1-\xi(r,X_{r},\Lambda_{r},w)}\leb(dw)dr\bigg),
\end{aligned}
\end{equation} 
It is straightforward to check that, given the choice of $(u,\theta,\xi)$, all the stochastic integrals in \eqref{girsanov zt} are well defined. By the Girsanov theorem for jump diffusions (e.g. \cite{oksendal2019applied}, Theorem 1.33), if $\bE_{\bP^0}[Z^{(u,\theta,\xi)}_{t,s}]=1$, then for every $(x,i)\in\R^d\times\rS$, $\bQ^{x,i}_{t,s}:=Z^{(u,\theta,\xi)}_{t,s}\bP^{x,i}_{t,s}$ is a path measure restricted to $\cF_s$, where we recall that $\bP^{x,i}_{t,s}$ is defined in \cref{cor:relation between bP^ and bP from t to s}; further, under $\bQ^{x,i}_{t,s}$, $dB^{(u,\theta,\xi)}_t:=u\,dt+dB_t$ is a Brownian motion, \[\tilde{N}^{(u,\theta,\xi)}(dt,dz):=N(dt,dz)-(1-\theta)\nu(dz)dt\] is the compensated Poisson random measure of $N(dt,dz)$, and \[\tilde{N}_1^{(u,\theta,\xi)}(dt,dw):=N_1(dt,dw)-\xi\leb(dw)dt\] is the compensated Poisson random measure of $N_1(dt,dw)$. 

Moreover, we observe that if $\bE_{\bP^0}[Z^{(u,\theta,\xi)}_{t,s}]=1$ for all $0<t<s<T$, then \[\set{\bQ^{x,i}_{t,s}:0<t<s<T,x\in\R^d,i\in\rS}\]
is a consistent family of path measures; that is, for every $0\leq t\leq t^\prime<s^\prime\leq s<T$, $x,y\in\R^d$ and $i,j\in\rS$, if $\big(\bQ^{x,i}_{t,s}\big)^{t^\prime,\,y,\,j}$ denotes  the conditional distribution of  $\bQ^{x,i}_{t,s}$ conditioning on its canonical process $(X_{t^\prime},\Lambda_{t^\prime})\equiv(y,j)$, then $\big(\bQ^{x,i}_{t,s}\big)^{t^\prime,\,y,\,j}$ restricted to $\cF_{s\prime}$ is identical with $\bQ^{y,j}_{t^\prime,s^\prime}$. Hence there exists a path measure $\bP^{(u,\theta,\xi)}$  on $D((0,T);\R^d\times\rS)$ such that, for every $0<t<s<T$ and $(x,i)\in\R^d\times\rS$,\begin{equation}\label{eq:consistency of P^(u,theta,xi)}
   \big(\bP^{(u,\theta,\xi)}\big)^{t,x,i}\bigg|_{\cF_s}=\bQ^{x,i}_{t,s}=Z^{(u,\theta,\xi)}_{t,s}\bP^{x,i}_{t,s}. 
\end{equation} 
We will again view $\bP^{(u,\theta,\xi)}$ as a Girsanov transform from $\bP^0$ by the triple $(u,\theta,\xi)$.

Thus, by rewriting the original SDE \eqref{Reference SDE general theory}, we see that the canonical process $\set{(X_t,\Lambda_t)}_{(0,T)}$ under $\bP^{(u,\theta,\xi)}$ is again a regime-switching jump diffusion satisfying a new SDE: 
\begin{equation}\label{Girsanov SDE - bQ,u,theta,theta1}
\left\{
\begin{aligned}
dX_t&=
 \bigg[b(t,X_t,\Lambda_t)-\sigma(t,X_t,\Lambda_t) u(t,X_t,\Lambda_t)\\
 &\hspace{4cm}-\int_{|z|\leq 1}\theta(t,X_t,\Lambda_t,z)\gamma(t,X_t,\Lambda_t,z)\nu(dz)\bigg]dt\\
 &\hspace{1cm}+\sigma(t,X_t,\Lambda_t)dB^{(u,\theta,\xi)}_t + \int_{|z|\leq 1}\gamma(t,X_{t-},\Lambda_{t-},z)\tilde{N}^{(u,\theta,\xi)}(dt,dz)\\
 &\hspace{1cm}+\int_{|z|> 1}\gamma(t,X_{t-},\Lambda_{t-},z)N(dt,dz)
 + \int_{\R_+}\beta(t,X_{t-},\Lambda_{t-},w)N_1(dt,dw), \\
    d\Lambda_t &= \int_{\R_+}\alpha(t,X_{t-},\Lambda_{t-},w)N_1(dt,dw)\qquad\textup{ for }\;t\in(0,T).    
\end{aligned}
\right.
\end{equation}
We can also determine the generator of $\set{(X_t,\Lambda_t)}_{(0,T)}$ under $\bP^{(u,\theta,\xi)}$, denoted by $L^{(u,\theta,\xi)}$: for $f$ of class $C^{1,2}_c$ and $(t,x,i)\in(0,T)\times\R^d\times\rS$,
\begin{equation}\label{Girsanov generator}
    \begin{aligned}
        L^{(u,\theta,\xi)}&f(t,x,i)\\:=
        &\frac{1}{2}\sum_{m,n} (\sigma\sigma^\top)_{m,n}\frac{\partial^2 f}{\partial x_m \partial x_n}(t,x,i)\\
        &\hspace{2cm}+\pran{b-\sigma u- \int_{|z|\leq 1} \theta(\cdot,z)\gamma(\cdot,z)\nu(dz)}(t,x,i)\cdot \nabla f(t,x,i) \\
        & +\int_{\R^\ell} \bigg(f(t,x+\gamma(t,x,i,z),i)-f(t,x,i)\\
        &\hspace{3cm}- \one_{|z|\leq 1}\nabla f(t,x,i) \cdot \gamma(t,x,i,z)\bigg)(1-\theta(t,x,i,z))\nu(dz) \\
        &+\int_{\R_+} \bigg(f(t,x+\beta(t,x,i,w),i+\alpha(t,x,i,w))-f(t,x,i)\bigg)\xi(t,x,i,w)\leb(dw),
    \end{aligned}
\end{equation}
where, by the equivalence \eqref{Equation: beta and alpha to Qij}, the last term can also be written as
\[
\sum_{j\in\rS}\pran{f(t,\psi_{ij}(t,x),j)-f(t,x,i)}\int_{\Delta_{ij}(t,x)}\xi(t,x,i,w)\leb(dw).
\]
\subsubsection*{The ``Bridge'' Girsanov Transform}
We now apply the above general discussions on the Girsanov transform to prove the main result of this subsection.
\begin{thm}\label{Theorem: General P-hat} 
Suppose that Assumption \textup{\ref{Assumption: general harmonic varphi}} holds. Let $\widehat\bP$ be the Schr\"odinger bridge as in \eqref{equation:bP^ as h-transf of bP^0}. Then, the canonical process $\set{(X_t,\Lambda_t)}_{(0,T)}$ under $\widehat\bP$ satisfies the SDE:
\begin{equation}\label{Phat SDE general theory}
\left\{
    \begin{aligned}
        dX_t&=b^\varphi(t,X_{t},\Lambda_t)dt+\sigma(t,X_{t},\Lambda_t)dB_t + \int_{|z|\leq 1}\gamma(t,X_{t-},\Lambda_{t-} ,z)\tilde{N}^\varphi(dt,dz)\\
        &\hspace{1cm} +\int_{|z|>1}\gamma(t,X_{t-},\Lambda_{t-} ,z)N(dt,dz) + \int_{\R_+}\beta(t,X_{t-},\Lambda_{t-},w)N_1(dt,dw), \\
    d\Lambda_t &= \int_{\R_+}\alpha(t,X_{t-},\Lambda_{t-},w)N_1(dt,dw),\;\textup{ for }0<t<T,     
    \end{aligned}
\right.
\end{equation}
with $(X_0,\Lambda_0)$ and $(X_T,\Lambda_T)$ having distributions $\rho_0$ and $\rho_T$ respectively, where for every $(t,x,i)\in(0,T)\times\R^d\times\rS$, $z\in\R^\ell$ and $w\in\R_+$,
\begin{enumerate}
    \item The drift coefficient is given by \begin{equation*}
    \begin{aligned}
    b^\varphi(t,x,i):=b(t,x,i)&+\sigma\sigma^\top\nabla\log\varphi(t,x,i) + \\
 & \int_{|z|\leq 1} \frac{\varphi(t,x+\gamma(t,x,i,z),i)-\varphi(t,x,i)}{\varphi(t,x,i)}\gamma(t,x,i,z)\nu(dz)
    \end{aligned}
\end{equation*}
\item $\set{B_t}_{t\geq 0}$ is a $d$-dimensional Brownian motion.\\

\item $N(dt,dz)$ is an $\ell$-dimensional Poisson random measure with intensity measure \[\nu^\varphi(t,x,i;dz):=\frac{\varphi(t,x+\gamma(t,x,i,z),i)}{\varphi(t,x,i)}\nu(dz) \] and $\tilde{N}^\varphi(dt,dz):=N(dt,dz)-\nu^\varphi(t,x,i;dz)dt$ is its $\widehat\bP$-compensated counterpart.\\
\item $N_1(dt,dw)$ is a one-dimensional Poisson random measure with intensity measure \[\leb^\varphi(t,x,i;dw)= \frac{\varphi(t,x+\beta(t,x,i,w),i+\alpha(t,x,i,w))}{\varphi(t,x,i)}\leb(dw).\]Or equivalently, for $w\in\Delta_{ij}(t,x)$, $\leb^\varphi(t,x,i;dw)=\frac{\varphi(t,\psi_{ij}(t,x),j)}{\varphi(t,x,i)}\leb(dw)$.

\end{enumerate}

Moreover, if $L_{\widehat\bP}$ denotes the generator of $\set{(X_t,\Lambda_t)}_{(0,T)}$ under $\widehat\bP$, then for every function $f$ of class $C^{1,2}_c$ and $(t,x,i)\in(0,T)\times\R^d\times\rS$,
\begin{equation}\label{P-hat generator}
    \begin{aligned}
    &L_{\widehat\bP}f(t,x,i)\\
    &\;:=b^\varphi(t,x,i)\cdot \nabla f(t,x,i) +\frac{1}{2}\sum_{m,n} (\sigma\sigma^\top)_{m,n}\frac{\partial^2 f}{\partial x_m \partial x_n}(t,x,i)\\
    & \quad+\int_{\R^\ell} \bigg(f(t,x+\gamma(t,x,i,z),i)-f(t,x,i)-\one_{|z|\leq 1}\gamma(t,x,i,z) \cdot \nabla f(t,x,i)\bigg)\nu^\varphi(dz) \\
    &\hspace{3.3cm}+ \sum_{j} \frac{\varphi(t,\psi_{ij}(t,x),j)}{\varphi(t,x,i)}Q_{ij}(t,x)\pran{f(t,\psi_{ij}(t,x),j)-f(t,x,i)},
    \end{aligned}
\end{equation}
and further,
\begin{equation}\label{h-transform generator identity}
    \pran{\frac{\partial}{\partial t}+ L_{\widehat\bP}} f (t,x,i)= \frac{1}{f(t,x,i)}\pran{\frac{\partial}{\partial t}+L}(\varphi f)(t,x,i)
\end{equation}
whenever $f(t,x,i)\neq 0$.
\end{thm}
\begin{proof}
    This result is established by \cite[Theorem 4.2]{zlotchevski2024schrodinger} for jump diffusions without regime switching. We extend the proof to cover the regime-switching case. In particular, we will show that under Assumption \ref{Assumption: general harmonic varphi}, when restricted to $\cF_{T-}$, $\widehat\bP$ is in fact a Girsanov transform from $\bP^0$ with $(u,\theta,\xi)$ explicitly determined by $\varphi$. Throughout the proof, whenever possible, we suppress the argument of $\gamma(t,x,i,z)$, $\beta(t,x,i,w)$, and $\alpha(t,x,i,w)$ to simplify the notation.
    
   For now, we take $\set{(X_t,\Lambda_t)}_{[0,T]}$ to be the canonical process under $\bP^0$, and start with setting up It\^o's formula for $\set{\log \varphi(t,X_t,\Lambda_t)}_{(0,T)}$. This is still feasible even though $\varphi(t,x,i)$ is not differentiable in the third variable. 
   By direct computation and the assumption that $\varphi$ is harmonic, we have that
    \begin{equation}\label{-log phi identity}
    \begin{aligned}
    &\pran{\frac{\partial}{\partial t}+ L} (\log \varphi)(t,x,i) \\
    &\hspace{1cm}= -\frac{1}{2}|\sigma^\top\nabla\log \varphi|^2 (t,x,i)\\
    &\hspace{2cm}+\int_{\R^\ell} \brac{\log\frac{\varphi(t,x+\gamma,i)}{\varphi(t,x,i)}-\frac{\varphi(t,x+\gamma,i)-\varphi(t,x,i)}{\varphi(t,x,i)}}\nu(dz) \\
    &\hspace{2cm}+ \sum_{j\in\rS}Q_{ij}(t,x)\brac{\log\frac{\varphi(t,\psi_{ij}(t,x),j)}{\varphi(t,x,i)}-\frac{\varphi(t,\psi_{ij}(t,x),j)-\varphi(t,x,i)}{\varphi(t,x,i)}},
    \end{aligned}
    \end{equation}
    where we used the equivalence \eqref{Equation: beta and alpha to Qij} for the last term. 
   Combining the above computation with It\^o's formula yields, for every $0<t<s<T$,
    \begin{equation}\label{eq: Ito formula applied to -log(varphi)}
        \begin{aligned}
            &\log \pran{\frac{\varphi(s,X_s,\Lambda_s)}{\varphi(t,X_t,\Lambda_t)}}\\
            &\quad=\int_t^{s}\sigma^\top\nabla\log \varphi(r,X_r,\Lambda_r)\,dB_r-\frac{1}{2}\int_t^{s}|\sigma^\top\nabla\log \varphi(r,X_r,\Lambda_r)|^2\,dr\, \\
            &\qquad\qquad +\int_t^{s}\int_{\R^\ell} \log\frac{\varphi(r,X_{r-}+\gamma,\Lambda_{r-})}{\varphi(r,X_{r-},\Lambda_{r-})}\tilde{N}(dr,dz)\\
            &\qquad\qquad+\int_t^{s}\int_{\R^\ell} \pran{\log\frac{\varphi(r,X_{r-}+\gamma,\Lambda_{r-})}{\varphi(r,X_{r-},\Lambda_{r-})}+1-\frac{\varphi(r,X_{r-}+\gamma,\Lambda_{r-})}{\varphi(r,X_{r-},\Lambda_{r-})}}\nu(dz)dr\\
            &\qquad\qquad+\int_t^{s}\int_{\R_+}\log \frac{\varphi(r,X_{r-}+\beta,\Lambda_{r-}+\alpha)}{\varphi(r,X_{r-},\Lambda_{r-})} N_1(dr,dw)\\
            &\qquad\qquad+\int_t^{s}\sum_{j\in\rS}Q_{\Lambda_{r-}\,j}(r,X_{r-}) \pran{ 1- \frac{\varphi(r,\psi_{\Lambda_{r-}j}(r,X_{r-}),j)}{\varphi(r,X_{r-},\Lambda_{r-})}}dr.
        \end{aligned}
    \end{equation}
Since $\set{(X_t,\Lambda_t)}_{[0,T]}$ has c\`adl\`ag paths, $\bP^0$-a.e. $\set{(X_t,\Lambda_t)}_{[0,T]}$ is contained in a bounded subset of $\R^d\times\rS$. Thus, Assumption  \ref{Assumptions C} on the coefficients and Assumption \ref{Assumption: general harmonic varphi} on $\varphi$ are sufficient to guarantee that all the (stochastic) integrals in the expression above are well defined. Moreover, the right-hand side is precisely equal to $\log(Z^{(u,\theta,\xi)}_{t,s})$ as in  \eqref{girsanov zt} with the choice
\begin{equation}\label{P-hat Girsanov triple}
\begin{dcases}
u(t,x,i)=-\sigma^\top\nabla\log \varphi(t,x,i),\\
\theta(t,x,i,z) = 1-\frac{\varphi(t,x+\gamma(t,x,i,z),i)}{\varphi(t,x,i)},\\
\xi(t,x,i,w) =  \frac{\varphi(t,x+\beta(t,x,i,w),i+\alpha(t,x,i,w))}{\varphi(t,x,i)}.
\end{dcases}
\end{equation}
Now, by \eqref{eq: transition distribution under P^}, we have that \[
\bE_{\bP^0}\brac{Z^{(u,\theta,\xi)}_{t,s}}=\bE_{\bP^0}\brac{\frac{\varphi(s,X_s,\Lambda_s)}{\varphi(t,X_t,\Lambda_t)}}=1\textup{ for every }0<t<s<T,
\]
which, as stated above, yields a path measure $\bP^{(u,\theta,\xi)}$ on $D((0,T);\R^d\times\rS)$, satisfying \eqref{eq:consistency of P^(u,theta,xi)}, and serving as a Girsanov transform from $\bP^0$ with the triple $(u,\theta,\xi)$ in \eqref{P-hat Girsanov triple}.  Thus, combining \eqref{eq:relation betwel bP^ and bP over (t,s)} and \eqref{eq:consistency of P^(u,theta,xi)}, we arrive at, for every $0<t<s<T$ and $(x,i)\in\R^d\times\rS$, 
\[
\widehat\bP^{x,i}_{t,s}=Z_{t,s}^{(u,\theta,\xi)}\bP^{x,i}_{t,s}=\bQ^{x,i}_{t,s}=\big(\bP^{(u,\theta,\xi)}\big)^{t,x,i}\bigg|_{\cF_s}.
\]
We conclude that $\widehat\bP$, as a path measure on $D((0,T);\R^d\times\rS)$, is identical with  $\bP^{(u,\theta,\xi)}$, and hence can also be viewed as a Girsanov transform from $\bP^0$ by the same $(u,\theta,\xi)$ as above. Therefore, the SDE \eqref{Phat SDE general theory} and the generator \eqref{P-hat generator} claimed in the statement follow from \eqref{Girsanov SDE - bQ,u,theta,theta1} and \eqref{Girsanov generator} respectively. 
The equivalence \eqref{Equation: beta and alpha to Qij} then yields the regime-switching term in the generator $L_{\widehat\bP}$. Again, we note that under Assumptions \ref{Assumptions C} and \ref{Assumption: general harmonic varphi}, all the stochastic integrals concerned in \eqref{Phat SDE general theory} exist under $\widehat\bP$.

Finally, the identity \eqref{h-transform generator identity} can be verified by expanding both sides in terms of the SDE coefficients, incorporating \eqref{P-hat generator} and the assumption that $(\frac{\partial}{\partial t}+L)\varphi=0$.
\end{proof}
As seen from the theorem above, the canonical process under the Schr\"odinger bridge $\widehat\bP$ is again a regime-switching jump diffusion, where, in addition to the new drift $b^\varphi$ and the new jump intensity measure $\nu^\varphi$, the switching mechanism also has a new transition rate matrix $Q^\varphi(t,x)$ with entries $Q^\varphi _{ij}(t,x):=\frac{\varphi(t,\psi_{ij}(t,x),j)}{\varphi(t,x,i)}Q_{ij}(t,x), \,i,j \in \rS$.
\\

\noindent \textbf{Remark.} In the above discussion of the Girsanov transform perspective of the Schr\"odinger bridge, we have excluded the endpoints $t=0$ and $t=T$ (since we already know that the endpoint distributions are $\rho_0,\rho_T$), and focused only on the dynamics of the ``bridge''. However, if 
$\varphi$ is sufficiently regular such that It\^o's formula can be applied to $\set{-\log(\varphi(t,X_t,\Lambda_t))}_{[0,T]}$ (with $t=0,t=T$ included), then \eqref{eq: Ito formula applied to -log(varphi)} holds for all $0\leq t<s\leq T$ and hence, with $(u,\theta,
\xi)$ as in \eqref{P-hat Girsanov triple},
$Z_{0,T}^{(u,\theta,\xi)}=\frac{\varphi(T,X_T,\Lambda_T)}{\varphi(0,X_0,\Lambda_0)}$ and $\bE_{\bP^0}[Z_{0,T}^{(u,\theta,\xi)}]=1$. In this case, the Girsanov transform formulation of $\widehat\bP$ also extends to the endpoints in the sense that, $\widehat\bP=Z^{(u,\theta,\xi)}_{0,T}\bP^0$ is a Girsanov transform from $\bP_0$ as a path measure on $D([0,T];\R^d\times\rS)$. This will be the set-up of our work in the next subsection.

More generally, if the triple $(u,\theta,\xi)$ is such that the formula \eqref{girsanov zt} of $Z_{t,s}^{(u,\theta,\xi)}$ is well defined for all $0\leq t<s\leq T$ and $\bE_{\bP^0}[Z_{0,T}^{(u,\theta,\xi)}]=1$, then $\bP^{(u,\theta,\xi)}=Z_{0,T}^{(u,\theta,\xi)}\bP^0$ becomes a Girsanov transform from $\bP^0$, again, with sample paths over the entire interval $[0,T]$. In this case, we have an explicit integral representation of the KL divergence of $\bP^{(u,\theta,\xi)}$ from $\bP^0$. For simplicity, below we will write $u(t,X_t,\Lambda_t),\,\theta(t,X_{t-},\Lambda_{t-},z),\,\xi(t,X_{t-},\Lambda_{t-},w)$ as $u,\,\theta,\,\xi$ respectively. Assuming all the (stochastic) integrals in \eqref{girsanov zt} are well defined for every $t\in[0,T]$, we can rewrite $\log Z^{(u,\theta,\xi)}_T$ as
\begin{equation}\label{eq: ito formula on logZ}
\begin{aligned} 
    &-\int_0^Tu\;dB^{(u,\theta,\xi)}_r+\frac{1}{2}\int_0^t|u|^2dr\\ 
    &\quad+\int_0^T\int_{\R^\ell}\log(1-\theta)\,\tilde{N}^{(u,\theta,\xi)}(dr,dz) +\int_0^T\int_{\R^\ell}\pran{\log(1-\theta)(1-\theta)+\theta}\nu(dz)dr\\
    &\hspace{1.5cm}+\int_0^T\int_{\R_+}\log\xi\;\tilde{N}_1^{(u,\theta,\xi)}(dr,dw) +\int_0^T\int_{\R_+}\pran{\xi\log \xi+1- \xi}\leb(dw)dr,
\end{aligned}
\end{equation}
where $dB^{(u,\theta,\xi)},\, \tilde{N}^{(u,\theta,\xi)}(dt,dz),\,\tilde{N}_1^{(u,\theta,\xi)}(dt,dw)$ are as in the discussions before \eqref{Girsanov SDE - bQ,u,theta,theta1}, and hence the associated integrals will be martingales under $\bP^{(u,\theta,\xi)}$ when $(u,\theta,\xi)$ satisfies proper conditions. In this case, the KL divergence of $\bP^{(u,\theta,\xi)}$ from $\bP^0$ can be computed as
\begin{equation}\label{Equation: KL divergence Girsanov representation}
\begin{aligned}
    \KL{\bP^{(u,\theta,\xi)}}{\bP^0}&=\bE_{\bP^{(u,\theta,\xi)}}\brac{\log Z^{(u,\theta,\xi)}_T}\\ &=\bE_{\bP^{(u,\theta,\xi)}}\bigg[\int_0^T\bigg(\frac{1}{2}|u|^2+\int_{\R^\ell}\pran{(1-\theta)\log(1-\theta)+\theta}\nu(dz)\\
    &\hspace{4cm}+\int_{\R_+}\pran{\xi\log  \xi+1- \xi}\leb(dw)\bigg)dt\bigg].
\end{aligned}
\end{equation}
This representation of $\KL{\bP^{(u,\theta,\xi)}}{\bP^0}$ is relevant to the stochastic control viewpoint of the Schr\"odinger bridge, which we will introduce in the next subsection. 
\subsection{Stochastic Control Formulation and Solution}\label{Subsection: stochastic control} 
In this subsection, we will establish a stochastic control formulation for the SBP under investigation. Under stronger assumptions on $\varphi$ (compared to the previous subsection), we will show that the sought after Schr\"odinger bridge $\widehat\bP$ can be derived from the solution to a stochastic control problem. To achieve this goal, we adopt the following strengthening of Assumption \ref{Assumption: general harmonic varphi}.\\
\begin{enumerate}[label=(\textbf{H'})]
    \item The function $\varphi$ defined by \eqref{Definition: general varphi} satisfies Assumption \ref{Assumption: general harmonic varphi}. In addition, for every $i\in\rS$, $\varphi(T,\cdot,i)=\g(\cdot,i)>0$ $\bR_T$-a.e., and for every compact set $K\subseteq\R^d$, 
    \begin{equation*}
     \int_0^T\sup_{x\in K}\abs{\nabla\log\varphi(t,x,i)}^2dt<\infty\textup{ and }\int_0^T\sup_{j\neq i,x\in K}\abs{\frac{\varphi(t,x,j)}{\varphi(t,x,i)}}^2dt<\infty.
    \end{equation*}
    \label{asm : log varphi C12b}
\end{enumerate}
\noindent \textbf{Remark.} Let $\set{(X_t,\Lambda_t)}_{[0,T]}$ be the canonical process under $\bP^0$. Recall that, under Assumption \ref{Assumption: general harmonic varphi}, when applying It\^o's formula to $\set{-\log\big(\frac{\varphi(s,X_s,\Lambda_s)}{\varphi(0,X_0,\Lambda_0)}\big)}_{s\in [0,T)}$, we get the equation \eqref{eq: Ito formula applied to -log(varphi)} with $t=0$. Combining Assumption \ref{asm : log varphi C12b} with Assumption \ref{Assumptions C}, it is not hard to verify that all the (stochastic) integrals in the RHS of \eqref{eq: Ito formula applied to -log(varphi)} are still well defined when every ``$\int_t^s$'' is replaced by ``$\int_0^T$''. In addition, given the continuity of $\varphi$ as $t\nearrow T$, both sides of \eqref{eq: Ito formula applied to -log(varphi)} (with $t=0$), as processes parametrized by $s\in[0,T]$, have c\`adl\`ag sample paths over $[0,T]$. Therefore, $\bP^0$-a.e., the relation \eqref{eq: Ito formula applied to -log(varphi)} remains valid all the way to $t=0,s=T$. In this case, as we remarked at the end of the previous subsection, \[\widehat\bP=\bP^{(u,\theta,\xi)}=Z_{0,T}^{(u,\theta,\xi)}\bP^0\textup{ with }(u,\theta,\xi)\textup{ as in \eqref{P-hat Girsanov triple},}\]
i.e., $\widehat\bP$ is indeed the Girsanov transform from $\bP^0$ over paths on the entire interval $[0,T]$.\\   

To develop the stochastic control approach toward the SBP \eqref{Definition:DynamicSBP}, we first need to identify a proper \textit{control class}. \\

\noindent \textbf{Definition. }We define $\cU$ to be the class of triples $(u,\theta,\xi)$ such that 
\begin{enumerate}
    \item $u:(0,T)\times\R^d\times\rS \to \R$, $\theta:(0,T)\times\R^d\times\rS\times\R^\ell \to (-\infty, 1)$, and $\xi:(0,T)\times\R^d\times\rS\times\R_+ \to (0,\infty)$ are measurable functions.
    \item For every $i\in\rS$ and compact set $K\subseteq\R^d$, \begin{equation*}
        \int_0^T\sup_{x\in K}\abs{u(t,x,i)}^2dt<\infty\textup{ and }\int_0^T\sup_{x\in K,z\in\R^\ell}\frac{\abs{\theta(t,x,i,z)}^2}{1\wedge\abs{z}^2}dt<\infty,
    \end{equation*}
    and for every $(t,x)\in(0,T)\times\R^d$, $w\mapsto1-\xi(t,x,i,w)$ is supported on a compact interval, denoted by $I_{t,x}$, and \begin{equation*}
        \int_0^T\sup_{x\in K,w\in I_{t,x}}\abs{1-\xi(t,x,i,w)}^2dt<\infty.
    \end{equation*} 
    \item The above two properties on $(u,\theta,\xi)$ guarantee that all the (stochastic) integrals in \eqref{girsanov zt} are well defined with $t=0$ and $s=T$, in which case \eqref{girsanov zt} yields  $Z_{0,T}^{(u,\theta,\xi)}$. Further, we assume $(u,\theta,\xi)$ is such that $\bE_{\bP^0}[Z_{0,T}^{(u,\theta,\xi)}]=1$, and moreover, if $\bP^{(u,\theta,\xi )}:=Z_{0,T}^{(u,\theta,\xi)}\bP^0$ is the Girsanov transform from $\bP^0$ by $(u,\theta,\xi)$, and $L^{(u,\theta,\xi)}$ is the operator as in \eqref{Girsanov generator}, then the following is a $\bP^{(u,\theta,\xi)}$-martingale:
    \[
    \set{-\log\varphi(t,X_t,\Lambda_t)-\int_0^t\bigg(\frac{\partial}{\partial r}+L^{(u,\theta,\xi)}\bigg)(-\log\varphi)(r,X_r,\Lambda_r)dr:t\in[0,T]}.
    \]
\end{enumerate}

Given $(u,\theta,\xi)\in\cU$, let us revisit the SDE \eqref{Girsanov SDE - bQ,u,theta,theta1}, to which we now refer as the \textit{controlled SDE}, viewing it as the variation of the original SDE \eqref{Reference SDE general theory} under the control of $(u,\theta,\xi)$. By the analysis in the previous subsection, we know that, for every $(u,\theta,\xi)\in\cU$, the canonical process $\set{(X_t,\Lambda_t)}_{[0,T]}$ under $\bP^{(u,\theta,\xi)}$ is a strong Markov process satisfying the controlled SDE \eqref{Girsanov SDE - bQ,u,theta,theta1}, and with the generator given by $L^{(u,\theta,\xi)}$. Below, we will write $(u,\theta,\xi)\in\cU(\rho_T)$ if $(u,\theta,\xi)\in\cU$ and $(X_T,\Lambda_T)$ has distribution $\rho_T$ under $\bP^{(u,\theta,\xi)}$.\\

Returning to the SBP in \eqref{Definition:DynamicSBP}, we will prove in this subsection that, under Assumption \ref{asm : log varphi C12b}, determining the Schr\"odinger bridge $\widehat\bP$ is equivalent to solving the following \textit{Stochastic Control Problem} (SCP).
\begin{dfn}\label{Definition: stochastic control SBP - general}
Under the setting above, the stochastic control formulation of the SBP \eqref{Definition:DynamicSBP} is to determine
\begin{equation}\label{Equation: stochastic control objective}
\begin{aligned}
    &(u^*,\theta^*,\xi^*)\\
    &\qquad:= \arg\min_{(u,\theta,\xi) \in \cU(\rho_T)} \bE_{\bP^{(u,\theta,\xi)}}\bigg[\int_0^T\bigg(\frac{1}{2}\abs{u(t,X_t,\Lambda_t)}^2\\
    &\hspace{1cm}+\int_{\R^\ell}\brac{(1-\theta(t,X_t,\Lambda_t,z))\log(1-\theta(t,X_t,\Lambda_t,z))+\theta(t,X_t,\Lambda_t,z)}\nu(dz)\\
    &\hspace{1cm}+\int_{\R_+} \brac{\xi(t,X_t,\Lambda_t,w)\log \xi(t,X_t,\Lambda_t,w)+1-\xi(t,X_t,\Lambda_t,w)}\leb(dw)\bigg)dt\bigg].
\end{aligned}    
\end{equation}
\end{dfn}
\noindent As remarked at the end of the previous subsection, the minimization quantity in \eqref{Equation: stochastic control objective} is in fact equal to $\KL{\bP^{(u,\theta,\xi)}}{\bP^0}$ according to \eqref{Equation: KL divergence Girsanov representation}, whenever the stochastic integrals in the right-hand side of \eqref{eq: ito formula on logZ} are martingales under $\bP^{(u,\theta,\xi)}$. 

The SCP \eqref{Equation: stochastic control objective} has a terminal distribution contraint $\rho_T$. Instead of solving this SCP directly, we consider an ``unconstrained" variant of the SCP, which we may tackle by standard stochastic control techniques\footnote[2]{Expanding on ideas from the diffusion SBP case \cite{daipra1991stochastic,pra1990markov}, we use dynamic programming to approach the problem and we use Fleming's logarithmic transform \cite{fleming1982logarithmic} to solve the HJB equation.}. Namely, given $(u,\theta,\xi)\in\cU$, we define, for $t\in(0,T)$,
\begin{equation*}
\begin{aligned}
        F(t,u,\theta,\xi) :=\frac{1}{2}|u|^2+\int_{\R^\ell}\brac{(1-\theta)\log(1-\theta)+\theta}\nu(dz)
            +\int_{\R_+}\pran{\xi\log  \xi+1- \xi}\leb(dw),
\end{aligned} 
\end{equation*}
where $u=u(t,X_t,\Lambda_t)$,  $\theta=\theta(t,X_t,\Lambda_{t},z)$, and $\xi=\xi(t,X_{t},\Lambda_{t},w)$. We will write $F(t,u,\theta,\xi)(x,i)$ if $(X_t,\Lambda_t)$ are replaced by deterministic $(x,i)$ everywhere in the definition above. 
\begin{dfn}
Under the setting above, the unconstrained stochastic control formulation of the SBP \eqref{Definition:DynamicSBP} is to determine, for every $(t,x,i)\in [0,T] \times\R^d \times \rS$, 
\begin{equation}\label{Equation: Objective function - unconstrained}
\begin{aligned}
    &J^{*}(t,x,i)\\
    &\qquad:=\min_{(u,\theta,\xi)\in\cU} \bE_{\bP^{(u,\theta,\xi)}} \brac{\int_t^T F(r,u,\theta,\xi) dr - \log \g(X_T,\Lambda_T) \bigg| (X_t,\Lambda_t)=(x,i)},
\end{aligned}
\end{equation}
as well as the corresponding minimizer $(u^*,\theta^*,\xi^*)$.
\end{dfn}
\noindent We will solve the unconstrained SCP \eqref{Equation: Objective function - unconstrained} below, and, as we will prove later, the solution to \eqref{Equation: Objective function - unconstrained} coincides with the one sought in the constrained SCP \eqref{Definition: stochastic control SBP - general}. \\

The method of solving \eqref{Equation: Objective function - unconstrained} is standard. We will outline the steps here for the sake of completeness. 
We begin by writing down the Hamilton-Jacobi-Bellman (HJB) equation for \eqref{Equation: Objective function - unconstrained}, and that is, for every $(t,x,i)\in(0,T)\times\R^d \times \rS$, 
\begin{equation}\label{Equation: HJB}
    \begin{cases}
    \frac{\partial}{\partial t}J^*(t,x,i)+\inf_{(u,\theta,\xi)\in\cU}\set{L^{(u,\theta,\xi)}J^{*}(t,x,i)+F(t,u,\theta,\xi)(x,i)}=0,\\
    J^{*}(T,x,i)= -\log \g(x,i). \vspace{0.1cm}
    \end{cases}
\end{equation}
Let $(t,x,i)\in(0,T)\times\R^d \times \rS$ be fixed. When there is no ambiguity, we may omit writing $(t,x,i)$ as the argument of the controls $(u,\theta,\xi)$ and the coefficients $\sigma,\gamma,\alpha,
\beta$. Recall that $L^{(u,\theta,\xi)}$ is given by \eqref{Girsanov generator}.
Thus, expanding the terms inside the infimum expression of \eqref{Equation: HJB}, we obtain the following minimization problems:
\begin{equation*}
    \begin{cases}
        \displaystyle \inf_{u}\set{-\sigma u\cdot\nabla J^{*}+\frac{1}{2}\abs{u}^2},\\
        \displaystyle \inf_{\theta}\set{\int_{\R^\ell}\brac{-\theta(J^{*}(t,x+\gamma,i)-J^{*}(t,x,i))+\log(1-\theta)(1-\theta)+\theta}\nu(dz)},\\
        \displaystyle \inf_{\xi}\set{\int_{\R_+}\brac{(\xi-1)(J^{*}(t,x+\beta,i+\alpha)-J^{*}(t,x,i))+\xi\log \xi+1-\xi} \leb(dw)}.
    \end{cases}
\end{equation*}
Given $J^*(t,x,i)$ and $\nabla J^*(t,x,i)$, we can easily solve the above system for its unique global minimizers, and they are given by:
\begin{equation}\label{Equation: HJB minimizers}
   \begin{cases}
u^*(t,x,i)=\sigma^\top\nabla J^{*}(t,x,i),\\
\theta^*(t,x,i,z) = 1-\exp\pran{J^{*}(t,x,i)-J^{*}(t,x+\gamma(t,x,i,z),i)},\\
\xi^*(t,x,i,w)=\exp\pran{J^{*}(t,x,i)-J^{*}(t,x+\beta,i+\alpha)}.
\end{cases} 
\end{equation}

Next, we need to find the function $(t,x,i)\mapsto J^{*}(t,x,i)$ such that the equation in \eqref{Equation: HJB} is indeed satisfied with the above minimizers, meaning 
\begin{equation}\label{eq:HJB for J^* at minimizer}
    \pran{\frac{\partial}{\partial t}+L^{(u^*,\theta^*,\xi^*)}}J^{*}(t,x,i)+F(t,u^*,\theta^*,\xi^*)(x,i)=0
\end{equation}
for every $(t,x,i)\in(0,T)\times\R^d\times\rS$. To this end, we first note that, by \eqref{Girsanov generator},
\begin{equation*}
    \begin{aligned}
       \pran{\frac{\partial}{\partial t}+L^{(u^*,\theta^*,\xi^*)}}&J^{*}(t,x,i)\\
       &=\pran{\frac{\partial}{\partial t}+L}J^{*}(t,x,i)- (\sigma u^* \cdot \nabla J^*)(t,x,i)\\
       &\hspace{1.5cm}- \int_{\R^\ell}\theta^*(t,x,i,z)(J^*(t,x+\gamma,i)-J^*(t,x,i))\nu(dz)\\
       &\hspace{1.5cm}+\int_{\R_+}(\xi^*(t,x,i,w)-1)(J^*(t,x+\beta,i+\alpha)-J^*(t,x,i)) \leb(dw). 
    \end{aligned}
\end{equation*}
Thus, plugging in the values from \eqref{Equation: HJB minimizers}, \eqref{eq:HJB for J^* at minimizer} is reduced to
\begin{equation*}
\begin{aligned}
    &\pran{\frac{\partial}{\partial t}+L}J^{*}(t,x,i)\\
    &\qquad= \frac{1}{2}|\sigma^\top\nabla J^{*}|^2(t,x,i)\\& \hspace{1cm}+ \int_{\R^\ell} \pran{e^{J^{*}(t,x,i)-J^{*}(t,x+\gamma(t,x,i,z),i)}+J^{*}(t,x+\gamma(t,x,i,z),i)-J^{*}(t,x,i) -1}\,\nu(dz) \\
    &\hspace{2cm}+ \sum_{j\in\rS} Q_{ij}(t,x)\pran{e^{J^{*}(t,x,i)-J^{*}(t,\psi_{ij}(t,x),j)}+J^{*}(t,\psi_{ij}(t,x),j)-J^{*}(t,x,i)-1},
\end{aligned} 
\end{equation*}which, according to \eqref{-log phi identity}, has a solution $J^*(t,x,i)=-\log \varphi(t,x,i)$. Further, by \eqref{Definition: general varphi}, $J^*(T,x,i)=-\log\varphi(T,x,i)= -\log \g(x,i)$, which means that the boundary condition in \eqref{Equation: HJB} is also satisfied. In other words, taking $J^*$ to be $-\log \varphi$ in \eqref{Equation: HJB minimizers}, we have obtained a candidate solution $(u^*,\theta^*,\xi^*)$ to the SCP \eqref{Equation: Objective function - unconstrained}. We now prove it is indeed the optimizer.
\begin{thm}\label{Theorem : optimal controls}
    Suppose that Assumption \textup{\ref{asm : log varphi C12b}} holds. Then, the solutions to the SCPs \eqref{Equation: stochastic control objective} and \eqref{Equation: Objective function - unconstrained} coincide, both of which are given by:
    \begin{equation}\label{Equation: general theory - optimal controls}
        \begin{dcases}
        u^*(t,x,i)=-\sigma^\top\nabla\log \varphi(t,x,i)\\
        \theta^*(t,x,i,z) = 1-\frac{\varphi(t,x+\gamma(t,x,i,z),i)}{\varphi(t,x,i)}\\
        \xi^*(t,x,i,w) = \frac{\varphi(t,x+\beta(t,x,i,w),i+\alpha(t,x,i,w))}{\varphi(t,x,i)}
        \end{dcases}
    \end{equation}
for $(t,x,i)\in(0,T)\times\R^d\times\rS$, $z\in\R^\ell$ and $w\in\R_+$.

In particular, with $(u^*,\theta^*,\xi^*)$ given by \eqref{Equation: general theory - optimal controls}, $\bP^{(u^*,\theta^*,\xi^*)}=\widehat\bP$. That is, the solution to the SBP \eqref{Definition:DynamicSBP} can be obtained from solving the associated SCP \eqref{Equation: stochastic control objective} or \eqref{Equation: Objective function - unconstrained}. 
\end{thm}
\begin{proof}
Clearly, under Assumption \ref{asm : log varphi C12b}, the triple $(u^*,\theta^*,\xi^*)$ given by \eqref{Equation: general theory - optimal controls} satisfies the first two conditions in the definition of $\cU$. Moreover, by the remark below Assumption \ref{asm : log varphi C12b}, the third condition of $\cU$ also holds for $(u^*,\theta^*,\xi^*)$ due to the relation \eqref{eq: Ito formula applied to -log(varphi)} with $t=0,s=T$. Therefore, we have that $(u^*,\theta^*,\xi^*)\in\cU$. Since \eqref{Equation: general theory - optimal controls} coincides with \eqref{P-hat Girsanov triple}, the last statement of \cref{Theorem : optimal controls} follows immediately from \cref{Theorem: General P-hat}. In addition, we also have that $(u^*,\theta^*,\xi^*)\in\cU(\rho_T)$. 

We now prove that $(u^*,\theta^*,\xi^*)$ is indeed the minimizer for the unconstrained SCP \eqref{Equation: Objective function - unconstrained}.  It is easy to see that Assumption \ref{asm : log varphi C12b} endows $-\log\varphi$ with sufficient regularity so that the analysis performed above is valid. Therefore, we know that $-\log\varphi$ is a solution to the HJB equation \eqref{Equation: HJB}, which means that for every $(u,\theta,\xi)\in\cU$, \[\pran{\frac{\partial}{\partial t}+L^{(u,\theta,\xi)}}(-\log\varphi)(t,x,i) \geq -F(t,u,\theta,\xi)(x,i),\] with equality if and only if $(u,\theta,\xi)=(u^*,\theta^*,\xi^*)$. Meanwhile, with $\set{(X_t,\Lambda_t)}_{[0,T]}$ being the canonical process under $\bP^{(u,\theta,
\xi)}$, by the third condition in the definition of $\cU$, we have that for every $(t,x,i)\in[0,T]\times\R^d\times\rS$,
\begin{equation*}
\begin{aligned}
    &\bE_{\bP^{(u,\theta,\xi)}}\brac{-\log \g(X_T,\Lambda_T)|(X_t,\Lambda_t)=(x,i)}\\
    &\hspace{0.6cm}=-\log \varphi(t,x,i) \\
    &\hspace{1.5cm}+ \bE_{\bP^{(u,\theta,\xi)}}\brac{\int_t^{T} \pran{\frac{\partial}{\partial r}+L^{(u,\theta,\xi)}}(-\log \varphi)(r,X_r,\Lambda_r)dr \bigg| (X_t,\Lambda_t)=(x,i)}.
\end{aligned}
\end{equation*}
Combining this with the inequality above, we have that for every $(u,\theta,\xi)\in\cU$,
\begin{equation}\label{Equation : unconstrained objective lower bound}
    - \log \varphi(t,x,i)  \leq \bE_{\bP^{(u,\theta,\xi)}}\brac{\int_t^{T} F(r,u,\theta,\xi)dr -\log \g(X_T,\Lambda_T)\bigg|(X_t,\Lambda_t)=(x,i)},
\end{equation}
with equality if and only if $(u,\theta,\xi)=(u^*,\theta^*,\xi^*)$. Since the RHS of \eqref{Equation : unconstrained objective lower bound} is exactly the cost function in \eqref{Equation: Objective function - unconstrained}, we conclude that $(u^*,\theta^*,\xi^*)$ is indeed optimal for this SCP and $J^*(t,x,i)=- \log \varphi(t,x,i) $. 

Next, we turn our attention to the constrained SCP \eqref{Equation: stochastic control objective} and show that its solution is also given by  $(u^*,\theta^*,\xi^*)$. First, by taking $t=0$ in \eqref{Equation : unconstrained objective lower bound} and invoking the equality at $(u^*,\theta^*,\xi^*)$, we have that for every $(u,\theta,\xi)\in\cU$,
\begin{equation*}
\begin{aligned}
    \bE_{\bP^{(u^*,\theta^*,\xi^*)}}&\brac{\int_0^{T} F(r,u^*,\theta^*,\xi^*)dr -\log \g(X_T,\Lambda_T)}\\
    &\hspace{3cm}\leq \bE_{\bP^{(u,\theta,\xi)}}\brac{\int_0^{T} F(r,u,\theta,\xi)dr -\log \g(X_T,\Lambda_T)}.
\end{aligned}
\end{equation*}
Further, if $(u,\theta,\xi)\in \cU(\rho _T)$, i.e., $(X_T,\Lambda_T)$ has distribution $\rho_T$ under $\bP^{(u,\theta,\xi)}$, then 
\[
\bE_{\bP^{(u^*,\theta^*,\xi^*)}}[\log \g(X_T,\Lambda_T)]=\bE_{\bP^{(u,\theta,\xi)}}[\log \g(X_T,\Lambda_T)].
\]
Therefore, we can conclude that for every $(u,\theta,\xi)\in\cU(\rho_T)$,
\[
\bE_{\bP^{(u^*,\theta^*,\xi^*)}}\brac{\int_0^{T} F(r,u^*,\theta^*,\xi^*)dr}\leq \bE_{\bP^{(u,\theta,\xi)}}\brac{\int_0^{T} F(r,u,\theta,\xi)dr},
\]
meaning that $(u^*,\theta^*,\xi^*)$ is indeed the solution to the SCP \eqref{Equation: stochastic control objective}.
\end{proof}

\noindent \textbf{Remark. }We end this subsection on a remark regarding $\xi$, the control on regime switching. By \cref{Theorem : optimal controls}, the optimal $\xi^*$ for the SCPs \eqref{Equation: stochastic control objective} and \eqref{Equation: Objective function - unconstrained} is such that, for every $(t,x,i)\in(0,T)\times\R^d\times\rS$, $w\mapsto\xi^*(t,x,i,w)$ is constant on the interval $\Delta_{ij}(t,x)$ for every $j\in \rS$. In fact, this is no coincidence due to the following observation. Fix $(t,x)\in[0,T]\times\R^d$. Given a matrix $C=(C_{ij})_{i,j\in\rS}$ with $C_{ii}=0$ and $C_{ij}>0$ for $i\neq j$, if $(u,\theta,\xi)\in\cU$ is such that \[\int_{\Delta_{ij}(t,x)}\xi(t,x,i,w)\leb(dw)=Q_{ij}(t,x)C_{ij}\textup{ for }i,j\in \rS,\]then, by the definition of $L^{(u,\theta,\xi)}$ in \eqref{Girsanov generator}, the regime-switching rate governed by $\xi$ under $\bP^{(u,\theta,\xi)}$ is entirely determined by $C$. On the other hand, the function $y \mapsto y\log y +1 -y$ is convex, and hence Jensen's inequality yields that
\begin{align*}
    \int_{\Delta_{ij}(t,x)}\big(\xi\log\xi+1-\xi\big)(t,x,i,w)\,\leb(dw) \geq Q_{ij}(t,x)(C_{ij} \log C_{ij} +1 - C_{ij}),
\end{align*}
with equality if and only if $\xi(t,x,i,\cdot)\equiv C_{ij}$ on $\Delta_{ij}(t,x)$. Therefore, observing the cost functions in \eqref{Equation: stochastic control objective} and \eqref{Equation: Objective function - unconstrained}, we see that the concerned function $F(t,u,\theta,\xi)$ will always be minimized by $\xi$ that is constant on every interval $\Delta_{ij}(t,x)$. Such $\xi$ naturally takes a matrix form $(\xi_{ij}(t,x))_{i,j\in\rS}$ where $\xi_{ij}(t,x)$ is the constant value of $\xi$ on $\Delta_{ij}(t,x)$. With this form of $\xi$, we can observe more clearly the role of $\xi$ in controling the regime-switching rate: from regime $i$ to regime $j$, the switching rate under $\bP^{(u,\theta,\xi)}$ is exactly $\xi_{ij}(t,x) Q_{ij}(t,x)$.\\

By restricting $
\xi$ to the aforementioned matrix form, we now have an equivalent formulation of the SCP \eqref{Equation: stochastic control objective}: we aim to determine
\begin{equation*}
\begin{aligned}
    &(u^*,\theta^*,\xi^*)\\
    &\qquad:= \arg\min_{\big(u,\theta,\xi =\set{\xi_{ij}}_{i,j\in\rS}\big) \in \cU(\rho_T)} \bE_{\bP^{(u,\theta,\xi)}}\bigg[\int_0^T\bigg(\frac{1}{2}\abs{u(t,X_t,\Lambda_t)}^2\\
    &\hspace{1cm}+\int_{\R^\ell}\brac{(1-\theta(t,X_t,\Lambda_t,z))\log(1-\theta(t,X_t,\Lambda_t,z))+\theta(t,X_t,\Lambda_t,z)}\nu(dz)\\
    &\hspace{1.5cm}+\sum_{j\in\rS} Q_{\Lambda_t\,j}(t,X_t)\brac{\xi_{\Lambda_t\,j}(t,X_t)\log \xi_{\Lambda_t\,j}(t,X_t)+1-\xi_{\Lambda_t\,j}(t,X_t)}\bigg)dt\bigg].
\end{aligned}    
\end{equation*}
The solution $(u^*,\theta^*,\xi^*)$ to the above SCP is identical to the triple given in \eqref{Equation: general theory - optimal controls} with $\xi^*$ written in the matrix form as
\[
 \xi_{ij}^*(t,x) = \frac{\varphi(t,\psi_{ij}(t,x),j)}{\varphi(t,x,i)}\textup{ for }i,j\in\rS.
\]
\section{Under the Presence of a Transition Density Function}\label{Section: densities}
In this section, we will consider the situation when the solution to the SDE \eqref{Reference SDE general theory} admits a transition density function, in which case the Schr\"odinger bridge of the corresponding SBP \eqref{Definition:DynamicSBP} is more tractable, and additional information on its dynamics becomes accessible. 

Let $\bR,\rho_0,\rho_T$ be the same as before, $\bR_0,\bR_T$ be the initial and the terminal data, $\bR_{0T}$ be the joint endpoint distribution, and $\set{(X_t,\Lambda_t)}_{[0,T]}$ be the canonical process under $\bR$. Recall that we can write the Schr\"odinger bridge $\widehat\bP$ as  $\f(X_0,\Lambda_0)\g(X_T,\Lambda_T)\bR$, where $\f,\g$ are a solution to the static Schr\"odinger system \eqref{Equation: fg Schrodinger System with lambda} and $\f,\g$ are entirely determined by $\bR_{0T}$ and $\rho_0,\rho_T$. Throughout this section, in addition to the universal Assumptions \ref{Assumptions C} and \ref{Assumptions A}, we also impose the following ``density'' assumption.\\

\paragraph{\textbf{Assumption (D)}}\customlabel{Assumptions D}{\textbf{(D)}}
\begin{enumerate}[label=(\textbf{D\arabic*})]
\item \label{Assumption: transition densities} 
The reference measure $\bR$ admits a transition density function in the sense that, for every $(x,i)\in\R^d\times\rS$ and $0\leq t<s\leq T$, there exists a function  $(y,j)\in\R^d\times\rS\mapsto p_{ij}(t,x,s,y)$ such that for every Borel $B\subseteq\R^d$,
\begin{equation*}
    \bR(X_s\in B, \Lambda_s=j | (X_t, \Lambda_t)=(x,i))= \int_B p_{ij}(t,x,s,y)dy,
\end{equation*}
where $\set{(X_t,\Lambda_t)}_{[0,T]}$ is the canonical process under $\bR$. Moreover, for every $i,j\in\rS$ and $0\leq t<s\leq T$, the function $(x,y)\mapsto p_{ij}(t,x,s,y)$ is bounded on $\R^d\times\R^d$.
\item \label{Assumption: endpoint densities}  
$\bR_{0T},\rho_0,\rho_T$ are such that, for  every $i,j\in\rS$, \[
\int_{\R^d}\f(x,i)\bR_0(dx,i)<\infty\;\textup{ and }\sup_{y\in\textup{supp}(\bR_T(\cdot,j))}\g(y,j)<\infty. 
\]
That is, $\f(\cdot,i)$ is integrable under $\bR_0(\cdot,i)$ and $\g(\cdot,j)$ is bounded under $\bR_T(\cdot,j))$.
\end{enumerate}
Note that Assumption \ref{Assumption: transition densities} implies that the terminal data $\bR_T$ admits a density function
\[
\bR_T(y,j):=\sum_{i\in\rS}\int_{\R^d}p_{ij}(0,x,T,y)\bR_0(dx,i)\textup{ for }y\in\R^d,j\in\rS.
\]
Further, since $\rho_T$ is absolutely continuous with respect to $\bR_T$, we also have the existence of the density function $\rho_T(y,j)$. Finally, under Assumption \ref{Assumption: endpoint densities}, $\f(,i)$ is finite  $\bR_0(\cdot,i)$-a.e., and since the values of $\g(\cdot,j)$ outside of $\textup{supp}(\bR_T(\cdot,j))$ have no impact on the SBP, we may and will assume $\g(\cdot,j)$ is bounded on $\R^d$.\\
\subsection{The Schr\"odinger Potentials}
Under Assumption \ref{Assumptions D}, the definition \eqref{Definition: general varphi} of the function $\varphi$ can be rewritten as, for $(t,x,i)\in[0,T)\times\R^d\times\rS$,
\begin{equation}\label{Equation:phi-definition}
    \varphi(t,x,i):=\sum_{j\in\rS}\int_{\R^d}\g(y,j)p_{ij}(t,x,T,y)dy,
\end{equation}
with $\varphi(T,x,i):=\g(x,i)$. Note that $\varphi(t,x,i)$ is finite valued for every $i\in\rS$ and every ($\bR_0$-a.e., if $t=T$) $x\in\R^d$. We further define, for $(s,y,j)\in(0,T]\times\R^d \times \rS$,
\begin{equation}\label{Equation:phihat-definition}
    \widehat\varphi(s,y,j):=\sum_{i\in\rS}\int_{\R^d}\f(x,i)p_{ij}(0,x,s,y)\bR_0(dx,i),
\end{equation}
and $\widehat\varphi(0,dy,j):=\f(y,j)\bR_0(dy,j)$ as a measure on $\R^d$. Again, Assumption \ref{Assumptions D} guarantees the well-definedness of  $\widehat\varphi$. 

The pair of functions $(\varphi,\widehat\varphi)$ are called the \textit{Schr\"odinger potentials} associated with the static Schr\"odinger system \eqref{Equation: fg Schrodinger System with lambda}. As we will see below, while the definition of $(\varphi,\widehat\varphi)$ relies on the existence of the transition density function under $\bR$, $(\varphi,\widehat\varphi)$ characterizes the dynamics of the Schr\"odinger bridge in a fundamental way that is independent of the specific structure of the regime-switching jump diffusion which gives rise to $\bR$. We start with rewriting the static Schr\"odinger system \eqref{Equation: fg Schrodinger System with lambda} in terms of $(\varphi,\widehat\varphi)$.
\begin{thm}
    Suppose that Assumption \textup{\ref{Assumptions D}} holds. Let $(\varphi,\widehat\varphi)$ be defined as in \eqref{Equation:phi-definition} and \eqref{Equation:phihat-definition}. Then, the static Schr\"odinger system \eqref{Equation: fg Schrodinger System with lambda} can be rewritten as, for $\bR_0$-a.e. $x\in\R^d$, Lebesgue-a.e. $y\in\R^d$ and every $i,j\in \rS$, 
    \begin{equation}\label{Equation : phi-phihat Schrodinger System, general theory}
    \begin{dcases}
    \varphi(0,x,i)=\sum_{j\in\rS} \int_{\R^d} p_{ij}(0,x,T,y)\varphi(T,y,j)dy, \\
    \widehat\varphi(T,y,j)=\sum_{i\in\rS} \int_{\R^d} p_{ij}(0,x,T,y)\widehat\varphi(0,dx,i), \\
    \rho_0(dx,i)=\varphi(0,x,i)\widehat\varphi(0,dx,i),\\
    \rho_T(y,j)=\varphi(T,y,j)\widehat\varphi(T,y,j).
    \end{dcases}
\end{equation} 
We also refer to \eqref{Equation : phi-phihat Schrodinger System, general theory} as the $(\varphi,\widehat\varphi)$-Schr\"odinger system.
\end{thm}
\begin{proof}
The first two equations in 
\eqref{Equation : phi-phihat Schrodinger System, general theory} follow directly from the definitions of $\varphi$ and $\widehat\varphi$ and hence require no proof. We only need to verify the last two relations in \eqref{Equation : phi-phihat Schrodinger System, general theory}. Under Assumption \ref{Assumptions D}, the two equations in  \eqref{Equation: fg Schrodinger System with lambda} can be rewritten as, for $\bR_0$-a.e. $x\in\R^d$ and every $i\in\rS$, 
\begin{equation*}
    \begin{aligned}
        \rho_0(dx,i)&=\f(x,i)\,\Esub{\bR}{\g(X_T,\Lambda_T)|(X_0,\Lambda_0)=(x,i)}\,\bR_0(dx,i)\\
        &=\f(x,i)\pran{\sum_{j\in\rS} \int_{\R^d} p_{ij}(0,x,T,y)\g(y,j)dy}\bR_0(dx,i),
    \end{aligned}
\end{equation*}
and for Lebesgue-a.e. $y\in\R^d$ and every $j\in\rS$,
\begin{equation*}
\begin{aligned}
     \rho_T(y,j)&=\g(y,j)\,\Esub{\bR}{\f(X_0,\Lambda_0)|(X_T,\Lambda_T)=(y,j)}\cdot\bR_T(y,j)\\
     &=\g(y,j)\,\sum_{i\in\rS}\int_{\R^d}\f(x,i)p_{ij}(0,x,T,y)\bR_0(dx,i).
\end{aligned}
\end{equation*}
The desired relations follow immediately from plugging in \eqref{Equation:phi-definition} and \eqref{Equation:phihat-definition}.
\end{proof}

\noindent \textbf{Remark. }We remark that the third equation in \eqref{Equation : phi-phihat Schrodinger System, general theory} guarantees that $0<\varphi(0,x,i)<\infty$ for $\rho_0$-a.e. $x\in\R^d$ and every $i\in\rS$. The last equation in \eqref{Equation : phi-phihat Schrodinger System, general theory} implies that, for every $j\in\rS$, $\varphi(T,y,j)\widehat\varphi(T,y,j)<\infty$ for Lebesgue-a.e. (and hence $\rho_T$-a.e.) $y\in\R^d$. Meanwhile, $\varphi(T,y,j)\widehat\varphi(T,y,j)>0$ a.e. under $\rho_T$. We conclude that both $\varphi(T,y,j)$ and $\widehat\varphi(T,y,j)$ are positive and finite for $\rho_T$-a.e. $y\in\R^d$ and every $j\in\rS$. \\

Next, we turn our attention to the Schr\"odinger bridge $\widehat\bP$ solving the SBP \eqref{Definition:DynamicSBP} and recall the formulation \eqref{equation:bP^ as h-transf of bP^0} of $\widehat\bP$. With the help of the Schr\"odinger potentials $(\varphi,\widehat\varphi)$, we now have access to explicit expressions of the transition density function and the marginal density functions under $\widehat\bP$. 

\begin{thm}\label{Theorem: phi*phi-hat product density}
    Suppose that Assumption \textup{\ref{Assumptions D}} holds. Then, the Schr\"odinger bridge  $\widehat\bP$ admits a transition density function $\widehat p$ in the sense of \textup{\ref{Assumption: transition densities}}, and it is given by, for $x,y\in\R^d$, $i,j\in\rS$ and $0\leq t<s\leq T$,
\begin{equation}\label{eq:trans density p^ij}
    \widehat p_{ij}(t,x,s,y):=\begin{cases}
\frac{\varphi(s,y,j)}{\varphi(t,x,i)}p_{ij}(t,x,s,y)&\textup{ if }\varphi(t,x,i)>0,\\
0&\textup{ otherwise.}
    \end{cases}
\end{equation}
Further, for every $t\in(0,T]$, the marginal distribution $\widehat\bP_t$ possesses the density function
\begin{equation}\label{eq:marginal density of P-hat}
    \widehat\bP_t(x,i)=\varphi(t,x,i)\widehat\varphi(t,x,i)\;\textit{ for }(x,i)\in\R^d\times\rS.
\end{equation}
\end{thm}
\begin{proof}
Clearly, under Assumption \ref{Assumptions D}, the existence and the proposed formula \eqref{eq:trans density p^ij} of the transition density function under $\widehat\bP$ follows immediately from the result \eqref{eq: transition distribution under P^} proven in \cref{lem: transition distribution under bP^}. 
Therefore, we only need to examine the marginal distributions of $\widehat\bP$. Given $t\in(0,T]$, $i\in\rS$ and a Borel set $B\subseteq\R^d$, by \eqref{Equation:phihat-definition}, \eqref{Equation : phi-phihat Schrodinger System, general theory} and \eqref{eq:trans density p^ij}, we have that 
\begin{equation*}
    \begin{aligned}
      &\widehat\bP(X_t\in B, \Lambda_t=i)\\
      &\qquad=\sum_{j\in\rS}\int_B\int_{\R^d} \frac{\varphi(t,x,i)}{\varphi(0,y,j)}p_{ji}(0,y,t,x)\rho_0(dy,j)dx\\
      &\qquad=\int_B\varphi(t,x,i)\pran{\sum_{j\in\rS}\int_{\R^d} p_{ji}(0,y,t,x)\f(y,j)\bR_0(dy,j)}dx\\
      &\qquad=\int_B\varphi(t,x,i)\widehat \varphi(t,x,i)dx.
    \end{aligned}
\end{equation*}
We conclude that the marginal density function $\widehat\bP_t(x,i)$ exists and is given by \eqref{eq:marginal density of P-hat}.
\end{proof}
\noindent \textbf{Remark.} The statement of Theorem \ref{Theorem: phi*phi-hat product density} implies that, for every $t\in(0,T]$ and $i\in\rS$, $\varphi(t,x,i)\widehat\varphi(t,x,i)<\infty$ for Lebesgue-a.e. $x\in\R^d$, and both $\varphi(t,x,i)$ and $\widehat\varphi(t,x,i)$ are positive and finite for $\widehat\bP_t$-a.e. $x\in\R^d$. 

\subsection{The Backward and Forward PIDEs}
Continuing under the setting of the previous subsection, the goal of this part is to establish the partial integro-differential equations (PIDEs) that the Schr\"odinger potentials $(\varphi,\widehat\varphi)$ satisfy. These PIDEs should be viewed as the continuum-time extension of the $(\varphi,\widehat\varphi)$-Schr\"odinger system \eqref{Equation : phi-phihat Schrodinger System, general theory}, and they characterize the dynamics, both forward and backward in time, of the regime-switching jump diffusion process under the Schr\"odinger bridge $\widehat\bP$. In order to establish these PIDEs, we will work with a set of strengthened assumptions compared to those in the previous subsection. 

Let $\bR,\rho_0,\rho_T$ be the same as before, and let $\set{(X_t,\Lambda_t)}_{[0,T]}$ be the canonical process under $\bR$. Recall that the integro-differential operator $L$ defined in \eqref{Equation:Reference RS-generator} is the generator associated with the SDE \eqref{Reference SDE general theory}, describing the ``backward'' dynamics of the solution process $\set{(X_t,\Lambda_t)}_{[0,T]}$, and closely related to $\varphi$. We want to determine the counterpart for the ``forward'' dynamics. We expect that the forward generator is given by the \textit{adjoint} of $L$, but we first need to identify a proper notion of adjoint operator in the regime-switching setting. \\

\noindent \textbf{Definition. }For functions $f,g : [0,T]\times \R^d \times \rS \to \R$ such that, for every $i\in \rS$, $f(\cdot, i),g(\cdot,i)\in L^2([0,T]\times\R^d)$, consider the inner product 
\begin{equation}\label{Definition: bracket}
    \inn{f,g} := \sum_{i\in \rS}\int_0^T\int_{\R^d} f(t,x,i)g(t,x,i)dxdt.
\end{equation}
This is a combination of the $\ell^2(\rS)$ and $L^2([0,T]\times\R^d)$ inner products. 
We will say that $L^*$ is the \textit{adjoint operator} of $L$ if, for all $f,g$ of class $C^{1,2}_c$, 
\begin{equation}\label{eq: def of adjoint of L}
\inn{Lf,g}=\inn{f,L^*g}.
\end{equation}

We can derive the explicit form of $L^*$. For notational simplicity, we denote by $L_{i}$, $i\in\rS$, the generator of regime $i$ without regime switching, that is, for $f$ of class $C^{1,2}_c$,
\[
L_{i}f(t,x) := Lf(t,x,i) - \sum_{j\in\rS} Q_{ij}(t,x)\pran{f(t,\psi_{ij}(t,x),j)-f(t,x,i)}\textup{ for }(t,x)\in[0,T]\times\R^d.
\]
Hence, by \eqref{Definition: bracket}, for $f,g$ of class $C^{1,2}_c$,
\begin{equation}\label{Equation: bracket with split L0 and Q}
\begin{aligned}
\inn{Lf, g} = &\sum_{i\in\rS} \inn{L_{i}f, g}_{L^2([0,T]\times\R^d)} \\
&+  \sum_{i,j\in\rS} \int_0^T \int_{\R^d} Q_{ij}(t,x)\pran{f(t,\psi_{ij}(t,x),j)-f(t,x,i)}g(t,x,i)dx.
\end{aligned}
\end{equation}
From here, it is clear that in order to express $L^*$ as an integro-differential operator, we need to impose additional assumptions on the coefficients of the SDE \eqref{Reference SDE general theory}. \\

\paragraph{\textbf{Assumption ($\widehat{\textbf{C}}$)}}\customlabel{Assumptions E}{\textbf{($\widehat{\textbf{C}}$)}} Suppose that Assumption \ref{Assumptions C} holds for $b,\sigma,\gamma,\psi$. In addition,
\begin{enumerate}[label=(\textbf{$\widehat{\textbf{C}}$\arabic*})]
\item \label{asm: Kunita regularity}  as functions on $[0,T]\times\R^d\times\rS$, $b$ is of class $C^{1,1}$, $\sigma$ is of class $C^{1,2}$,  $\gamma(\cdot,z)$ is of class $C^{1,2}$ for every $z\in\R^\ell$; in addition, for every $(t,x,i)\in[0,T]\times\R^d\times\rS$, the functions $z\mapsto\gamma(t,x,i,z),\nabla\gamma(t,x,i,z),\nabla^2\gamma(t,x,i,z)$ are all continuously differentiable;
\item \label{asm:diffeomorphism phi} 
for every $(t,i,z)\in[0,T]\times\rS\times\R^\ell$, the map $x\mapsto\phi_{t,z,i}(x):=x+\gamma(t,x,i,z)$ is a $C^1$-diffeomorphism of $\R^d$, and if $\phi^{-1}_{t,z,i}(y)$ denotes the inverse of $y=\phi_{t,z,i}(x)$, then for every $(t,y,i)\in[0,T]\times\R^d\times\rS$, the function $z\mapsto\phi^{-1}_{t,z,i}(y)$ is continuously differentiable; 
\item \label{asm: psi diffeomorphism} for every $i,j \in \rS$ and $t\in[0,T]$, the map $x\mapsto\psi_{ij}(t,x)$ is a $C^1$-diffeomorphism of $\R^d$.
\end{enumerate}
For every $i\in\rS$, $L_i$ is a jump diffusion generator without regime switching. Under \ref{asm: Kunita regularity} and \ref{asm:diffeomorphism phi}, the adjoint operator (with respect to $L^2([0,T]\times\R^d)$) of $L_i$ has been explicitly determined (e.g., \cite[Theorem 5.1]{zlotchevski2024schrodinger} or \cite[Proposition 4.6.1]{kunita2019stochastic}) as, for $f$ of class $\in C^{1,2}_c$ and $(t,x)\in[0,T]\times\R^d$,
\begin{equation*}
\begin{aligned}
     L_i^*f(t,x)&:= -\nabla \cdot (bf) (t,x,i)+\frac{1}{2} \sum_{m,n}\frac{\partial^2[(\sigma\sigma^\top)_{mn}f]}{\partial x_m \partial x_n}(t,x,i) \\
    &\qquad+\int_{\R^\ell} \bigg(\det \nabla\phi_{t,z,i}^{-1}(x)f(t,\phi_{t,z,i}^{-1}(x),i)-f(t,x,i)\\ &\hspace{5cm}+\one_{|z|\leq 1}\nabla\cdot (\gamma f)(t,x,i,z) \bigg)\nu(dz),
\end{aligned}
\end{equation*}
where $\det \nabla\phi_{t,z,i}^{-1}$ is the Jacobian determinant of $\phi_{t,z,i}^{-1}$. We still need to derive the adjoint of the regime-switching component of $L$, for which we invoke  \ref{asm: psi diffeomorphism}. Given $f,g$ of class $C^{1,2}_c$, we use a change of variable $y=\psi_{ij}(t,x)$ to write the second term in \eqref{Equation: bracket with split L0 and Q} as
\begin{align*}
    \int_0^T\int_{\R^d} Q_{ij}&(t,x)\pran{f(t,\psi_{ij}(t,x),j)-f(t,x,i)}\cdot g(t,x,i)dxdt\\
    &=\int_0^T\int_{\R^d} f(t,y,j)\cdot\abs{\det\nabla\psi_{ij}^{-1}(t,y)} Q_{ij}(t,\psi_{ij}^{-1} (t,y))g(t,\psi_{ij}^{-1}(t,y),i)dydt\\
    &\hspace{5cm}-\int_0^T\int_{\R^d}f(t,x,i)\cdot Q_{ij}(t,x)g(t,x,i)dxdt,
\end{align*} 
where $\psi_{ij}^{-1}(t,y)$ is the inverse of $y=\psi_{ij}(t,x)$, and $\det \nabla\psi_{ij}^{-1}(t,y)$ is the Jacobian determinant of $\psi_{ij}^{-1}(t,y)$. We can get the explicit form of the desired adjoint operator by combining the above with \eqref{eq: def of adjoint of L} and \eqref{Equation: bracket with split L0 and Q}, and summing over $i,j\in\rS$. We summarize these derivations in a lemma. 
\begin{lem}\label{Lemma: adjoint of L}
    Given $L$ an integro-differential operator of the form \eqref{Equation:Reference RS-generator}, suppose that Assumption \textup{\ref{Assumptions E}} holds for the coefficients of $L$.
    Then, the adjoint operator $L^*$ in \eqref{eq: def of adjoint of L} is given by, for every $f$ of class $C^{1,2}_c$ and $(t,x,i)\in[0,T]\times\R^d\times\rS$, 
\begin{equation}\label{Equation: L* for rsjd}
\begin{aligned}
    L^*f(t,x,i) &=  L^*_{i}f(t,x)\\
    &\qquad+\sum_{j\in\rS} \bigg(\abs{\det\nabla\psi_{ji}^{-1}(t,x)}Q_{ji}(t,\psi_{ji}^{-1}(t,x))f(t,\psi_{ji}^{-1}(t,x),j)\\
    &\hspace{7cm}- Q_{ij}(t,x)f(t,x,i)\bigg).
\end{aligned}
\end{equation}
In particular, in the special case when $\psi_{ij}(t,x)=x$ for all $i,j\in\rS$ and $t\in[0,T]$ \footnote[2]{This case corresponds to the regime-switching model without hybrid jumps.}, 
\begin{equation*}
    L^*f(t,x,i) =  L^*_if(t,x)+\sum_{j\in\rS} \pran{Q_{ji}(t,x)f(t,x,j)- Q_{ij}(t,x)f(t,x,i)}.
\end{equation*}
\end{lem}

Going back to $(\varphi,\widehat\varphi)$, in order to study them in the setting of PIDE, we need to restrict ourselves to the situation when $\varphi$ and $\widehat\varphi$ are sufficiently regular. To this end, we impose a strengthening of Assumption \ref{Assumptions D}.\\

\paragraph{\textbf{Assumption ($\widehat{\textbf{D}}$)}}\customlabel{Assumptions D^}{($\widehat{\textbf{D}}$)} Suppose Assumption \ref{Assumptions D} holds for the reference measure $\bR$. In addition,  \ref{Assumption: transition densities} is strengthened to the following:
\begin{enumerate}[label=(\textbf{$\widehat{\textbf{D}}$\arabic*})]
\item \label{Assumption: D^ transition densities} 
The transition density function $p$ given by \ref{Assumption: transition densities} satisfies that, for every $i,j\in\rS$,
\begin{enumerate}
    \item $p_{ij}(t,x,T,y),\nabla_xp_{ij}(t,x,T,y),\nabla_x^2p_{ij}(t,x,T,y),\frac{\partial}{\partial t}p_{ij}(t,x,T,y)$ are all defined as continuous functions in $(t,x,y)\in(0,T)\times\R^d\times\R^d$, and are all bounded in $x\in\R^d$ with the bounds being Lebesgue integrable in $y\in\R^d$ for every $t\in(0,T)$.
    \item $p_{ij}(0,x,s,y),\nabla_yp_{ij}(0,x,s,y),\nabla_y^2p_{ij}(0,x,s,y),\frac{\partial}{\partial s}p_{ij}(0,x,s,y)$ are all defined as continuous functions in $(x,s,y)\in\R^d\times(0,T)\times\R^d$, and are all bounded in $(x,y)\in\R^d\times\R^d$ for every $s\in(0,T)$.
\end{enumerate}
\end{enumerate}
Note that \ref{Assumption: D^ transition densities} implies that, as functions on $(0,T)\times\R^d$, $(t,x)\mapsto p_{ij}(t,x,T,\cdot)$ and $(s,y)\mapsto p_{ij}(0,\cdot,s,y)$ are of class $C^{1,2}_b$. \\

\noindent \textbf{Remark. }Assumption \ref{Assumptions D^}, which consists of \ref{Assumption: D^ transition densities} and \ref{Assumption: endpoint densities}, provides a set of sufficient conditions for the regularity of $(\varphi,\widehat\varphi)$ required in later discussions. As we will see in the next section, we can readily find regime-switching jump diffusion models that satisfy Assumption \ref{Assumptions D^}. However, these conditions are not unique. Some variations of \ref{Assumption: D^ transition densities} and \ref{Assumption: endpoint densities} may also support the main theorem below. For example, if $p_{ij}$ in \ref{Assumption: D^ transition densities} satisfies the heat kernel estimate, then \ref{Assumption: endpoint densities} may be relaxed to $\f,\g$ having at most polynomial growth at infinity; or, if the function $(t,x,i)\mapsto p_{ij}(t,x,T,\cdot)$ is harmonic, then the $C^{1,2}_b$-requirement in \ref{Assumption: D^ transition densities}(a) on the function $(t,x)\mapsto p_{ij}(t,x,T,\cdot)$ may be replaced by a $C^{1,2}$-requirement. \\

\begin{thm}\label{Lemma: cL*p=0 for rsjd}
    Suppose that Assumptions \textup{\ref{Assumptions E}} and \textup{\ref{Assumptions D^}} hold. Let $L$ be the integro-differential operator in \eqref{Equation:Reference RS-generator}, $L^*$ be the adjoint operator given by \eqref{Equation: L* for rsjd},  and $\varphi,\widehat\varphi$ be defined as in \eqref{Equation:phi-definition} and \eqref{Equation:phihat-definition}. Then, $\varphi,\widehat\varphi$ are of class $C_b^{1,2}$. Moreover, $\varphi$ satisfies the \textup{Kolmogorov backward equation}:
    \begin{equation}\label{eq: backward PIDE for varphi}
        \pran{\frac{\partial}{\partial t}+L}\varphi(t,x,i)=0\textit{ for every }(t,x,i)\in(0,T)\times\R^d\times\rS,
    \end{equation}
    and 
    $\widehat\varphi$ satisfies the \textup{Kolmogorov forward equation}:
    \begin{equation}\label{eq: forward PIDE for varphi^}
        \pran{-\frac{\partial}{\partial s}+L^*}\widehat\varphi(s,y,j)=0\textit{ for every }(s,y,j)\in(0,T)\times\R^d\times\rS.
    \end{equation}
\end{thm}
\begin{proof}
Under Assumption \ref{Assumptions D^}, the fact that $\varphi,\widehat\varphi$ are of class $C_b^{1,2}$ follows from a simple application of the dominated convergence theorem, as all the concerned derivatives of $\varphi,\widehat\varphi$ can be passed onto $p_{ij}$ inside the integrals in \eqref{Equation:phi-definition} and \eqref{Equation:phihat-definition}. 

Let $\set{(X_t,\Lambda_t))}_{[0,T]}$ be the canonical process under $\bR$. By the equivalent definition \eqref{Definition: general varphi} of $\varphi$, we observe that for every $(t,x,i)\in(0,T)\times\R^d\times\rS$ and $h>0$ small, 
\begin{equation*}
  \begin{aligned}
     \varphi(t,x,i)&=\bE_{\bR}[\g(X_T,\Lambda_T)\mid (X_{t},\Lambda_{t})=(x,i)]\\
     &=\bE_{\bR}[\varphi(t+h,X_{t+h},\Lambda_{t+h})\mid (X_{t},\Lambda_{t})=(x,i)]\\
     &=\bE_{\bP^{t,x,i}}[\varphi(t+h,X_{t+h}, \Lambda_{t+h})]. 
  \end{aligned}  
\end{equation*}
Since $\varphi$ is of class $C^{1,2}_b$, by \eqref{eq: L as the infinitismal generator}, we have that
\begin{equation*}
    \begin{aligned}
        -\frac{\partial}{\partial t}\varphi(t,x,i)&=\lim_{h\searrow 0}\frac{1}{h}\pran{\varphi(t,x,i)-\varphi(t+h,x,i)}\\
        &=\lim_{h\searrow 0}\frac{1}{h}\pran{\bE_{\bP^{t,x,i}}[\varphi(t+h, X_{t+h}, \Lambda_{t+h})]-\varphi(t+h,x,i)}\\
        &=L\varphi(t,x,i),
    \end{aligned}
\end{equation*}
where in the last step we used the continuity of the coefficients of $L$ in the temporal variable. This is exactly the backward equation \eqref{eq: backward PIDE for varphi}.

For the forward equation, given the fact that $\widehat\varphi$ is of class $C^{1,2}_b$, it suffices to show that $\widehat\varphi$ satisfies \eqref{eq: forward PIDE for varphi^} in the weak sense with respect to the action $\inn{\cdot,\cdot}$ in \eqref{Definition: bracket}.
To this end, we take an arbitrary $f$ of class $C^{1,2}_c$ and apply It\^o's formula to $\set{f(t,X_t,\Lambda_t)}_{[0,T]}$ to get, for every $(s,x,i)\in(0,T)\times\R^d\times\rS$,
\begin{multline*}
    \bE_{\bR}[f(s,X_s,\Lambda_s)|(X_0,\Lambda_0)=(x,i)]\\=f(0,x,i)+\bE_{\bR}\brac{\int_0^s \pran{\frac{\partial}{\partial r}+L}f(r,X_{r},\Lambda_{r})dr\bigg|(X_0,\Lambda_0)=(x,i)}
\end{multline*}
which can be rewritten as
\begin{multline*}
\sum_{j\in\rS} \int_{\R^d} p_{ij}(0,x,s,y)f(s,y,j) dy \\
=f(0,x,i)+\sum_{j\in\rS} \int_{\R^d} \int_0^s \pran{\frac{\partial}{\partial r}+L}f(r,y,j) p_{ij}(0,x,r,y)dr dy.
\end{multline*}
Under Assumption \ref{Assumptions D^}, differentiating in $s$ both sides of the above and using the dominated convergence theorem, we arrive at
\[
\sum_{j\in\rS} \int_{\R^d} \frac{\partial}{\partial s} p_{ij}(0,x,s,y)\; f(s,y,j) dy = \sum_{j\in\rS} \int_{\R^d} Lf(s,y,j) \;p_{ij}(0,x,s,y)dy.
\]
In particular,
\begin{multline*}
    \sum_{i\in\rS}\int_{\R^d}\pran{\sum_{j\in\rS} \int_{\R^d} \frac{\partial}{\partial s} p_{ij}(0,x,s,y)\; f(s,y,j) dy}\f(x,i)\bR_0(dx,i) \\= \sum_{i\in\rS}\int_{\R^d}\pran{\sum_{j\in\rS} \int_{\R^d} Lf(s,y,j) \;p_{ij}(0,x,s,y)dy}\f(x,i)\bR_0(dx,i).
\end{multline*}
By \eqref{Equation:phihat-definition} and Fubini's theorem, we get 
\begin{equation*}
    \sum_{j\in\rS} \int_{\R^d} \frac{\partial}{\partial s} \widehat\varphi(s,y,j)\cdot f(s,y,j) dy=\sum_{j\in\rS} \int_{\R^d} Lf(s,y,j)\cdot \widehat\varphi(s,y,j)dy.
\end{equation*}
This is exactly
$\inn{-\frac{\partial}{\partial s}f,  \widehat\varphi} = \inn{Lf, \widehat\varphi}$. Since $f$ of class $C^{1,2}_c$ is arbitrary, we conclude that
$\pran{-\frac{\partial}{\partial s}+L^*}\widehat\varphi(s,y,j)=0$ in the weak sense, and hence in the strong sense as well.
\end{proof}
\noindent Note that the statement in Theorem \ref{Lemma: cL*p=0 for rsjd} on $\varphi$ implies that, under Assumption \ref{Assumptions D^}, $\varphi$ is in fact harmonic under $L$, as defined in Assumption \ref{Assumption: general harmonic varphi}. 

As a direct consequence of Theorem \ref{Lemma: cL*p=0 for rsjd}, we obtain the following ``dynamic'' formulation of the $(\varphi,\widehat\varphi)$-Schr\"odinger system. 
\begin{cor}\label{Theorem : dynamic schrodinger system - general}
Under the same assumptions as in Theorem \ref{Lemma: cL*p=0 for rsjd},
the pair $(\varphi,\widehat\varphi)$  solves the \textup{dynamic Schr\"odinger system}
\begin{equation}\label{Dynamic Schrodinger system - general}
    \begin{dcases}
        \pran{\frac{\partial}{\partial t}+L} \varphi(t,x,i) = 0 & \textit{ for }(t,x,i)\in(0,T)\times\R^d\times\rS,\\
        \pran{-\frac{\partial}{\partial s}+L^*} \widehat\varphi(s,y,j) = 0 &\textit{ for } (s,y,j)\in(0,T)\times\R^d\times\rS,\\
        \rho_0(dx,i) = \varphi(0,x,i)\widehat\varphi(0,dx,i) & \textit{  as measures on }\R^d\textit{ for every }i\in\rS,\\
        \rho_T(y,j) = \varphi(T,y,j)\widehat\varphi(T,y,j) & \textit{ for Lebesgue-a.e. }y\in\R^d\textit{ and every }j\in\rS.
    \end{dcases}
\end{equation}
\end{cor}

Finally, we will return to the Schr\"odinger bridge $\widehat\bP$ and revisit its dynamics in the context of PIDE. Recall that the generator $L_{\widehat\bP}$ corresponding to $\widehat\bP$ is also an integro-differential operator taking the form \eqref{Equation:Reference RS-generator} but with different coefficients. Thus, we can apply \cref{Lemma: adjoint of L} to obtain its adjoint operator $L_{\widehat\bP}^*$, and we expect that the density function of the marginal distribution $\widehat\bP_t$ satisfies the corresponding forward equation $\pran{-\frac{\partial}{\partial t}+L^*_{\widehat\bP}}\widehat\bP_t=0$. This is indeed the case, as we will verify below.
\begin{thm}\label{Theorem: forward equation for Phat density - general}
    Suppose that Assumptions \textup{\ref{Assumptions E}} and \textup{\ref{Assumptions D^}} hold. In addition, assume $\varphi(t,x,i)>0$ for every $(t,x,i)\in(0,T)\times\R^d\times\rS$. Then, the $\inn{\cdot,\cdot}$-adjoint operator of $L_{\widehat\bP}$, denoted by $L_{\widehat\bP}^*$, is given by, for $f$ of class $C^{1,2}_c$ and $(t,x,i)\in(0,T)\times\R^d\times\rS$, 
\begin{equation}\label{eqn: P-hat adjoint}
    \begin{aligned}
    L_{\widehat\bP}^*&f(t,x,i)\\
    =  &-\nabla \cdot (b^\varphi f)(t,x,i) +\frac{1}{2} \sum_{m,n}\frac{\partial^2[(\sigma\sigma^\top)_{mn}f]}{\partial x_m \partial x_n}(t,x,i) \\
    &+\int_{\R^\ell} \bigg[\det \nabla\phi_{t,z,i}^{-1}(x)f(t,\phi_{t,z,i}^{-1}(x),i)\frac{\varphi(t,x,i)}{\varphi(t,\phi_{t,z,i}^{-1}(x),i)}-f(t,x,i)\frac{\varphi(t,\phi_{t,z,i}(x),i)}{\varphi(t,x,i)} \\
    &\hspace{3cm}+\one_{|z|\leq 1}\nabla\cdot \pran{\gamma(t,x,i,z)f(t,x,i)\frac{\varphi(t,\phi_{t,z,i}(x),i)}{\varphi(t,x,i)}  }\bigg]\nu(dz) \\
    &+\sum_{j\in\rS}\bigg( \abs{\det\nabla\psi_{ji}^{-1}(t,x)}Q_{ji}(t,\psi_{ji}^{-1}(t,x))\frac{\varphi(t,x,i)}{\varphi(t,\psi_{ji}^{-1}(t,x),j)}f(t,\psi_{ij}^{-1}(t,x),j)\\
    &\hspace{6cm}- Q_{ij}(t,x)\frac{\varphi(t,\psi_{ij}(t,x),j)}{\varphi(t,x,i)}f(t,x,i)\bigg),
    \end{aligned}
\end{equation}
where $b^\varphi$ is as defined in \cref{Theorem: General P-hat}.

    Moreover, if $\widehat\bP_t(x,i)$, $(x,i)\in\R^d\times\rS$, is the density function of the marginal distribution $\widehat\bP_t$, then it satisfies the Kolmorov forward equation corresponding to $\widehat\bP$:
    \begin{equation*}
        \frac{\partial}{\partial t}\;\widehat\bP_t(x,i) = L^*_{\widehat\bP}\;\widehat\bP_t(x,i)\textit{ for }(t,x,i)\in(0,T)\times\R^d\times\rS.
    \end{equation*}
\end{thm}
\begin{proof}
Since $\varphi$ is of class $C_b^{1,2}$ under Assumption \ref{Assumptions D^}, the coefficients of $L_{\widehat\bP}$ given by \eqref{P-hat generator} still satisfy Assumption \ref{Assumptions E}, which, by \cref{Lemma: adjoint of L}, guarantees the existence of $L^*_{\widehat\bP}$. Then, the formula \eqref{eqn: P-hat adjoint} of $L^*_{\widehat\bP}$ can be derived following the same steps as in the proof of \cref{Lemma: adjoint of L}, with the only differences being the coefficients $\frac{\varphi(t,\phi_{t,z,i}(x),i)}{\varphi(t,x,i)}$ in front of the jump component and $\frac{\varphi(t,\psi_{ij}(x),j)}{\varphi(t,x,i)}$ in front of the regime-switching component. While the latter case is straightforward, we need to verify that the former change does not cause any trouble in the integral with respect to $\nu(dz)$. To this end, we set
 \[
    f^{\varphi}(t,x,i,z):=f(t,x,i)\frac{\varphi(t,x+\gamma(t,x,i,z),i)}{\varphi(t,x,i)},
\]
and note that the integral with respect to $\nu(dz)$ in \eqref{eqn: P-hat adjoint} can be rewritten as 
\begin{equation*}
    \begin{aligned}
        &\int_{\R^\ell} \bigg[\det \nabla\phi_{t,z,i}^{-1}(x)f^{\varphi}(t,\phi_{t,z,i}^{-1}(x),i,z)-f^{\varphi}(t,x,i,z) \\
    &\hspace{3.5cm}+\one_{|z|\leq 1}\nabla\cdot \pran{\gamma(t,x,i,z)f^\varphi(t,x,i,z)  }\bigg]\nu(dz). \\
    \end{aligned}
\end{equation*}
Since $f^\varphi$ is again of class $C_c^{1,2}$, the above integral is well defined. 

For the statement on the PIDE, we recall from \cref{Theorem: phi*phi-hat product density} that the density function of $\widehat\bP_t$ is given by 
\[
\widehat\bP_t(x,i)=\varphi(t,x,i)\widehat\varphi(t,x,i).
\]
Again, because $\varphi\cdot\widehat\varphi$ is of class $C^{1,2}_b$, it suffices to show that the desired PIDE is satisfied in the weak sense.
Take an arbitrary $f$ of class $C^{1,2}_c$. Under Assumption \ref{Assumptions D^} and the hypothesis that $\varphi>0$, we know that Assumption \ref{Assumption: general harmonic varphi} holds for $\varphi$. Therefore, Theorem \ref{Theorem: General P-hat} applies. In particular, by \eqref{h-transform generator identity} in \cref{Theorem: General P-hat}, we have that 
\begin{equation*}
    \inn{(\frac{\partial}{\partial t}+L_{\widehat\bP})f,\,\varphi\cdot\widehat\varphi}=\inn{\frac{1}{\varphi}(\frac{\partial}{\partial t}+L)(f \cdot\varphi),\,\varphi\cdot\widehat\varphi}=\inn{(\frac{\partial}{\partial t}+L)(f\cdot\varphi),\,\widehat\varphi}=0,
\end{equation*}
where the last equality is due to the fact that $(-\frac{\partial}{\partial t}+L^*)\widehat\varphi=0$.
Hence, we have proven the desired PIDE for $\varphi\cdot\widehat\varphi$.
\end{proof}

\subsection{The Fortet-Sinkhorn Algorithm}\label{Subsection: Algorithm}
In this subsection, we discuss an algorithmic approach toward finding the solution $\widehat\bP$ to the SBP \eqref{Definition:DynamicSBP}, for which it suffices to solve the static Schr\"odinger system \eqref{Equation: fg Schrodinger System with lambda}. We will only consider the situation when the initial data $\bR_0$ admits a density function $x\in\R^d\mapsto\bR_0(x,i)$ for every $i\in\rS$ (and hence $\rho_0$ also admits a density function $\rho_0(\cdot,i)$), in which case $\widehat\varphi(0,x,i)=\f(x,i)\bR_0(x,i)$ for Lebesgue-a.e. $x\in\R^d$ and every $i\in\rS$. Under Assumption \ref{Assumptions D}, solving \eqref{Equation: fg Schrodinger System with lambda} is equivalent to finding $(\varphi,\widehat\varphi)$ that satisfies the system \eqref{Equation : phi-phihat Schrodinger System, general theory}, where the third equation in \eqref{Equation : phi-phihat Schrodinger System, general theory} becomes \[
\rho_0(x,i)=\varphi(0,x,i)\widehat\varphi(0,x,i)\textup{ for Lebesgue-a.e. $x\in\R^d$ and every $i\in\rS$.}
\]

The existence and uniqueness of a solution to \eqref{Equation : phi-phihat Schrodinger System, general theory}, as well as how to find the solution,  is a long-studied problem in the history of the SBP \cite{fortet1940resolution,beurling1960automorphism,jamison1974reciprocal,chen2016entropic}. One approach is to construct recursively a sequence of functions that converges to a solution of \eqref{Equation : phi-phihat Schrodinger System, general theory}. This is the idea of the \textit{Fortet-Sinkhorn algorithm}, originally proposed in \cite{fortet1940resolution}, which consists of alternating between the four equations \eqref{Equation : phi-phihat Schrodinger System, general theory}. More specifically, for any function $f:\R^d\times\rS \to \R\setminus\set{0}$, we define the maps $\cD,\cE_{\rho_T},\cE_{\rho_0}$ on $f$ as
\begin{equation*}
\begin{aligned}
    \cD (f)(x,i)&:=\frac{1}{f(x,i)},\\
    \cE_{\rho_T}(f)(x,i)&:=\sum_{j\in\rS} \int_{\R^d} p_{ij}(0,x,T,y)\rho_T(y,j)f(y,j)dy,\\
    \cE_{\rho_0}(f)(y,j)&:=\sum_{i\in\rS} \int_{\R^d} p_{ij}(0,x,T,y)\rho_0(x,i)f(x,i)dx,\\   
\end{aligned}
\end{equation*}
for $x,y\in\R^d$ and $i,j\in\rS$. 

Clearly, if $(\varphi,\widehat\varphi)$ solves the Schr\"odinger system \eqref{Equation : phi-phihat Schrodinger System, general theory} and  $\varphi(0,\cdot),\widehat\varphi(T,\cdot)$ are positive everywhere, then we have the following chain relation:
\begin{equation*}
    {
    \widehat\varphi(T,\cdot)
    }
        \stackrel{\cD}{\mapsto}
    {
    \frac{1}{\widehat\varphi(T,\cdot)}
    }
        \stackrel{\cE_{\rho_T}}{\mapsto} 
    {
    \varphi(0,\cdot)
    }
        \stackrel{\cD}{\mapsto}
    {
    \frac{1}{\varphi(0,\cdot)}
    }
        \stackrel{\cE_{\rho_0}}{\mapsto} 
    {
    \widehat\varphi(T,\cdot)
    }
\end{equation*}
Therefore, if we define $\cC:= \cE_{\rho_0} \circ \cD \circ \cE_{\rho_T} \circ \cD$, then $\cC(\widehat\varphi(T,\cdot))=\widehat\varphi(T,\cdot)$. In other words, $\widehat\varphi(T,\cdot )$ is a fixed point under $\cC$. Under certain conditions, the algorithm that consists of
repeatedly applying the map $\cC$ will in fact converge to the unique fixed point $\widehat\varphi(T,\cdot)$, with \textit{any} starting function. This algorithm is well-understood and studied in many SBP contexts \cite{fortet1940resolution,chen2016entropic,chen2022most}, as well as beyond\footnote[2]{This algorithm is known under different names - Sinkhorn-Knopp algorithm, Iterative Proportional Fitting Procedure, RAS method, etc.}. Below, we formulate this convergence result rigorously in the regime-switching setting.
\begin{thm}\label{Theorem: algorithm convergence - General}
    Suppose that Assumption \textup{\ref{Assumptions D}} holds and $\bR_0$ admits a density function $x\in\R^d\mapsto\bR_0(x,i)$ for every $i\in\rS$. In addition, suppose that for every $i,j\in \rS$, the density functions $\rho_0(\cdot,i)$ and $\rho_T(\cdot,j)$ have compact support $A^i_0$ and $A^j_T$ respectively, and the transition density function $p_{ij}(0,\cdot,T,\cdot)$ is continuous and positive on $A^i_0 \times A^j_T$. For a function $f$ on $\R^d\times\rS$, we define 
    \begin{equation}\label{eq: conv of algorithm}
    \norm{f}^2_{(A_T^1,\cdots,A_T^{|\rS|})}:=\sum_{j\in\rS}\norm{f(\cdot,j)}_{L^2(A_T^j)}^2. 
    \end{equation}
    Then, for any starting function $f_0:\R^d\times\rS \to \R_+$ that is bounded with $\inf_{A_T^j}\,f_0(\cdot,j)>0$ for every $j\in\rS$, 
    \begin{equation*}
        \frac{\cC^k(f_0)}{\|\cC^k(f_0)\|_{(A_T^1,\cdots,A_T^{|\rS|})}}=:f_k\rightarrow f_\infty \textit{ as }k\rightarrow\infty \textit{ under }\norm{\cdot}_{(A_T^1,\cdots,A_T^{|\rS|})} \footnote{This convergence also holds with respect to the Hilbert metric; for more details, see \cite{chen2016entropic,chen2022most}.},
    \end{equation*}  
    where $f_\infty$ is the unique fixed point under $\cC$ with  $\norm{f_\infty}_{(A_T^1,\cdots,A_T^{|\rS|})}=1$. 
    
    Moreover, the quadruple
    \begin{equation*}
    \begin{dcases}
        \widehat\varphi(T,\cdot):=f_\infty,\\
        \varphi(T,\cdot):=\rho_T(\cdot)/\widehat\varphi(T,\cdot),\\
        \varphi(0,\cdot):=\cE_{\rho_T} \circ \cD(f_\infty),\\
        \widehat\varphi(0,\cdot):= \rho_0(\cdot)/\varphi(0,\cdot)
    \end{dcases}
    \end{equation*}
    is the unique solution of the Schr\"odinger system \eqref{Equation : phi-phihat Schrodinger System, general theory} (up to rescaling by a constant).    
\end{thm}
\begin{proof}
It suffices to rewrite the problem in terms of probability measures with Lebesgue densities on $\R^{d+1}$, where the convergence of the algorithm is well-understood. To this end, we consider the unit-length intervals $I_i:=[2i,2i+1]$ for $i\in \rS$, which are pairwise disjoint. In this proof we will denote an element in $\R^{d+1}$ by $\mathbf{x}=\db{x,x_{d+1}}$ where $x\in\R^d$ and $x_{d+1} \in \R$. For $\mathbf{x}=\db{x,x_{d+1}},\mathbf{y}=\db{y,y_{d+1}}\in\R^{d+1}$ and $i,j\in\rS$, we define
\begin{align*}     
    \tilde{\rho}_{0}(\mathbf{x})&:=\rho_0(x,i) \text{ if } x_{d+1} \in I_i, \;0\,\text{ otherwise}, \\
    \tilde{\rho}_{T}(\mathbf{y})&:=\rho_T(y,j) \text{ if } y_{d+1} \in I_j, \;0\,\text{ otherwise}, \\
    \tilde p(0,\mathbf{x},T,\mathbf{y}) &:= p_{ij}(0,x,T,y) \text{ if $x_{d+1} \in I_i$ and $y_{d+1} \in I_j$}, \;0\,\text{ otherwise}.  
\end{align*}
Constructed in this way, $\tilde{\rho}_{0}$ and $\tilde{\rho}_{T}$ are probability densities supported on $\tilde A_0:=\bigcup_{i\in\rS}A_0^i\times I_i$ and $\tilde A_T:=\bigcup_{j\in\rS}A_T^j\times I_j$ respectively, both of which are compact subsets of $\R^{d+1}$, and $\tilde p$ is everywhere continuous and positive on $\tilde{A}_0\times\tilde{A}_T$.
Now, consider the system
\begin{equation}\label{modified system}
    \begin{dcases}
    \Phi(0,\mathbf{x})=\int_{\tilde A_T} \tilde p(0,\mathbf{x},T,\mathbf{y})\Phi(T,\mathbf{y})d\mathbf{y}\\
    \widehat\Phi(T,\mathbf{y})=\int_{\tilde A_0}  \tilde p(0,\mathbf{x},T,\mathbf{y})\widehat\Phi(0,\mathbf{x})d\mathbf{x}\\
    \tilde{\rho}_{0}(\mathbf{x})=\Phi(0,\mathbf{x})\widehat{\Phi}(0,\mathbf{x})\\
    \tilde{\rho}_{T}(\mathbf{y})=\Phi(T,\mathbf{y})\widehat{\Phi}(T,\mathbf{y})
    \end{dcases}
\end{equation} 
for $(\mathbf{x,y})\in\tilde A_0\times\tilde A_T$. 

If we defined a mapping $\tilde \cC$ similar to $\cC$, as discussed at the beginning of this subsection, but with the chain relation given by the equations in the system \eqref{modified system}, then it is known from \cite[Proposition 4]{chen2016entropic} that the algorithm of repeatedly applying $\tilde\cC$ leads to a unique (up to a rescaling constant) solution to \eqref{modified system}. Namely, given any starting function $\tilde f_0(\mathbf{y}):=f_0(y)$, where $\mathbf{y}=\db{y,y_{d+1}}\in\tilde A_T$ and $f_0$ is as in the statement of the theorem, $\tilde\cC^k(\tilde f_0)/\big\|\tilde\cC^k(\tilde f_0)\big\|_{L^2(\tilde A_T)}$ converges to $\tilde f_\infty$ in $L^2(\tilde A_T)$, as $k\rightarrow\infty$, with $\tilde f_\infty$ being the unique fixed point under $\tilde\cC$ with $\|\tilde f_\infty\|_{L^2(\tilde A_T)}=1$. Thus, we obtain a unique solution to \eqref{modified system} by setting $\widehat\Phi(T,\cdot):=\tilde f_\infty $ and deriving $\widehat\Phi(0,\cdot),\Phi(T,\cdot),\Phi(0,\cdot)$ by the equations in \eqref{modified system} accordingly.

Observe that if a pair of functions $(\Phi,\widehat\Phi)$ solves \eqref{modified system}, then $\Phi$ and $\widehat\Phi$ must be constant in $x_{d+1}\in I_i$ and $y_{d+1}\in I_i$ respectively for every $i\in\rS$ due to the definitions of $\tilde p,\tilde{\rho}_{0},\tilde{\rho}_{T}$. Indeed, for every $i\in\rS$, if $x_{d+1}\in I_i$, then
\begin{align*}
    \Phi(0,\db{x,x_{d+1}})&=\int_{\tilde A_T} \tilde p(0,\mathbf{x},T,\mathbf{y})\Phi(T,\mathbf{y})d\mathbf{y}\\
    &=\sum_{j\in\rS} \int_{I_j} \int_{A_T^j} p_{ij}(0,x,T,y)\Phi(T,\db{y,y_{d+1}}) dy dy_{d+1} \\
    &= \sum_{j\in\rS}  \leb(I_j) \int_{A_T^j} p_{ij}(0,x,T,y)\Phi(T,\db{y,2j}) dy,  
\end{align*}
which is constant in $x_{d+1}\in I_i$, so the above equation reduces to
\begin{align*}
    \Phi(0,\db{x,2i}) &= \sum_{j\in\rS} \int_{A_T^j} p_{ij}(0,x,T,y)\Phi(T,\db{y,2j}) dy\\
    &=\sum_{j\in\rS} \int_{\R^d} p_{ij}(0,x,T,y)\Phi(T,\db{y,2j}) dy,
\end{align*}
where the last line is due to the fact that, by the fourth equation in \eqref{modified system},  $\Phi(T,\db{y,2j})=0$ if $y\notin A_T^j$ for every $j\in\rS$. By similarly rewriting the other equations in the system \eqref{modified system}, we see that \eqref{modified system} reduces to \eqref{Equation : phi-phihat Schrodinger System, general theory}, and the algorithm of iterating $\tilde\cC$ coincides with the one that iterates $\cC$, as for every $k\geq0$ and $j\in\rS$, $\tilde\cC^k(\tilde f_0)(\db{\cdot,y_{d+1}})=\cC^k(f_0)(\cdot,j)$ whenever $y_{d+1}\in I_j$, and hence $\big\|\tilde\cC^k(\tilde f_0)\big\|_{L^2(\tilde A_T)}=\norm{\cC^k(f_0)}_{(A^1_T,\cdots,A^d_T)}$.

Therefore, for $x,y\in\R^d$ and $i,j\in\rS$, by setting
\[
\begin{dcases}
    \varphi(0,x,i):=\Phi(0,\db{x,2i}),\;\varphi(T,x,i):=\Phi(T,\db{x,2i}),\\
    \widehat\varphi(0,y,j):=\widehat\Phi(0,\db{y,2j}),\;\widehat\varphi(T,y,j):=\widehat\Phi(T,\db{y,2j}),
\end{dcases}
\]
we have that $(\varphi,\widehat\varphi)$ is a solution to \eqref{Equation : phi-phihat Schrodinger System, general theory}. Moreover, $\tilde\cC^k(\tilde f_0)\rightarrow\widehat\Phi(T,\cdot)$ in $L^2(\tilde A_T)$ implies that  $\cC^k(f_0)\rightarrow\widehat\varphi(T,\cdot)$ under the norm defined in \eqref{eq: conv of algorithm}. Finally, the pair $(\varphi,\widehat\varphi)$ is unique up to rescaling because of the uniqueness of $(\Phi,\widehat\Phi)$.
\end{proof}

\section{An Example: The Unbalanced Schr\"odinger Bridge Problem}\label{Section: uSBP example}
In this section we present an application of the regime-switching jump-diffusion SBP investigated in the previous sections in the context of the \textit{unbalanced Schr\"odinger bridge problem} (uSBP). Unlike classical SBPs aiming at ``bridging'' probability distributions, the model of uSBP arises from the study of  stochastic particle systems with \textit{mass loss}, which results in the terminal distribution being a \textit{sub-probability measure} (i.e., with total mass less than $1$). Such a mechanism can be naturally interpreted in terms of regime switching, where we view the mass loss as the consequence of particles switching from the \textit{active} regime to the \textit{dead} regime. Based on this consideration, we will adopt the theory of regime-switching SBP developed above to examine a uSBP corresponding to a jump diffusion with killing. 

We will set up the model using the same framework of regime-switching jump diffusion as presented in \cref{Subsection: setup} but with only two regimes. Namely, we take $\rS=\set{\a,\d}$, with ``$\a$'' being the active regime and ``$\d$'' being the dead regime, and consider the jump diffusion determined by the SDE \eqref{SDE for Xt without hybrid jumps} on the Skorokhod space $\Omega=D([0,T];\R^d\times\set{\a,\d})$. We assume that, when restricted to regime ``$\d$'', all the SDE coefficients $b,\sigma,\gamma$ are constant zero, which means that there is no dynamic in the dead regime. Moreover, we require that, for every $(t,x)\in[0,T]\times\R^d$, 
\begin{enumerate}
    \item $Q_{12}(t,x)=V(t)$ for some function $V:[0,T]\rightarrow\R_+$ that provides the \textit{killing rate} at every moment $t\in[0,T]$;
    \item $Q_{21}(t,x)=0$, and hence switching from the dead regime to the active one is disabled;
    \item $\psi_{12}(t,x)=x$, i.e., there is no hybrid jump and hence a particle is kept at its location upon being killed.
\end{enumerate}
Under the above setting, the SDE \eqref{Reference SDE general theory} that characterizes the concerned regime-switching jump diffusion is reduced to
\begin{equation}\label{Reference SDE - uSBP}
\left\{
    \begin{aligned}
        dX_t&=b(t,X_{t},\Lambda_t)dt+\sigma(t,X_{t},\Lambda_t)dB_t + \int_{|z|\leq 1}\gamma(t,X_{t-},\Lambda_{t-} ,z)\tilde{N}(dt,dz)\\
        &\hspace{5.5cm} +\int_{|z|>1}\gamma(t,X_{t-},\Lambda_{t-} ,z)N(dt,dz) \\
    d\Lambda_t &= \int_{[0,V(t)]}\one_{\set{1}}(\Lambda_{t-})N_1(dt,dw).  
    \end{aligned}
\right.
\end{equation}
Let $\bR$ be a path measure on $\Omega$ whose canonical process $\set{(X_t,\Lambda_t)}_{[0,T]}$ is a strong Markov process satisfying \eqref{Reference SDE - uSBP}. Then, we can interpret the killing rate $V$ as $V(T)=0$, and for every $t\in[0,T)$,
    \begin{equation}\label{eq: def of V}
        V(t):=\lim_{\epsilon\searrow0}\frac{1}{\epsilon}\bR(\Lambda_{t+\epsilon}=\d\mid\Lambda_t=\a)\;\textup{ (assuming the limit exists)}.
    \end{equation}
Finally, we assume that $\Lambda_0\equiv\a$, which means that all the particles in the system start in the active regime.

Since $b,\sigma,\gamma$ are trivial functions within the dead regime, from now on, we will simply write $b(t,x,\a),\sigma(t,x,\a),\gamma(t,x,\a,z)$ as $b(t,x),\sigma(t,x),\gamma(t,x,z)$ respectively for $(t,x)\in[0,T]\times\R^d$ and $z\in\R^\ell$. Then, the integro-differential operator $L$ corresponding to \eqref{Reference SDE - uSBP} is given by, for $f$ of class $C^{1,2}_c$ and $(t,x)\in[0,T]\times\R^d$,
\begin{equation*}
\left\{
    \begin{aligned}
        Lf(t,x,\a)&=b(t,x)\cdot\nabla f(t,x,\a)+\frac{1}{2}\sum_{m,n}(\sigma\sigma^\top)_{mn}(t,x)\frac{\partial^2f}{\partial x_m \partial x_n}(t,x,\a)\\
        &\hspace{0.5cm}+\int_{\R^\ell} \bigg(f(t,x+\gamma(t,x,z),1)-f(t,x,1)\\
        &\hspace{4cm}-\one_{|z|\leq 1}\nabla f(t,x,1) \cdot \gamma(t,x,z)\bigg)\nu(dz) \\
        &\hspace{0.5cm}+ V(t)(f(t,x,\d)-f(t,x,\a)),\\
        Lf(t,x,\d) &= 0.
    \end{aligned}
\right.
\end{equation*}

For this part we will work under strong regularity assumptions on the coefficients $b,\sigma,\gamma$, which, as we will see below, guarantee that all the relevant assumptions from the previous sections are satisfied. Therefore, we will be able to put our general theory of SBP for regime-switching jump diffusions into practice, and demonstrate our techniques and results in a concrete uSBP setting. In particular, we adopt the following assumptions (inspired by the content from \cite{kunita2019stochastic}, Chapters 4-6) throughout this section:\\
\paragraph{\textbf{Assumptions (C$^\dagger$)}}\customlabel{Assumptions C-dagger}{\textbf{(C$^\dagger$)}}Suppose that Assumption \ref{Assumptions C} holds for the coefficients $b,\sigma,\gamma$ of the SDE \eqref{Reference SDE - uSBP}. In addition,
\begin{enumerate}
    \item as functions on $[0,T]\times\R^d$, $b$, $\sigma$, and $\gamma(\cdot,z)$ for every $z\in\R^\ell$, are of class $C^{1,\infty}_b$, and $\sigma\sigma^\top$ is uniformly positive definite; in addition, for every $(t,x)\in[0,T]\times\R^d$ and $k\in\mathbb{N}$, the function $z\mapsto\nabla^k\gamma(t,x,z)$ is twice continuously differentiable; 
    \item for every $(t,z)\in[0,T]\times\R^\ell$, the map $x\mapsto\phi_{t,z}(x):=x+\gamma(t,x,z)$ is a $C^\infty$-diffeomorphism of $\R^d$, and for every $(t,x)\in[0,T]\times\R^d$, the function $z\mapsto\phi^{-1}_{t,z}(x)$ is twice continuously differentiable;
    \item the intensity $\nu(dz)$ has a \textit{weak drift} and satisfies the \textit{order condition}\footnote[2]{A typical case is when $\nu(dz)$ is an $\alpha$-stable (or $\alpha$-stable-like) L\'evy measure for $0<\alpha<1$.} for some $\alpha\in(0,1)$ in the sense that, for every $\eta\in\R^\ell$, both $ \lim_{r\searrow0}\int_{r\leq\abs{z}\leq 1}(\eta\cdot z)\nu(dz)$ and $
    \lim_{r\searrow0}r^{\alpha-2}\int_{\abs{z}\leq r}(\eta\cdot z)^2\nu(dz)$ exist as finite values.
    \item The function $V:[0,T]\rightarrow\R_+$ is non-negative and continuous.\\
\end{enumerate}

Obviously, Assumption \ref{Assumptions C-dagger} is a strengthening of Assumption \ref{Assumptions E}.  Assumption \ref{Assumptions C-dagger} is also sufficient for Assumption \ref{Assumption: SDE} to hold for the SDE \eqref{Reference SDE - uSBP}. Let $\bR$ be a path measure that solves \eqref{Reference SDE - uSBP}. Since $\Lambda_0\equiv\a$, $\bR_0(\cdot,\a)$ is a probability measure on $\R^d$ and $\bR_0(\cdot,\d)$ is trivial. On the other hand, because $\psi_{12}(t,x)=x$ for every $(t,x)\in[0,T]\times\R^d$ and there is no dynamic in the dead regime, while $\bR_T(\cdot,\a)$ corresponds to the terminal position of the active particles, $\bR_T(\cdot,\d)$ is exactly the distribution of the ``killing'' location of the dead particles. Therefore, by solving the SBP as defined in \eqref{Definition:DynamicSBP} in such a setting, we tackle the uSBP by not only ``bridging'' the distributions of the surviving particles, but also controlling \textit{where} the dead particles were killed. Let us now formulate this uSBP rigorously.
\begin{dfn}[unbalanced Schr\"odinger bridge problem]
    Under Assumption \textup{\ref{Assumptions C-dagger}}, let $\bR$ be as described above. Given $\rho_0:=\bigg(\rho_0(\cdot,\a),0\bigg)$and $\rho_T:=\bigg(\rho_T(\cdot,\a),\rho_T(\cdot,\d)\bigg)$ two probability measures on $\R^d \times \set{\a,\d}$,
we aim to determine the path measure $\widehat\bP$ on $\Omega$ such that 
\begin{equation}\label{dfn: uSBP}
    \widehat\bP=\arg\min\set{\KL{\bP}{\bR} \text{ such that } \bP_0=\rho_0,\bP_T=\rho_T }.
\end{equation}
\end{dfn} 
\noindent We will simply write $\bR_0(\cdot,\a),\rho_0(\cdot,\a)$ as $\bR_0,\rho_0$ whenever there is no ambiguity.

Throughout \cref{Section: uSBP example}, we will also assume Assumption \ref{asm: A2 - leads to fg existence} holds for the above $\bR,\rho_0,\rho_T$, and hence Assumption \ref{Assumptions A} is again in effect for this part. In the upcoming subsections, we will establish that, under Assumption \ref{Assumptions C-dagger}, Assumptions \ref{Assumption: transition densities} and \ref{Assumption: general harmonic varphi} hold for the jump diffusion with killing under investigation, and hence Theorems \ref{Theorem: General P-hat} and \ref{Theorem: phi*phi-hat product density} apply to the uSBP in \eqref{dfn: uSBP}.
\subsubsection*{Background on uSBP} 
Stochastic dynamical systems with absorbing states or disappearing particles are ubiquitous in nature. These situations can range from radioactive decay \cite{zoia2008continuoustime} to cellular event triggers \cite{ham2024stochastic} to oceanic current tracers sinking \cite{chen2022most}. On the other hand, jump diffusion models with killing provide a mathematically rich framework for modeling such dynamical systems where the stochastic motion is described by a combination of drift, diffusion, and/or jump mechanics. The uSBP for this class of stochastic processes thus becomes an important tool for constructing models that are consistent with observational data. A ``modern'' version of the uSBP was introduced in \cite{chen2022most} as a mathematical model following Schr\"odinger's original quest\footnote[3]{This formulation was in contrast with the \textit{Feynman-Kac reweighting} framework previously adopted to study SBPs with mass loss (see the discussions in \cite[Section 5]{chen2022most} for details).}, where the reference measure $\bR$ is given by a diffusion process in $\R^d$ with killing, and the target distributions $\rho_0,\rho_T$ prescribe the distribution of the surviving particles and the total amount of mass lost by time $T$. 
This problem was approached by augmenting $\R^d$ to include a ``coffin state'' $\c$ to which the process would jump upon being killed, and then treated as a classic SBP for c\`adl\`ag paths in $\R^d \cup \set{\c}$. This version of the uSBP was then further studied in different contexts including neural networks 
\cite{pariset2023unbalanced}, entropic optimal transport \cite{lavenant2024mathematical}, and random walks on directed graphs \cite{eldesoukey2024schrodingers}. A variation of the uSBP was examined in \cite{eldesoukey2024excursion} and \cite{eldesoukey2025inferring} where the target at the terminal time is imposed only on the killing time and the killing location of the dead particles (but not the distribution of the surviving particles).

Compared with the version of uSBP in \cite{chen2022most}, we add two additional layers of complexity to our uSBP model \eqref{dfn: uSBP}: first, the reference measure $\bR$ is a \textit{jump-diffusion} with killing; second, the terminal target distribution $\rho_T$ not only encodes the distribution of the surviving particles, but also records \textit{where} the dead particles were killed. In a forthcoming paper, we will take one step further to study uSBPs under a diversity of settings. This will help shed light on the connections among existing uSBP works. 

\subsection{Jump Diffusion with Killing}
Following the same set-up as above and invoking Assumption \ref{Assumptions C-dagger}, let $\bR$ be the path measure whose canonical process $\set{(X_t,\Lambda_t)}_{[0,T]}$ is a strong Markov process satisfying the SDE \eqref{Reference SDE - uSBP}, and define $\tau:=\inf\set{t\geq0: \Lambda_t=\d}$. Then, $\tau$ is the regime-switching optional time, or, the \textit{killing time}. Therefore, $\set{X_t}_{[0,T]}$ under $\bR$ satisfies  the following SDE with killing: 
\begin{equation}\label{eqn: uSBP reference SDE Xt only}
\begin{aligned}
    dX_t&=b(t,X_t)dt +\sigma(t,X_t)dB_t\\
    &\hspace{0.5cm}+\int_{|z|\leq 1}\gamma(t,X_{t-},z)\tilde{N}(dt,dz) +\int_{|z|>1}\gamma(t,X_{t-},z)N(dt,dz),\;\textup{ for }0\leq t<\tau. \\
\end{aligned}
\end{equation}
Further, if we denote by $L_0$ the operator of the concerned jump diffusion without regime switching; that is, for every $f:[0,T]\times\R^d\rightarrow\R$ of class $C^{1,2}_c$ and $(t,x)\in(0,T)\times\R^d$,
\begin{equation}\label{eq:def of L0 usbp}
    \begin{aligned}
        L_0f(t,x)&=b(t,x)\cdot\nabla f(t,x)+\frac{1}{2}\sum_{m,n}(\sigma\sigma^\top)_{mn}(t,x)\frac{\partial^2f}{\partial x_m \partial x_n}(t,x)\\
        &\hspace{0.5cm}+\int_{\R^\ell} \bigg(f(t,x+\gamma(t,x,z))-f(t,x)-\one_{|z|\leq 1}\nabla f(t,x) \cdot \gamma(t,x,z)\bigg)\nu(dz),
    \end{aligned}
\end{equation}
then the generator of the jump diffusion with killing is $L_V:=L_0-V$. In addition, if $L_0^*$ denotes the adjoint operator of $L_0$ with respect to $L^2([0,T]\times\R^d)$, then $L_V^*:=L_0^*-V$ is the adjoint operator of $L_V$.
\begin{prop}\label{q existence and C12 uSBP}
    Suppose that Assumption \textup{\ref{Assumptions C-dagger}} holds for the coefficients of the SDE \eqref{Reference SDE - uSBP}. Let $\bR$, $\set{(X_t,\Lambda_t)}_{[0,T]}$, $\tau$, $L_V$ and $L_V^*$ be the same as above. Then, $\set{X_t:t<\tau}$ the jump diffusion with killing admits a transition density function in the sense that there exists a non-negative function $q(t,x,s,y)$ for $x,y\in\R^d$ and $0\leq t < s \leq T$ such that, for every Borel $B\subseteq \R^d$,
\begin{equation*}
    \bR(X_s \in B, s<\tau | X_t=x, t<\tau) = \int_B q(t,x,s,y)dy.
    \end{equation*}
Moreover, this transition density function $q$ satisfies that
\begin{enumerate}
    \item for every $0\leq t<s\leq T$, the function $(x,y)\in\R^d\times\R^d\mapsto q(t,x,s,y)$ is smooth with partial derivatives of all orders (including $q$ itself) being rapidly decaying in $y$ uniformly in $x$\footnote{That is, for any integer $k$ and any multi-index $\mathbf{i,j}$, $\sup_{x\in
        \R^d}\abs{\partial_y^{\mathbf{j}}\partial_x^{\mathbf{i}}q(t,x,s,y)}(1+|y|)^k$ converges to $0$ as $|y|\to\infty$.};
    \item for every $(s,y)\in(0,T]\times\R^d$, the function $(t,x)\in(0,s)\times\R^d\mapsto q(t,x,s,y)$ is of class $C^{1,\infty}_b$ 
    and satisfying  
\begin{equation}\label{eq: backward equation with killing}
\frac{\partial}{\partial t}q(t,x,s,y)=-L_Vq(\cdot,s,y) \,(t,x);
\end{equation}
    \item for every $(t,x)\in[0,T)\times\R^d$, the function $(s,y)\in(t,T)\times\R^d\mapsto q(t,x,s,y)$ is of class $C^{1,\infty}_b$
    and satisfying \begin{equation}\label{eq: forward equation with killing}
\frac{\partial}{\partial s}q(t,x,s,y)=L_V^*q(t,x,\cdot) \,(s,y).    
    \end{equation}
\end{enumerate}
\end{prop}
\begin{proof}
    Let us first consider the jump diffusion $\set{X_t}$ governed by the SDE \eqref{eqn: uSBP reference SDE Xt only} but without killing. Then, by the known results on jump diffusions in the smooth stochastic flow setting (\cite{kunita2019stochastic}, $\S$6.5-6.6), under Assumption \ref{Assumptions C-dagger}, $\set{X_t}$ admits a transition density function $q_0(t,x,s,y)$ for $x,y\in\R^d$ and $0\leq t<s\leq T$ that is sufficiently regular. In particular, we have that
    \begin{enumerate}
        \item for every $0\leq t<s\leq T$, the function $(x,y)\in\R^d\times\R^d\mapsto q_0(t,x,s,y)$ satisfies the first property in the statement of the proposition. 
        \item for every $(s,y)\in(0,T]\times\R^d$, the function $(t,x)\in (0,s)\times\R^d\mapsto q_0(t,x,s,y)$ satisfies the second property with $L_V$ in \eqref{eq: backward equation with killing} replaced by $L_0$;
        \item for every $(t,x)\in[0,T)\times\R^d$, the function  $(s,y)\in(t,T)\times\R^d\mapsto q_0(t,x,s,y)$ satisfies the third property with $L_V^*$ in \eqref{eq: forward equation with killing} replaced by $L_0^*$.
    \end{enumerate} 
Since the killing rate $V$ is homogeneous in the spatial variable, it is easy to see that $q$, the desired transition density function under killing, is given by
\begin{equation}\label{eq: formula of density under killing}
    q(t,x,s,y)=e^{-\int_t^sV(r)dr}\,q_0(t,x,s,y)
\end{equation}
for $x,y\in\R^d$ and $0\leq t<s\leq T$.
Given the aforementioned properties of $q_0$, and the assumption that $V$ is continuous on $[0,T]$, we immediately obtain all the properties of $q$ as claimed in the statement of the proposition.
\end{proof}
\noindent \textbf{Remark. }The existence of the transition density function $q(t,x,s,y)$ under killing can be extended to the case when the killing rate is spatial-dependent, i. e., $V=V(t,x)$ as a continuous and bounded function in $(t,x)\in[0,T]\times\R^d$. Indeed, $q(t,x,s,y)$ can be constructed from $q_0(t,x,s,y)$ following the standard Duhamel's method, but the regularity theory of $q(t,x,s,y)$ is more involved in such a case. For simplicity, we only consider the spatial-homogeneous killing rate in our uSBP example. However, if $q_0(t,x,s,y)$ is known to satisfy some refined estimates (e.g., certain heat kernel estimates), then the results below also extend to the case with a spatial-dependent killing rate.\\

Next, we re-connect the jump diffusion with killing with the regime-switching model in \eqref{Reference SDE - uSBP}, and re-interpret $q(t,x,s,y)$ found above in terms of the regime-switching process.

\begin{cor}\label{cor:p_11 p_12 satisfying D^1}
    Suppose that Assumption \textup{\ref{Assumptions C-dagger}} holds for the coefficients of the SDE \eqref{Reference SDE - uSBP}. Let $\bR$, $\set{(X_t,\Lambda_t)}_{[0,T]}$, $\tau$ and $q(t,x,s,y)$ be the same as in Proposition \ref{q existence and C12 uSBP}. Then, Assumption \textup{\ref{Assumption: D^ transition densities}} holds in the sense that   $\set{(X_t,\Lambda_t)}_{[0,T]}$ admits a transition density function $p_{ij}$ for $i=\a,j\in\set{\a,\d}$ as \[
    p_{11}(t,x,s,y):=q(t,x,s,y)\;\textit{ and }\;p_{12}(t,x,s,y) := \int_t^s V(r)q(t,x,r,y)dr
    \]
    for $x,y\in\R^d$ and $0\leq t<s\leq T$, and $p_{11},p_{12}$ satisfies the regularity conditions in \textup{\ref{Assumption: D^ transition densities}}.
\end{cor}
\begin{proof}
It is clear from the set-up of the model \eqref{Reference SDE - uSBP} that for every $t\in[0,T]$, $t<\tau$ is equivalent to $\Lambda_t=\a$, and hence the dynamics of $\set{X_t:t<\tau}$ is exactly that of the particles in the active regime, which means that $p_{11}=q$ is the transition density function within regime ``$\a$''. 

Next, for any $0 \leq t < s \leq T$, if $\Lambda_t=\a$ and $\Lambda_s=\d$, then $\tau\in(s,t]$, and hence by the strong Markov property of $\set{X_t}$ and the definition \eqref{eq: def of V} of the killing rate $V$, we have that for every $x\in\R^d$ and Borel $B\subseteq\R^d$,

\begin{equation*}
    \begin{aligned}
     \bR(X_s\in B&,\Lambda_s=\d| (X_t,\Lambda_t)=(x,\a))\\
     &=\int_t^s \lim_{\epsilon\searrow0}\frac{1}{\epsilon}\bR(X_s\in B,\Lambda_s=\d,\tau\in(r,r+\epsilon]| (X_t,\Lambda_t)=(x,\a))dr\\
     &=\int_t^s \lim_{\epsilon\searrow0}\frac{1}{\epsilon}\bR(\Lambda_{r+\epsilon}=\d,X_r\in B,\Lambda_r=\a| (X_t,\Lambda_t)=(x,\a))dr\\
     &=\int_t^s \int_B \lim_{\epsilon\searrow0}\frac{1}{\epsilon}\bR(\Lambda_{r+\epsilon}=\d| (X_r,\Lambda_r)=(y,\a))q(t,x,r,y)dydr\\
     &=\int_B\int_t^s V(r)q(t,x,r,y)drdy.
    \end{aligned}
\end{equation*}
From here we conclude that the transition density function $p_{12}$ exists and takes the desired form in the statement. 

The regularity properties of $p_{11},p_{12}$ follow trivially from Proposition \ref{q existence and C12 uSBP}. 
\end{proof}
\subsection{Setting up the uSBP}
Following the setting and the discussions from the previous subsection, we are ready to apply the general theory established in Sections \ref{Section: General Theory}-\ref{Section: densities} to the uSBP \eqref{dfn: uSBP} whenever $\bR_{0T},\rho_0,\rho_T$ satisfy Assumptions \ref{asm: A2 - leads to fg existence} and \ref{Assumption: endpoint densities}. For pedagogical purposes, we will choose a particular configuration of $\bR_{0T},\rho_0,\rho_T$, which yields simpler steps and explicit expressions in the study of the uSBP. 

For the rest of this section, we assume that $\rho_0=\bR_0=\delta_{x_0}$, the Dirac delta distribution at some fixed point $x_0\in\R^d$. Then, according to \cref{cor:p_11 p_12 satisfying D^1}, the density function of $\bR_T$ is given by, for every $y\in\R^d$, \begin{equation}\label{eq: explicit formula of R_T usbp}
    \bR_T(y,\a)=q(0,x_0,T,y) \;\textup{ and }\;\bR_T(y,\d)=\int_0^TV(s)q(0,x_0,s,y)ds.
\end{equation}
Let us re-examine Assumption \ref{asm: A2 - leads to fg existence} with the help of the above explicit expressions. First, it is trivally true in this case that $\bR_{0T}\ll\bR_0\otimes\bR_T$ as both of them are $\delta_{x_0}\otimes\bR_T$.
Further, with $\rho_0=\delta_{x_0}$, it is clear that $\KL{\rho_0\otimes\rho_T}{\bR_{0T}}=\KL{\rho_T}{\bR_T}$. Hence, to have Assumption \ref{asm: A2 - leads to fg existence}, we need to make $\KL{\rho_T}{\bR_T}<\infty$ which, by \eqref{eq: explicit formula of R_T usbp}, expands to
\begin{equation}\label{eq: rewrite A2 for usbp}
\int_{\R^d}\log\frac{\rho_T(y,\a)}{q(0,x_0,T,y)}\rho_T(y,\a)dy+\int_{\R^d}\log\frac{\rho_T(y,\d)}{\int_0^TV(s)q(0,x_0,s,y)ds}\rho_T(y,\d)dy<\infty.
\end{equation}

From now on, we assume $\rho_T$ satisfies \eqref{eq: rewrite A2 for usbp}, and hence Assumption \ref{asm: A2 - leads to fg existence} is satisfied in this case. Next, we turn to Assumption \ref{Assumption: endpoint densities} and assume that $(\f,\g)$ is a solution to the Schr\"odinger system \eqref{Equation: fg Schrodinger System with lambda}. Under the current setting, the two equations in  \eqref{Equation: fg Schrodinger System with lambda} become 
\begin{equation}\label{eq: rewrite schrodinger sys usbp}
    \begin{cases}
        \f(x_0,\a)\bigg(\int_{\R^d}\g(y,\a)\bR_T(y,\a)dy+\int_{\R^d}\g(y,\d)\bR_T(y,\d)dy\bigg)=1,\\
        \g(y,\a)\f(x_0,\a)=\frac{\rho_T(y,\a)}{\bR_T(y,\a)},\;\g(y,\d)\f(x_0,\a)=\frac{\rho_T(y,\d)}{\bR_T(y,\d)},\,\textup{ for $\bR_T$-a.e. } y\in\R^d.
    \end{cases}
\end{equation}
Without loss of generality, we take $\f\equiv1$, which trivially satisfies the condition $\f\in L^1(\bR_0)$ required in \ref{Assumption: endpoint densities}. Then, combining \eqref{eq: rewrite schrodinger sys usbp} and \eqref{eq: explicit formula of R_T usbp}, we have the explicit formula of $\g$: 
\begin{equation}\label{eq: def of g in usbp}
\begin{dcases}
   \g(y,\a):=\frac{\rho_T(y,\a)}{q(0,x_0,T,y)},\; y\in S_\a:=\set{y:\,q(0,x_0,T,y)>0},\\
   \g(y,\d):=\frac{\rho_T(y,\d)}{\int_0^TV(s)q(0,x_0,s,y)ds},\;y\in S_\d:=\set{y:\,\int_0^T V(s)q(0,x_0,s,y)ds>0}.
\end{dcases}
\end{equation}
It is clear from \eqref{eq: explicit formula of R_T usbp} that  $S_\a=\textup{supp}(\bR_T(\cdot,\a))$ and $S_\d=\textup{supp}(\bR_T(\cdot,\d))$, and hence the definition \eqref{eq: def of g in usbp} has sufficiently determined $\g$.

Finally, the condition $\g$ being bounded under $\bR_T$ in \ref{Assumption: endpoint densities} corresponds to
\begin{equation}\label{eq: rewrite D2 for usbp}
    \sup_{y\in S_\a}\frac{\rho_T(y,\a)}{q(0,x_0,T,y)}+\sup_{y\in S_\d}\frac{\rho_T(y,\d)}{\int_0^TV(s)q(0,x_0,s,y)ds}<\infty.
\end{equation}
Therefore, by further requesting $\rho_T$ to satisfy \eqref{eq: rewrite D2 for usbp}, we have recovered Assumption \ref{Assumption: endpoint densities}.
\subsection{Solution to the uSBP}
We continue with the set-up above, i.e., suppose Assumption \ref{Assumptions C-dagger} holds for the coefficients of the SDE \eqref{Reference SDE - uSBP}, $q$ is the transition density function found in \cref{q existence and C12 uSBP}, $\rho_0=\bR_0=\delta_{x_0}$ for some $x_0\in\R^d$, $\rho_T$ satisfies \eqref{eq: rewrite A2 for usbp} and \eqref{eq: rewrite D2 for usbp}, and $\f\equiv1$ and $\g$ is defined by \eqref{eq: def of g in usbp}. Then, Assumptions \ref{Assumptions A} and \ref{Assumptions D^} are fulfilled under this set-up. 

By the general theory of SBP, the Schr\"odinger bridge $\widehat\bP$ in this case, which is also the solution to the uSBP \eqref{dfn: uSBP}, is given by 
\begin{equation}\label{eq: schrodinger bridge P^ usbp}
   \widehat\bP:=\bigg(\one_{\set{\Lambda_T=\a}}\,\frac{\rho_T(X_T,\a)}{q(0,x_0,T,X_T)}+\one_{\set{\Lambda_T=\d}}\,\frac{\rho_T(X_T,\d)}{\int_0^TV(s)q(0,x_0,s,X_T)ds}\bigg)\bR. 
\end{equation}
We now apply the results from the previous sections to study various aspects of the dynamics under $\widehat\bP$. 

\subsubsection*{The Schr\"odinger Potentials of the uSBP} We begin with writing down the Schr\"odinger potentials $(\varphi,\widehat\varphi)$. By \eqref{Definition: general varphi} and \cref{cor:p_11 p_12 satisfying D^1}, we obtain the explicit formula for $\varphi$: for $(t,x)\in[0,T)\times\R^d$,
\begin{equation}\label{eq: varphi-usbp}
\begin{cases}
        \varphi(t,x,\a)
        := \int_{S_\a}\frac{q(t,x,T,y)}{q(0,x_0,T,y)}\rho_T(y,\a)dy+\int_{S_\d}\frac{\int_t^T V(r)q(t,x,r,y)dr}{\int_0^TV(s)q(0,x_0,s,y)ds}\rho_T(y,\d)dy,\\
        \varphi(t,x,\d)
        :=\frac{\rho_T(x,\d)}{\int_0^TV(s)q(0,x_0,s,x)ds}\textup{ for }x\in S_\d,\;0\textup{  otherwise},
\end{cases}
\end{equation}
with $\varphi(T,\cdot)=\g$.
On the other hand, using \eqref{Equation:phihat-definition}, we can also express $\widehat\varphi$ as, for $(s,y)\in(0,T]\times\R^d$,
\begin{equation}\label{eq:varphihat-usbp}
    \begin{cases}
        \widehat\varphi(s,y,\a)=q(0,x_0,s,y),\\
        \widehat\varphi(s,y,\d)=\int_0^sV(r)q(0,x_0,r,y)dr,
    \end{cases}
\end{equation}
with $\widehat\varphi(0,\cdot,\a)=\delta_{x_0}$ and $\widehat\varphi(0,\cdot,\d)\equiv0$. 

It is easy to verify that $(\varphi,\widehat\varphi)$ defined above is indeed a solution to the system \eqref{Equation : phi-phihat Schrodinger System, general theory}, and the solution to the uSBP \eqref{dfn: uSBP} can be rewritten as $\widehat\bP=\frac{\varphi(T,X_T,\Lambda_T)}{\varphi(0,X_0,\Lambda_0)}\bR$. In addition, by \cref{q existence and C12 uSBP}, the function $(t,x)\mapsto\varphi(t,x,\a)$ is of class $C^{1,2}$ and for every $t\in[0,T)$, the functions $x\in \R^d\mapsto\varphi(t,x,\a),\nabla\varphi(t,x,\a),\nabla^2\varphi(t,x,\a)$ are continuous and bounded. Similarly, $(s,y)\mapsto\widehat\varphi(s,y,\a)$ is of class $C^{1,2}$ and $\widehat\varphi,\nabla\widehat\varphi,\nabla^2\widehat\varphi$ are continuous and bounded in $y\in\R^d$ for every $s\in(0,T]$.
\subsubsection*{The Schr\"odinger Bridge Dynamics of the uSBP}
Since Assumption  \ref{Assumptions D^} holds in the current set-up, we can readily apply \cref{Lemma: cL*p=0 for rsjd} to write down the forward and the backward PIDEs satisfied by $\varphi,\widehat\varphi$. However, with the explicit formulas of $\varphi,\widehat\varphi$ above, it is also easy to verify the relevant PIDEs directly. Indeed, by applying \cref{q existence and C12 uSBP} to \eqref{eq: varphi-usbp}, we get
\begin{equation*}
    \begin{aligned}
        \frac{\partial}{\partial t}\varphi(t,x,\a)&=\int_{\R^d}\frac{-L_Vq(\cdot,T,y)\,(t,x)}{q(0,x_0,T,y)}\rho_T(y,\a)dy-V(t)\frac{\rho_T(x,\d)}{\int_0^Tq(0,x_0,s,x)ds}\\
        &\hspace{3cm}+\int_{\R^d}\frac{\int_t^T V(r)(-L_Vq(\cdot,r,y))dr\,(t,x)}{\int_0^TV(s)q(0,x_0,s,y)ds}\rho_T(y,\d)dy\\
        &=(-L_V)\varphi(t,x,\a)-V(t)\varphi(t,x,\d).
    \end{aligned}
\end{equation*}
Recalling that $L_V=L_0-V$ with $L_0$ as in \eqref{eq:def of L0 usbp}, we arrive at \[
\frac{\partial}{\partial t}\varphi(t,x,\a)=-L_0\varphi(t,x,\a)+V(t)(\varphi(t,x,\a)-\varphi(t,x,\d)).
\]
Similarly, by \eqref{eq:varphihat-usbp} and, again, \cref{q existence and C12 uSBP}, we derive the PIDE satisfied by $\widehat\varphi$:
\begin{equation*}
    \begin{aligned}
        &\frac{\partial}{\partial s}\widehat\varphi(s,y,\a)=L_V^*\widehat\varphi(s,y,\a)=L^*_0\widehat\varphi(s,y,\a)-V(s)\widehat\varphi(s,y,\a),\\
        &\frac{\partial}{\partial s}\widehat\varphi(s,y,\d)=V(s)q(0,x_0,s,y)=V(s)\widehat\varphi(s,y,\a),
    \end{aligned}
\end{equation*}
where $L^*_0$ is the adjoint operator of $L_0$ with respect to $L^2([0,T]\times\R^d)$.\\

We summarize the above PIDEs into the following system:\\

\noindent \textbf{Unbalanced Dynamic Schr\"odinger System:}
\begin{equation*}
        \begin{dcases}
    \pran{\frac{\partial }{\partial t}+L_0}\varphi(t,x,\a)+V(t)(\varphi(t,x,\d)-\varphi(t,x,\a))=0,\;&\frac{\partial \varphi}{\partial t}(t,x,\d)=0.\\
    \pran{-\frac{\partial }{\partial s}+ L^*_0}\widehat\varphi(s,y,\a)-V(s)\widehat\varphi(s,y,\a)=0,\;&\frac{\partial \widehat\varphi}{\partial s}(s,y,\d)-V(s)\widehat\varphi(s,y,\a)=0.\\
    \varphi(0,x,\a)\widehat\varphi(0,dx,\a)=\rho_0(dx)=\delta_{x_0}(dx),\;&\varphi(0,x,\d)\widehat\varphi(0,x,\d)=0.\\
    \varphi(T,y,\a)\widehat\varphi(T,y,\a)=\rho_T(y,\a),\;&\varphi(T,y,\d)\widehat\varphi(T,y,\d)=\rho_T(y,\d).
    \end{dcases}
\end{equation*}
In the above system, the first two lines hold for $(t,x),(s,y)\in(0,T)\times\R^d$, the third line in the sense of measure, and the last line for $\rho_T$-a.e. $y\in\R^d$. \\

\noindent \textbf{Remark. }Although seemingly different, the above system is consistent with the Schr\"odinger system \eqref{Dynamic Schrodinger system - general} stated in the general theory. Indeed, upon extending $L_0$ from regime ``$\a$'' to regime ``$\d$'' trivially, i.e., setting all the coefficients of $L_0$ to zero in regime ``$\d$'', and taking into account  the ``one-way'' (from regime ``$\a$'' to regime ``$\d$'') only switching mechanism, we see that the above uSBP Schr\"odinger system is exactly \eqref{Dynamic Schrodinger system - general} where the jump diffusion component of $L$ is simply $L_0$, and the regime-switching rates are $Q_{\a\d}=V,\,Q_{\d\a}=0$.\\

Moreover, we are also ready to apply \cref{Theorem: phi*phi-hat product density} to the uSBP, which implies that the Schr\"odinger bridge $\widehat\bP$ admits the following transition and marginal densities:\\

\noindent \textbf{Unbalanced Schr\"odinger Bridge Transition Density Function:}
\begin{equation*}
    \begin{dcases}
        \widehat p_{11}(t,x,s,y)=\frac{\varphi(s,y,\a)}{\varphi(t,x,\a)}q(t,x,s,y)\\
        \widehat p_{12}(t,x,s,y)=\frac{\varphi(s,y,\d)}{\varphi(t,x,\a)}\int_t^sV(r) q(t,x,r,y)dr
    \end{dcases}
    \;\;\textup{ if }\varphi(t,x,\a)>0,\;0\textup{ otherwise, and}
\end{equation*}
\noindent \textbf{Unbalanced Schr\"odinger Bridge Marginal Density Function:}
\begin{equation}\label{eq:marginal density bP^ usbp}
    \begin{dcases}
        \widehat\bP_s(y,\a)=\varphi(s,y,\a)\widehat\varphi(s,y,\a)=\varphi(s,y,\a)q(0,x_0,s,y)\\
        \widehat\bP_s(y,\d)=\varphi(s,y,\d)\widehat\varphi(s,y,\d)=\varphi(s,y,\d)\int_0^sV(r)q(0,x_0,r,y)dr
    \end{dcases}
\end{equation}
for every $0\leq t<s\leq T$ and $x,y\in\R^d$.\\

We further observe that, since $\varphi(s,y,\d)=\g(y,\d)$ is constant in the temporal variable,
\begin{equation}\label{eq: relation between p^_11 and p^_12}
    \begin{aligned}
        \widehat p_{12}(t,x,s,y)
        &=\int_t^s  \frac{\varphi(r,y,\d)}{\varphi(t,x,\a)}V(r)q(r,x,r,y)dr\\
        &=\int_t^s  \frac{\varphi(r,y,\d)}{\varphi(r,y,\a)}V(r)\frac{\varphi(r,y,\a)}{\varphi(t,x,\a)}q(t,x,r,y)dr\\
        &=\int_t^s  \frac{\varphi(r,y,\d)}{\varphi(r,y,\a)}V(r)\widehat p_{11}(t,x,r,y)dr.
    \end{aligned}
\end{equation}
In view of \cref{cor:p_11 p_12 satisfying D^1}, the above relation between $\widehat p_{\a\a}$ and $\widehat p_{\a\d}$ suggests that the killing rate under $\widehat\bP$ is $\frac{\varphi(t,x,\d)}{\varphi(t,x,\a)}V(t)$. We will confirm in the upcoming discussion that this is indeed the case.
\subsubsection*{The Schr\"odinger Bridge SDE and Generator} As above, suppose that Assumption \ref{Assumptions C-dagger} holds, $\rho_0=\delta_{x_0}$,  and $\rho_T$ satisfies \eqref{eq: rewrite A2 for usbp} and \eqref{eq: rewrite D2 for usbp}. We now transport \cref{Theorem: General P-hat} into the uSBP setting. Again, only the dynamics within the active regime are relevant. In order to have \cref{Theorem: General P-hat}, we need to invoke Assumption \ref{Assumption: general harmonic varphi} which, in addition to the regularity and the PIDE of $\varphi$ discussed above, also requires $\varphi$ to be positive (except possibly at the terminal time). To this end, for the rest of this section, we impose an additional assumption:\\

\begin{enumerate}[label=($\star$)]
    \item \label{asm : star} For every $0\leq t< T$ and $x\in\R^d$, $q(t,x,T,y)>0$ for Lebesgue-a.e. $y\in S_\a$, and $\int_t^T V(r)q(t,x,r,y)dr>0$ for Lebesgue-a.e. $y\in S_\d$.\\
\end{enumerate}

\noindent We recall that $S_\a=\textup{supp}(\bR_T(\cdot,\a))$ and $S_\d=\textup{supp}(\bR_T(\cdot,\d))$. There are typical settings in which the property \ref{asm : star} is valid. For example, given \eqref{eq: formula of density under killing}, \ref{asm : star} is clearly true if  $q_0(t,x,s,y)$ is positive everywhere, where $q_0$ is the transition density function of the jump diffusion $\set{X_t}$ in the SDE \eqref{eqn: uSBP reference SDE Xt only} but without killing; this is equivalent to the transition distributions of $\set{X_t}$ being irreducible, which will be valid in the current setting if all the coefficients $b,\sigma,\gamma$ in \eqref{eqn: uSBP reference SDE Xt only} are time-homogeneous \cite{kunwai2020feller}.

Under Assumption \ref{asm : star}, it is easy to see from \eqref{eq: varphi-usbp} that $\varphi(t,x,\a)>0$ for every $(t,x)\in[0,T)\times\R^d$, which, combined with Assumption \ref{Assumptions C-dagger}, guarantees that Assumption \ref{Assumption: general harmonic varphi} is satisfied. Therefore, applying \cref{Theorem: General P-hat} yields the SDE and generator formulation of the Schr\"odinger bridge corresponding to the uSBP.\\

Let $\widehat\bP$ be the solution to the uSBP \eqref{dfn: uSBP} (and hence $\widehat\bP$ takes the form of \eqref{eq: schrodinger bridge P^ usbp}), and $\varphi$ be as in \eqref{eq: varphi-usbp}. Then, $\widehat\bP$ solves the following SDE:\\

\noindent \textbf{Unbalanced Schr\"odinger Bridge SDE:}
\begin{equation*}
    \left\{
    \begin{aligned}
        &dX_t=b^\varphi(t,X_{t})dt+\sigma(t,X_{t})dB_t + \int_{|z|\leq 1}\gamma(t,X_{t-} ,z)\tilde{N}^\varphi(dt,dz)\\
        &\hspace{3cm} +\int_{|z|>1}\gamma(t,X_{t-} ,z)N(dt,dz),\textup{ for }0<t<\tau:=\inf\set{s\geq 0:\Lambda_s=\d}, \\
    &d\Lambda_t = \int_{[0,V(t)]}\one_{\set{1}}(\Lambda_{t-})N_1(dt,dw),\qquad\textup{ for }0<t<T,
    \end{aligned}
\right.
\end{equation*}
with $(X_0,\Lambda_0)\equiv (x_0,\a)$ and $(X_T,\Lambda_T)$ having distribution $\rho_T$, where
\begin{equation*}
    \begin{aligned}
    b^\varphi(t,x):=b(t,x)&+\sigma\sigma^\top\nabla\log\varphi(t,x,\a) + \\
 & \int_{|z|\leq 1} \frac{\varphi(t,x+\gamma(t,x,z),\a)-\varphi(t,x,\a)}{\varphi(t,x,\a)}\gamma(t,x,z)\nu(dz),
    \end{aligned}
    \end{equation*}
and under $\widehat\bP$,  $\set{B_t}_{t\geq 0}$ is a $d$-dimensional Brownian motion, $N(dt,dz)$ is an $\ell$-dimensional Poisson random measure with intensity measure $\nu^\varphi(t,x;dz):=\frac{\varphi(t,x+\gamma(t,x,z),\a)}{\varphi(t,x,\a)}\nu(dz) $ and $\tilde{N}^\varphi(dt,dz):=N(dt,dz)-\nu^\varphi(t,x;dz)dt$ is its compensated counterpart, and  $N_1(dt,dw)$ is the Poisson random measure on $\R_+$ with intensity measure $\leb^\varphi(t,x;dw)=\frac{\varphi(t,x,\d)}{\varphi(t,x,\a)}\leb(dw)$.\\

Again, by trivially extending the coefficients $b,\sigma,\gamma$ from regime ``$\a$'' to ``$\d$'', we can write the generator associated with $\widehat\bP$, denoted $L_{\widehat\bP}$, as\\

\noindent \textbf{Unbalanced Schr\"odinger Bridge Generator (backward):}
\begin{equation*}
\begin{dcases}
     \begin{aligned}
        &L_{\widehat\bP} f(t,x,\a)\\
        &\quad=b^\varphi(t,x)\cdot \nabla f(t,x,\a) + \sum_{m,n}(\sigma\sigma^\top)_{mn}\frac{\partial^2 f}{\partial x_m \partial x_n}(t,x,\a) \\
        &\qquad+\int_{\R^\ell} \bigg(f(t,x+\gamma(t,x,z),\a)-f(t,x,\a)-\one_{|z|\leq 1}\gamma(t,x,z) \cdot \nabla f(t,x,\a)\bigg)\nu^\varphi(t,x,dz) \\
        &\hspace{7.5cm}+ \frac{\varphi(t,x,\d)}{\varphi(t,x,\a)}V(t)\pran{f(t,x,\d)-f(t,x,\a)},
\end{aligned}\\
L_{\widehat\bP}f(t,x,\d)=0,\qquad\textup{ for $f$ of class }C^{1,2}_c\textup{ and }(t,x)\in(0,T)\times \R^d.
\end{dcases}
\end{equation*}
It is clear from the form of $L_{\widehat\bP}$ and $L^*_{\widehat\bP}$ that the killing rate under $\widehat\bP$ is $\frac{\varphi(t,x,\d)}{\varphi(t,x,\a)}V(t)$, as we have observed in \eqref{eq: relation between p^_11 and p^_12}. 

Further, applying \cref{Theorem: phi*phi-hat product density} to the uSBP, we also get the adjoint of $L_{\widehat\bP}$ as\\

\noindent \textbf{Unbalanced Schr\"odinger Bridge Generator (forward):}
\begin{equation*}
    \begin{dcases}
      \begin{aligned}
    &L_{\widehat\bP}^*f(t,x,\a)\\
    &\quad=-\nabla \cdot (b^\varphi f)(t,x,\a) +\frac{1}{2} \sum_{m,n}\frac{\partial^2[(\sigma\sigma^\top)_{mn}f]}{\partial x_m \partial x_n}(t,x,\a) \\
    &\qquad+\int_{\R^\ell} \bigg[\det \nabla\phi_{t,z}^{-1}(x)f(t,\phi_{t,z}^{-1}(x),\a)\frac{\varphi(t,x,\a)}{\varphi(t,\phi_{t,z}^{-1}(x),\a)}-f(t,x,\a)\frac{\varphi(t,\phi_{t,z}(x),\a)}{\varphi(t,x,\a)} \\
    &\hspace{4cm}+\one_{|z|\leq 1}\nabla\cdot \pran{\gamma(t,x,z)f(t,x,\a)\frac{\varphi(t,\phi_{t,z}(x),\a)}{\varphi(t,x,\a)}  }\bigg]\nu(dz) \\
    &\hspace{8.5cm}-\frac{\varphi(t,x,\d)}{\varphi(t,x,\a)}V(t)f(t,x,\a),
    \end{aligned}\\
    L_{\widehat\bP}^*f(t,x,\d)=\frac{\varphi(t,x,\d)}{\varphi(t,x,\a)}V(t)f(t,x,\a), \qquad\textup{ for $f$ of class }C^{1,2}_c\textup{ and }(t,x)\in(0,T)\times \R^d. 
    \end{dcases}
\end{equation*}

By \eqref{eq:marginal density bP^ usbp} and \cref{q existence and C12 uSBP}, it is possible to verify directly that the marginal density function of the surviving particles satisfies that \[
\frac{\partial}{\partial t}\widehat\bP_t(x,\a)=L_{\widehat\bP}^*\widehat\bP_t(x,\a),\;\textup{ for }(t,x)\in(0,T)\times\R^d,
\]
as stated in \cref{Theorem: phi*phi-hat product density}. On the other hand, since $\varphi(t,x,\d)$ is constant in $t$, a simple calculation leads to 
\[
\frac{\partial}{\partial t}\widehat\bP_t(x,\d)=\varphi(t,x,\d)V(t)q(0,x_0,t,x)=\frac{\varphi(t,x,\d)}{\varphi(t,x,\a)}V(t)\widehat\bP_t(x,\a),
\]
which is exactly $\frac{\partial}{\partial t}\widehat\bP_t(x,\d)=L_{\widehat\bP}^*\widehat\bP_t(x,\d)$, the PIDE the marginal density function of the dead particles are expected to satisfy. The dynamics of $\widehat\bP_t(\cdot,\d)$ once agains confirms that the killing rate under $\widehat\bP$ is $\frac{\varphi(t,x,\a)}{\varphi(t,x,\a)}V(t)$. 
\subsubsection*{Stochastic Control Formulation of the uSBP}
Let $\g$ be defined as in \eqref{eq: def of g in usbp} and $\varphi$ be as in \eqref{eq: varphi-usbp}. To formulate our uSBP in terms of a stochastic control problem (SCP) as we did in \cref{Subsection: stochastic control}, we need to further recover Assumption \ref{asm : log varphi C12b}, which entails extending the regularity and the positivity of $\varphi(t,x,\a)$ to $t\in[0,T]$. To this end, we invoke the condition:\\

\begin{enumerate}[label=($\star\star$)]
    \item \label{asm: double star} The functions $y\in S_\a\mapsto\g(y,\a)$ and $y\in S_\d\mapsto\g(y,2)$ are positive Lebesgue-a.e. Moreover, for every compact set $K\subseteq\R^d$,
\begin{equation*}
     \int_0^T\sup_{x\in K}\abs{\nabla\log\varphi(t,x,\a)}^2dt<\infty\textup{ and }\int_0^T\sup_{x\in K}\abs{\frac{\varphi(t,x,\d)}{\varphi(t,x,\a)}}^2dt<\infty.
    \end{equation*}

\end{enumerate}
\noindent Clearly, the condition \ref{asm: double star} relies on the properties of $q(t,x,T,y)$ and $\rho_T$. \\

With the collective effect of \eqref{eq: rewrite A2 for usbp}, \eqref{eq: rewrite D2 for usbp} and Assumptions \ref{Assumptions C-dagger}, \ref{asm : star} and \ref{asm: double star}, we have validated Assumption \ref{asm : log varphi C12b} in the setting of the uSBP. Therefore, by the results in \cref{Subsection: stochastic control}, the uSBP solution $\widehat\bP$ is identical to $\bP^{(u^*,\theta^*,\xi^*)}$, which is the Girsanov transform from $\bR$ by the measurable functions $u^*:(0,T)\times\R^d\rightarrow\R^d$, $\theta^*:(0,T)\times\R^d\times\R^\ell\rightarrow(-\infty,1)$ and $\xi^*:(0,T)\times\R^d\rightarrow(0,\infty)$. Moreover, the triple $(u^*,\theta^*,\xi^*)$ is the solution to\\

\noindent \textbf{Unbalanced Stochastic Control Problem:}
\begin{equation*}
\begin{aligned}
    (u^*,\theta^*,\xi^*)
    &:= \arg\min_{(u,\theta,\xi) \in \cU(\rho_T)} \bE_{\bP^{(u,\theta,\xi)}}\bigg[\int_0^T\bigg(\frac{1}{2}\abs{u(t,X_t)}^2\\
    &\hspace{2cm}+\int_{\R^\ell}\brac{(1-\theta(t,X_t,z))\log(1-\theta(t,X_t,z))+\theta(t,X_t,z)}\nu(dz)\\
    &\hspace{4cm}+ V(t)\brac{\xi(t,X_t)\log \xi(t,X_t)+1-\xi(t,X_t)}\bigg)dt\bigg]
\end{aligned}    
\end{equation*}
and the solution to this problem is given by, for every $(t,x)\in(0,T)\times\R^d$,
\begin{equation*}
   u^*(t,x):=-\sigma^\top\nabla\log \varphi(t,x,\a),\; \theta^*(t,x,z) := 1-\frac{\varphi(t,x+\gamma(t,x,z),\a)}{\varphi(t,x,\a)},\;\xi^*(t,x) := \frac{\varphi(t,x,\d)}{\varphi(t,x,\a)}. 
\end{equation*}
  
\subsection{Further Questions on the uSBP}
In this example, we discussed one specific uSBP scenario: when the target is on the spatial distributions of the surviving and the killed particles. Since the target distributions are often extracted from the actual observations of the system, a further question that is worth exploring is how to adapt the uSBP formulation to scenarios where the observations provide access to different levels of information about the particles. As mentioned earlier, a few examples involving different observational constraints have been addressed in the literature \cite{eldesoukey2024excursion,chen2022most,eldesoukey2025inferring}, but a general investigation of this question is still missing. We hope to establish a unified framework, based on the regime-switching approach, to systematically study and compare uSBPs in a broad range of target distribution scenarios. The key observation is that by varying $\psi(t,x)$, we can choose which information to ``record'' about the killed particles, and this exact information will be exhibited in the terminal distribution in regime $\d$. In this way, we can treat events that occur in the \textit{interior} of the path as part of the reference \textit{terminal} data, and ``match'' it with our target when solving the uSBP. 

The regime-switching approach also enables us to study uSBP models in which, in addition to active particles being killed, the dead particles are also allowed to revive. This is done simply by changing the regime-switching rate $Q_{21}$ to be non-zero. The general theory of the regime-switching SBP can then be applied to derive properties of the unbalanced Schr\"odinger bridge in this setting. Moreover, it is also possible to distinguish between ``revived'' and ``never-killed'' particles by introducing a third regime and letting $Q_{23}$ be the revival rate (instead of $Q_{21}$). 

In summary, the regime-switching approach may lead to a deeper and more comprehensive theory of the uSBP, and we hope this work helps build its foundations.

\bibliographystyle{plain} 
\bibliography{references}

\end{document}